\documentclass[aihp]{imsart}

\RequirePackage{amsthm,amsmath,amsfonts,amssymb}
\RequirePackage[numbers]{natbib}
\RequirePackage[colorlinks,citecolor=blue,urlcolor=blue]{hyperref}
\RequirePackage{graphicx}

\usepackage{mathabx,mathtools}
\usepackage{tikz}
\usepackage{color, enumerate}
\usepackage{pgfplots}
\usepackage{bm}
\usepackage{stmaryrd}
\usepackage{caption}
\usepackage{subcaption}
\usepackage{epstopdf}
\usepackage{mathrsfs}

\arxiv{2005.00999}

\startlocaldefs
\numberwithin{equation}{section}
\theoremstyle{plain}
\newtheorem{theorem}{Theorem}[section]
\newtheorem*{theorem*}{Theorem}
\newtheorem{lemma}[theorem]{Lemma}
\newtheorem{assumption}[theorem]{Assumption}
\newtheorem*{lemma*}{Lemma}

\newtheorem*{corollary*}{Corollary}
\newtheorem{proposition}[theorem]{Proposition}
\newtheorem*{proposition*}{Proposition}
\newtheorem{claim}[theorem]{Claim}
\newtheorem{definition}[theorem]{Definition}
\newtheorem*{definition*}{Definition}
\theoremstyle{remark}

\newtheorem*{example*}{Example}
\newtheorem{remark}[theorem]{Remark}

\newtheorem*{remark*}{Remark}
\newtheorem*{remarks*}{Remarks}

\DeclareMathOperator{\OO}{O}
\DeclareMathOperator{\oo}{o}

\DeclareMathOperator{\re}{Re}
\DeclareMathOperator{\im}{Im}

\DeclareMathOperator{\supp}{supp}

\DeclareMathOperator{\bv}{\mathbf{v}}
\DeclareMathOperator{\bu}{\mathbf{u}}
\DeclareMathOperator{\bw}{\mathbf{w}}
\DeclareMathOperator{\sI}{\mathcal{I}}

\DeclareMathOperator{\sW}{\mathcal{W}}

\newcommand{\ii}{\mathrm{i}} 
\newcommand{\dd}{\mathrm{d}}

\renewcommand{\cal}{\mathcal}

\newcommand{\be}{\begin{equation}}
\newcommand{\ee}{\end{equation}}
\newcommand{\beq}{\begin{equation}}
\newcommand{\bEq}{\end{equation}}

\renewcommand{\le}{\leq}
\renewcommand{\ge}{\geq}

\newcommand{\E}{\mathbb{E}}
\newcommand{\R}{\mathbb{R}}
\newcommand{\C}{\mathbb{C}}
\newcommand{\N}{\mathbb{N}}

\renewcommand{\Im}{{\rm{Im}}}
\renewcommand{\Re}{{\rm{Re}}}
\newcommand{\e}{{\varepsilon}}

\newcommand{\al}{\alpha}

\newcommand{\bt} {{\bf t }}
\newcommand{\fa}{{\mathfrak a}}
\newcommand{\fb}{{\mathfrak b}}

\newcommand{\si}{\sigma}
\newcommand{\wh}{\widehat}
\newcommand{\wt}{\widetilde}

\endlocaldefs

\begin{document}

\begin{frontmatter}
\title{Linear spectral statistics of eigenvectors of anisotropic sample covariance matrices}
\runtitle{Linear spectral statistics of eigenvectors}

\begin{aug}
\author[A]{\fnms{Fan} \snm{Yang}\ead[label=e1]{fyangmath@mail.tsinghua.edu.cn}\orcid{0000-0001-6972-0784}}

\address[A]{Yau Mathematical Sciences Center, Tsinghua University, and Beijing Institute of Mathematical Sciences and Applications
\printead[presep={,\ }]{e1}}

\end{aug}

\begin{abstract}
Consider sample covariance matrices of the form $Q:=\Sigma^{1/2} X X^\top \Sigma^{1/2}$, where $X=(x_{ij})$ is an $n\times N$ random matrix whose entries are independent random variables with mean zero and variance $N^{-1}$, and $\Sigma$ is a deterministic positive-definite covariance matrix. We study the limiting behavior of the eigenvectors of $Q$ through the so-called eigenvector empirical spectral distribution $F_{\mathbf v}$, which is an alternative form of empirical spectral distribution with weights given by $|\mathbf v^\top \xi_k|^2$, where $\mathbf v$ is a deterministic unit vector and $\xi_k$ are the eigenvectors of $Q$. We prove a functional central limit theorem for the linear spectral statistics of $F_{\mathbf v}$, indexed by functions with H{\"o}lder continuous derivatives. We show that the linear spectral statistics converge to some Gaussian processes both on global scales of order 1 and on local scales that are much smaller than 1 but much larger than the typical eigenvalue spacing $N^{-1}$. Moreover, we give explicit expressions for the covariance functions of the Gaussian processes, where the exact dependence on $\Sigma$ and $\mathbf v$ is identified for the first time in the literature. 
\end{abstract}




\end{frontmatter}

 \section{Introduction}

Consider a centered random vector $\mathbf y\in \mathbb R^n$ with population covariance $\Sigma=\mathbb E \mathbf y \mathbf y^\top$. Given $N$ i.i.d. samples $(\mathbf y_1, \ldots, \mathbf y_N)$ of $\mathbf y$, the simplest estimator for $\Sigma$ is the sample covariance matrix \smash{$Q := N^{-1}\sum_i \mathbf y_i \mathbf y_i^\top$}. Large dimensional sample covariance matrices have been a central object of study in high-dimensional statistics. In many modern applications, such as statistics \cite{DT2011,IJ,IJ2,IJ2008}, economics \cite{Economics} and population genetics \cite{Genetics}, the advance of technology has led to high dimensional data where $n$ is comparable to or even larger than $N$. In this setting, the law of large numbers does not hold and $\Sigma$ cannot be approximated by $Q$ directly. However, with more advanced tools in random matrix theory, it is still possible to infer some properties of $\Sigma$ from the eigenvalue and eigenvector statistics of $Q$. 

In this paper, we consider sample covariance matrices of the form $Q_1:=\Sigma^{1/2} X X^\top \Sigma^{1/2}$, where $X=(x_{ij})$ is an $n\times N$ real data matrix whose entries are independent 
random variables satisfying
\begin{align}
 \mathbb{E} x_{ij} =0, \ \ \mathbb{E} | x_{ij} |^2  = N^{-1}, \ \ 1\le i \le n, \ 1\le j \le N, \label{entry_assm}
\end{align}
and the population covariance matrix $\Sigma$ is an $n\times n$ deterministic positive-definite matrix. 
Define the aspect ratio $d_N:={n}/{N}.$ We are interested in the high dimensional setting with $d_N \to d \in (0,\infty)$ as $N\to \infty$. 
We will also use the $N \times N$ matrix $Q_2:=X^\top \Sigma X$, which share the same nonzero eigenvalues with $Q_1$.

In the study of eigenvalue statistics of large dimensional sample covariance matrices, one of the most fundamental subjects of study is the asymptotic behavior of the empirical spectral distribution (ESD). When $\Sigma=I$, i.e., the population covariance is trivial, it is well-known that the ESD of $Q_1$ converges weakly to the famous Mar{\v c}enko-Pastur (MP) law $F_{MP}$ \cite{MP}. The convergence rate was first established in \cite{bai1993_2}, and later improved in \cite{GT2004} to $\OO(N^{-1/2})$ in probability under the finite 8th moment condition. In \cite{PY}, the authors proved an almost optimal bound $\OO(N^{-1+\epsilon})$ with high probability for any small constant $\epsilon>0$ under the sub-exponential decay assumption. For the limiting spectral statistics, a functional CLT was proved in \cite{baiCLT} for the ESD of $Q_1$. Roughly speaking, it was proved that given an analytic function $f(x)$, the random variable
\be\nonumber \sum_{i=1}^n f(\lambda_i) - n\int f(x)\dd F_{MP}(x)\ee
converges in distribution to a centered Gaussian random variable, where $\lambda_i$ are the eigenvalues of $Q_1$. In fact, \cite{baiCLT} proved a more general multivariate statement that for any analytic functions $f_1(x),\ldots, f_k(x)$, the random vector
\be\nonumber \left(\sum_{i=1}^n f_s(\lambda_i) - n\int f_s(x)\dd F_{MP}(x)\right)_{1\le s \le k}\ee
converges in distribution to a centered Gaussian vector. Later, this result was extended to include more general functions with continuous third order derivatives \cite{najim2016}. This kind of functional CLT is usually referred to as ``linear eigenvalue statistics". Recently, in \cite{Mesoscopicsample} the authors extended it to mesoscopic eigenvalue statistics, that is, for any fixed $E>0$ and scale parameter $n^{-1}\ll \eta \ll 1$, the random vector
 \be\nonumber \left(\sum_{i=1}^n f_s\left(\frac{\lambda_i-E}{\eta}\right) - n\int f_s\left(\frac{x-E}{\eta}\right)\dd F_{MP}(x)\right)_{1\le s \le k}\ee
converges in distribution to a centered Gaussian vector. We shall call such a result the ``local linear eigenvalue statistics". 

The concept of ESD can be also extended to encode the information of sample eigenvectors. Following \cite{Bai2007,JS1989,JS1990,XZ2016,XYZ2013}, we define the following concept of {\it{eigenvector empirical spectral distribution}} (VESD). Suppose
\begin{equation}\label{SVD_X}
\Sigma^{1/2}X = \sum\limits_{ k =1 }^{N\wedge n} {\sqrt {\lambda_k} \xi_k } \zeta_k^\top
\end{equation}
is a singular value decomposition of $\Sigma^{1/2}X$, where $\lambda_1\ge \lambda_2 \ge \ldots \ge \lambda_{N\wedge n} \ge 0 = \lambda_{N\wedge n+1} = \ldots = \lambda_{N\vee n}$ are the eigenvalues of $Q_1$ and $Q_2$, $\{\xi_{k}\}_{k=1}^{n}$ are the left-singular vectors, and $\{\zeta_k\}_{k=1}^{N}$ are the right-singular vectors. Then, for any deterministic vector $\mathbf v\in\mathbb R^n$, we define the VESD of $Q_{1}$ as
\begin{equation}\label{defn_VESD}
F_{\mathbf v}(x) := \sum_{k=1}^n |\langle \xi_k,{\mathbf v}\rangle|^2 \mathbf{1}_{\{\lambda_k \leq x\}}.
\end{equation}
In this paper, we use the notation $\langle \bu,\bv\rangle:=\bu^*\bv$ to denote the inner product of two (possibly complex) vectors, where $\bu^*$ denotes the conjugate transpose of $\bu$. In the null case with $\Sigma=I_{n}$, it was proved in \cite{Bai2007,isotropic} that $F_{\mathbf v_n}$ converges weakly to the MP law for any sequence of unit vectors $\mathbf v_n $. In \cite{XYZ2013}, the convergence rate was shown to be $\OO(N^{-1/4+\epsilon})$ almost surely, which was later improved to $\OO(N^{-1/2+\epsilon})$ in \cite{XYY_VESD}.
In fact, \cite{XYY_VESD} considered a more general setting where the population covariance matrix $\Sigma$ is not necessarily proportional to identity. In this case, it was found that $F_{\mathbf v_n}(x)$ does not converge to the MP law anymore. Instead, it converges to a distribution depending on $\bv_n$,
$ F_{1c,\bv_n}(x):=\langle \bv_n,\mathbf F_{1c}(x) \bv_n\rangle,$ 
where $\mathbf F_{1c}(x)$ is a matrix-valued function determined by $\Sigma$. We will refer to the class of distributions $F_{1c,\bv}$ as \emph{anisotropic MP laws}.

As for the ESD theory, the next piece of the VESD theory is the functional CLT for $F_{\mathbf v}$. More precisely, we are interested in the CLT for random vectors of the form
\be\label{linearvector} \left(\sqrt{n}\sum_{k=1}^n |\langle \xi_k,\mathbf v\rangle|^2 f_s(\lambda_k)- \sqrt{n}\int f_s(x)\dd F_{1c,\bv}(x)\right)_{1\le s \le k}.\ee
In this paper, we refer to this kind of result as the ``linear eigenvector statistics". In the null case with $\Sigma=I_{n}$, the linear eigenvector statistics were studied in \cite{JS1990} when $\bv$ takes the form \smash{$(\pm n^{-1/2},\pm n^{-1/2},\ldots,\pm n^{-1/2})$}. Later, this result was extended to the case with arbitrary unit vector $\bv$ and general analytic functions $f_s$ in \cite{Bai2007}. In \cite{Xia2015}, the class of functions is extended to include all functions with continuous third order derivatives. In fact, \cite{Bai2007} considered slightly more general $\Sigma$, requiring that the sequence of vectors $\bv_n$ satisfies the condition
\be\label{isointro}\sup_{z\in \cal D}\sqrt{n}\left|\bv_n^\top \frac{1}{1+m_{2c}(z)\Sigma}\bv_n - \int \frac{1}{1+m_{2c}(z)t}\pi_{\Sigma}(\dd t)\right|\to 0,\ee
where $\pi_{\Sigma}$ is the ESD of $\Sigma$, $\cal D$ is an open neighborhood of the support of the MP law in the complex plane, and $m_{2c}(z)$ is the Stieltjes transform of the MP law (cf. \eqref{deformed_MP21}). The condition \eqref{isointro} is essentially an isotropic condition, under which the VESD $F_{\bv_n}(x)$ still converges to the MP law $F_{MP}$, and the information of the vectors $\bv_n$ is missed in the asymptotic limit. In general, when \eqref{isointro} does not hold, it is still unknown whether the functional CLT still holds and, if the functional CLT indeed holds, how the mean and covariance of the limiting Gaussian vector depend on the covariance matrix $\Sigma$ and the vectors $\bv_n$.

The main goal of this paper is to solve this problem. More precisely, we consider sample covariance matrices with completely general population covariance matrices $\Sigma$ (up to some technical regularity assumptions). We prove that for any sequences of unit vectors \smash{$\bv_s\equiv \bv_s^{(n)}$}, $1\le s \le k$, the random vector 
\be\label{linearvector2} 
\left(\sqrt{n}\sum_{k=1}^n |\langle \xi_k,\mathbf v_s\rangle|^2 f_s(\lambda_k)- \sqrt{n}\int f_s(x)\dd F_{1c,\bv_s}(x)\right)_{1\le s \le k}
\ee
converges to a centered Gaussian vector. Moreover, we obtain an explicit expression for the  covariance matrix of the Gaussian vector, which allows us to characterize precisely how the anisotropy of the covariance matrix $\Sigma$ affects the linear eigenvector statistics. 
We also extend the result to ``local linear eigenvector statistics". That is, for any fixed $E>0$ and scale parameter $n^{-1}\ll \eta \ll 1$, we prove that the random vector
\be\label{linearvector2local} 
\left(\sqrt{n\eta}\sum_{k=1}^n |\langle \xi_k,\mathbf v_s\rangle|^2 \frac1\eta f_s\left(\frac{\lambda_k-E}{\eta}\right)- \sqrt{n\eta}\int \frac1\eta f_s\left(\frac{x-E}{\eta}\right)\dd F_{1c,\bv_s}(x)\right)_{1\le s \le k}
\ee
also converges in distribution to a centered Gaussian vector. In addition, we find that in global linear eigenvector statistics, the covariance matrix of the Gaussian vector depends on the fourth cumulants of the $X$ entries, while in local linear eigenvector statistics it does not, which suggests that the local eigenvector statistics is ``more universal" than the global eigenvector statistics. This kind of phenomenon is actually pretty common in random matrix theory and has been identified in many previous works on linear spectral statistics of random matrices; see e.g., \cite{AZ_PTRF,BWZ_Bernoulli,BY_Bernoulli,Bai2007,baiCLT,MK_meso,Mesoscopic,Mesoscopicsample,LX_Bernoulli,LS_AAP,Meso_Wigner,Cumulant1, Cumulant2, SW_Advance,Shc11,Xia2015, Zheng_AIHP}.

For any $z\in \C\setminus \R$, we define the resolvent (or Green's function) of the sample covariance matrix $Q_1$ as $R(z):=(Q_1-z)^{-1}$. As a byproduct of the proof, we also obtain a CLT for $R_{\mathbf u \mathbf v}(z):=\langle \bu,R(z)\bv\rangle$, where $\bu,\bv\in \R^n$ are arbitrary deterministic unit vectors. Moreover, we prove the CLT for both the case where $\eta:=\im z$ is of global scale $\eta\sim1$ and the case where $\eta$ is of local scale $n^{-1}\ll \eta \ll 1$. In this paper, we shall call $R_{\mathbf u \mathbf v}$ a \emph{generalized resolvent entry}. Besides the application in linear eigenvector statistics, it is known that the CLT for generalized resolvent entries is also crucial in studying the limiting distributions of outlier eigenvalues and eigenvectors of deformed Wigner matrices \cite{KY,KY_AOP} and spiked sample covariance matrices with trivial population covariance $\Sigma=I$ \cite{Dingevector1, Dingevector2}. Hence, we expect our CLT to be of independent interest in studying the asymptotic distribution of outlier eigenvalues and eigenvectors for spiked sample covariance matrices with general population covariance, which we leave to future study.

The VESD was originally introduced in \cite{JS1989,JS1990} to study the asymptotic property of sample eigenvectors. The study of eigenvectors of large random matrices is generally harder and much less developed compared with the study of eigenvalues. On the other hand, eigenvectors play an important role in principal component analysis (PCA), which is now favorably recognized as a powerful technique for dimensionality reduction. The early work on sample eigenvectors goes back to Anderson \cite{Anderson}, where it was proved that the eigenvectors of a Wishart matrix are asymptotically normal as $N\to\infty$ if $n$ is fixed. In the high dimensional setting, Johnstone \cite{IJ2} proposed the famous spiked model, which is now a standard model for the study of PCA of large random matrices. Later, Paul \cite{Paul2007} studied the directions of sample eigenvectors of the spiked model. The reader can also refer to \cite{Cai2013,Ma2013} and references therein for more recent literature on sparse PCA and spiked covariance matrices. 

PCA focuses on the first couple of eigenvectors corresponding to the largest few eigenvalues. On the other hand, studying the asymptotic properties of all eigenvectors at the same time (or, more precisely, the eigenmatrix) is much harder. In fact, even formulating the terminology ``asymptotic property of the eigenmatrix" is far from trivial, since the sample dimension $n$ is increasing. For this purpose, the VESD serves as a manageable tool to discuss about the asymptotic behavior of all eigenvectors as a whole. In \cite{Bai2007,XZ2016,XYZ2013}, when $\Sigma=I_n$, the VESD was used to characterize the asymptotical Haar property of the eigenmatrix, that is, the eigenmatrix is expected to be asymptotically uniformly distributed over the orthogonal group. When $\Sigma$ is not isotropic, the eigenmatrix is not asymptotically Haar distributed anymore, and our results in this paper describe precisely how the VESD behaves along every direction. In addition, with the extension to general $\Sigma$, our results provide more flexibility in applying VESD to the study of sample covariance matrices.

Before concluding the introduction, we summarize the main contributions of our work. 
\begin{itemize}
\item We extend the function CLT for VESD in \cite{Bai2007,Xia2015} to anisotropic sample covariance matrices with general population covariance $\Sigma$. This result is presented as Theorem \ref{main_thm}, which is stronger than the ones in \cite{Bai2007,Xia2015} in several senses (see Remark \ref{rem stronger} below).

\item Besides the global linear eigenvector statistics, we also study the local linear eigenvector statistics, and prove the function CLT for VESD on all scales $\eta$ such that $n^{-1}\ll \eta \ll 1$; see Theorem \ref{main_thm2}.

\item We prove a CLT of generalized resolvent entries for both the global scale $\eta\sim 1$ and the mescoscopic scale $n^{-1}\ll \eta \ll 1$; see Theorems \ref{main_thm3} and \ref{main_thm4}. 



\end{itemize}

This paper is organized as follows. In Section \ref{main_result}, we state the main results of this paper: Theorems \ref{main_thm} and \ref{main_thm2}, which give the functional CLT of the VESD, and Theorems \ref{main_thm3} and \ref{main_thm4}, which give the CLT of the generalized resolvent entries. For these results, we assume that the entries of $X$ have finite $(8+\e)$-th moment. In Section \ref{sec_maintools}, we collect some basic tools that will used in the proof, and in Section \ref{sec_overview}, we give a brief overview of the proof strategy. Then, in Section \ref{sec resolvent}, we prove Theorems \ref{main_thm3} and \ref{main_thm4} under a stronger moment assumption that the entries of $X$ have finite moments up to any order. Based on the results in Section \ref{sec resolvent}, we prove Theorems \ref{main_thm} and \ref{main_thm2} in Section \ref{sec generalf} under the stronger moment assumption. Finally in Section \ref{sec relax}, using a Green's function comparison argument, we relax the moment assumption to the finite $(8+\e)$-th moment assumption in the main theorems.

\vspace{5pt}

\noindent{\bf Conventions.} 
The fundamental large parameter is $N$ and we assume that $n$ is comparable to and depends on $N$. We use $C$ to denote a generic large positive constant, whose value may change from one line to the next. Similarly, we use $\epsilon$, $\tau$, $\delta$ and $c$ to denote generic small positive constants. If a constant depends on a quantity $a$, we use $C(a)$ or $C_a$ to indicate this dependence. For two quantities $a_N$ and $b_N$, the notation $a_N = \OO(b_N)$ means that $|a_N| \le C|b_N|$ for some constant $C>0$, and $a_N=\oo(b_N)$ or $|a_N|\ll |b_N|$ means that $|a_N| / |b_N|\to 0$ as $N\to \infty$. We also use the notations $a_N \lesssim b_N$ if $a_N = \OO(b_N)$, and $a_N \sim b_N$ if $a_N = \OO(b_N)$ and $b_N = \OO(a_N)$. For a matrix $A$, we use $\|A\|\equiv \|A\|_{l^2 \to l^2}$ to denote its operator norm; 
for a vector $\mathbf v=(v_i)_{i=1}^n$, $\|\mathbf v\|\equiv \|\mathbf v\|_2$ stands for the Euclidean norm. Given a matrix $A$ and $a\in \R$, we write $A=\OO(a)$ if $\|A\|=\OO(a)$. In this paper, we often write an identity matrix as $I$ or $1$ without specifying its dimension.

\section{Definitions and Main Result}\label{main_result}

\subsection{The model}

We consider a class of real sample covariance matrices of the form $\mathcal Q_1:=\Sigma^{1/2}XX^\top \Sigma^{1/2}$, where $\Sigma$ is a deterministic positive semi-definite matrix. We assume that $X=(x_{ij})$ is an $n\times N$ random matrix with independent entries $x_{ij}$, $1 \leq i \leq n$, $1 \leq j \leq N$, satisfying
\begin{equation}\label{assm1}
\mathbb{E} x_{ij} =0, \ \quad \ \mathbb{E} \vert x_{ij} \vert^2  =N^{-1}.
\end{equation}
We will also use the $N \times N$ matrix $\mathcal Q_2:= X^\top \Sigma X  $. 
We assume that the aspect ratio $d_N:= n/N$ satisfies  
\begin{equation}
 \tau \le d_N \le \tau^{-1}, \label{assm2}
 \end{equation}
for some constant $0<\tau <1$. 
For simplicity of notations, we will often abbreviate $d_N$ as $d$ in this paper. We denote the eigenvalues of $\mathcal Q_1$ and $\mathcal Q_2$ in descending order as $\lambda_1(\mathcal Q_1)\geq \ldots \geq \lambda_{n}(\mathcal Q_1)$ and $\lambda_1(\mathcal Q_2) \geq \ldots \geq \lambda_N(\mathcal Q_2)$. Since $\mathcal Q_1$ and $\mathcal Q_2$ share the same nonzero eigenvalues, for simplicity we will write $\lambda_j$, $1\le j \le N\vee n$, to denote the $j$-th eigenvalue while keeping in mind that $\lambda_j=0$ for $j>N\wedge n$. 
We assume that $\Sigma^{1/2}$ has eigendecomposition
\be\label{eigen}\Sigma= O^{\top} \Lambda O,  \quad \Lambda=\text{diag}(\si_1, \ldots, \si_n) ,
\ee
where $\si_1 \ge \si_2 \ge \ldots \ge \si_n \ge 0 $ are the eigenvalues of $\Sigma$. We denote the empirical spectral density of $\Sigma$ as
\begin{equation}\label{sigma_ESD}
\pi_\Sigma  \equiv \pi_\Sigma  ^{(n)} := \frac{1}{n} \sum_{i=1}^n \delta_{\si_i} .
\end{equation}
We assume that there exists a small constant $0<\tau<1$ such that for all $N$ large enough,
\begin{equation}\label{assm3}
 \si_1 \le \tau^{-1}, \quad  \pi_\Sigma  ^{(n)}([0,\tau]) \le 1 - \tau .
\end{equation}
The first condition means that the operator norms of $\Sigma$ is bounded by $\tau^{-1}$, and the second condition means that the spectrums of $\Sigma$ does not concentrate at zero.

\subsection{Resolvents and limiting law}

In this paper, we will study the eigenvalue and eigenvector statistics of $\mathcal Q_{1}$ and $\mathcal Q_2$ through their {\it{resolvents}} (or  {\it{Green's functions}}). 
In fact, it is equivalent to study the matrices 
\be\label{Qtilde}
\wt{\mathcal Q}_1(X):=\Lambda^{1/2} O XX^{\top}O^{\top}\Lambda^{1/2}, \quad  \wt{\mathcal Q}_2(X) \equiv \cal Q_2(X)= X^{\top} \Sigma X  .
\ee
In this paper, we shall denote the upper half complex plane and the right half real line by 
\be\nonumber\mathbb C_+:=\{z\in \mathbb C: \im z>0\}, \quad \mathbb R_+:=(0,\infty).\ee 

\begin{definition}[Resolvents]\label{resol_not}
For $z = E+ \ii \eta \in \mathbb C_+,$ we define the resolvents for $\wt {\mathcal Q}_{1,2}$ as
\begin{equation}\label{def_green}
  \mathcal G_1(X,z):=\left( \wt{\mathcal Q}_1(X) -z\right)^{-1} , \ \ \ \mathcal G_2 (X,z):=\left(\wt{\mathcal Q}_2(X)-z\right)^{-1} .
\end{equation}
 We denote the empirical spectral density $\rho^{(n)}$ of $\wt {\mathcal Q}_{1}$ and its Stieltjes transform as
\be\label{defn_m}
\rho\equiv \rho^{(n)} := \frac{1}{n} \sum_{i=1}^n \delta_{\lambda_i( \wt{\mathcal Q}_1)},\quad m(z)\equiv m^{(n)}(z):=\int \frac{1}{x-z}\rho^{(n)}(\dd x)=\frac{1}{n} \mathrm{Tr} \, \mathcal G_1(z).
\ee
Note $\rho^{(n)}$ and $m^{(n)}$ are also the empirical spectral density and its Stieltjes transform for ${\cal Q}_1$. We define the following two random quantities:
\be\nonumber m_1(z)\equiv m_1^{(n)}(z):= \frac{1}{N}\sum_{i=1}^n\sigma_i (\mathcal G_1(z) )_{ii},\quad m_2(z)\equiv m_2^{(N)}(z):=\frac{1}{N}\sum_{\mu=1}^N  (\mathcal G_2(z) )_{\mu\mu}. \ee

\end{definition}

If $d_N \to d \in (0,\infty)$ and $\pi_\Sigma$ converges weakly to some distribution $\pi$ as $N\to \infty$, then it was shown in \cite{MP} that the ESD of $Q_{2}$ converges in probability to some deterministic distribution, which is called the (deformed) Mar{\v c}enko-Pastur (MP) law. For any $N\in \N$, we describe the deformed MP law \smash{$F_{2c}^{(N)}$} through its Stieltjes transform
\begin{equation}\nonumber
m_{2c}(z)\equiv m^{(N)}_{2c}(z):=\int_{\mathbb R} \frac{\dd F^{(N)}_{2c}(x)}{x-z}, \quad z = E+ \ii\eta \in \mathbb C_+.
\end{equation}
We define $m_{2c}$ as the unique solution to the self-consistent equation
\begin{equation}\label{deformed_MP21}
\frac{1}{m_{2c}(z)} = - z + {d_N}\int\frac{t}{1+m_{2c}(z) t} \pi_{\Sigma}(\dd t), 
\end{equation}
subject to the conditions that $\Im \, m_{2c}(z) > 0$ and $\Im (zm_{2c}(z)) > 0$ for $z\in \mathbb C_+$. It is well known that the functional equation (\ref{deformed_MP21}) has a unique solution that is uniformly bounded on $\mathbb C_+$ under the assumption (\ref{assm3}) \cite{MP}. Letting $\eta \downarrow 0$, we can recover the asymptotic eigenvalue density $\rho_{2c}$ with the inverse formula
\begin{equation}\label{ST_inverse}
\rho_{2c}(E) = {\pi}^{-1}\lim_{\eta\downarrow 0} \Im\, m_{2c}(E+\ii\eta).
\end{equation}
Then, from $\rho_{2c}$, we can recover the ESD $F_{2c}\equiv F_{2c}^{(N)}$. Since $Q_1$ share the same nonzero eigenvalues with $Q_2$ and has $n-N$ more (or $N-n$ less) zero eigenvalues, we can then obtain the asymptotic ESD for $Q_1$:
\begin{equation}\nonumber
F_{1c}\equiv F^{(n)}_{1c}=d_N^{-1} {F}^{(N)}_{2c}+(1-{d_N^{-1}})\mathbf{1}_{[0,\infty)}.
\end{equation}
In \cite{XYY_VESD}, it was shown that the VESD $F_{\bv}$ of $\cal Q_1$ converges to the anisotropic MP law $F_{1c,\mathbf v}\equiv F^{(n)}_{1c,\mathbf v}$, whose density $\rho_{1c,\mathbf v}$ is given by  
\be\label{anisoMP} \rho_{1c,\mathbf v}(E):= \bv^T \frac{\rho_{2c}(E)\Sigma}{E\left( 1 + 2\re m_{2c}(E) \Sigma + |m_{2c}(E)|^2 \Sigma^2\right)}\bv.\ee
For the rest of this paper, we will often omit the super-indices $N$ and $n$ from our notations. The properties of $m_{2c}$ and $\rho_{2c}$ have been studied extensively; see e.g., \cite{Bai1998,Bai2006,BPZ,HHN,Anisotropic,Silverstein1995,SC}. The following Lemma \ref{Structure_lem} describes some basic properties of $\rho_{2c}$. For its proof, one can refer to \cite[Appendix A]{Anisotropic}.  
\begin{lemma}[Support of the deformed MP law]\label{Structure_lem}
The density $\rho_{2c}$ is a disjoint union of connected components:
\begin{equation}\label{support_rho1c}
{\rm{supp}} \, \rho_{2c} \cap (0,\infty) = \bigcup_{ k=1}^L [a_{2k}, a_{2k-1}] \cap (0,\infty),
\end{equation}
where $L\in \mathbb N$ depends only on $\pi_\Sigma$. Moreover, \smash{$N\int_{a_{2k}}^{a_{2k-1}} \rho_{2c}(x)dx$} is an integer for any $k=1,\ldots, L$, which gives the classical number of eigenvalues in the bulk component $[a_{2k},a_{2k-1}]$. Finally, we have that $a_1 \le C$ for some constant $C>0$ and $m_{2c}(a_1)\equiv m_{2c}(a_1 + \ii 0_+) \in (-\sigma_1^{-1}, 0)$.
\end{lemma}

We shall call $a_k$ the edges of $\rho_{2c}$. Moreover, following the standard notation in random matrix literature, we shall denote the rightmost and leftmost edges as $\lambda_+:=a_1$ and $\lambda_-:=a_{2L}$, respectively. 
To establish our main result, we need to make some extra assumptions on $\Sigma$, which takes the form of the following regularity conditions.

\begin{definition}[Regularity]\label{def_regular}
(i) Fix a (small) constant $\tau>0$. We say an edge $a_k$, $1\le k \le 2L$, is $\tau$-regular if
\begin{equation}
a_k\ge \tau, \quad \min_{l:l\ne k} |a_k - a_l| \ge \tau, \quad \min_{i}|1+m_{2c}(a_k)\sigma_i| \ge \tau , \label{regular1}
\end{equation}
where $m_{2c}(a_k)\equiv m_{2c}(a_k + \ii 0_+)$.

\vspace{5pt}
\noindent (ii) We say that the bulk component $(a_{2k}, a_{2k-1})$ is regular if for any fixed $\tau'>0$, there exists a constant $c\equiv c_{\tau'}>0$ such that the density of $\rho_{2c}$ in $[a_{2k}+\tau', a_{2k-1}-\tau']$ is bounded from below by $c$.
\end{definition}

\begin{remark} The edge regularity conditions (i) has previously appeared (in slightly different forms) in several works on sample covariance matrices \cite{BPZ1, Karoui,HHN, Anisotropic,LS,Regularity4}. The condition (\ref{regular1}) ensures a regular square-root behavior of $\rho_{2c}$ near $a_k$. The bulk regularity condition (ii) was introduced in \cite{Anisotropic}, and it imposes a lower bound on the asymptotic density of eigenvalues away from the edges. These conditions are satisfied by quite general classes of $\Sigma$; see e.g., \cite[Examples 2.8 and 2.9]{Anisotropic}.
\end{remark}

\subsection{Main results}
For any fixed $a,b>0$, we define the class of functions $\cal C^{1,a,b}(\R_+)$ as 
\begin{align*}
\cal C^{1,a,b}(\R_+):=&\left\{f\in \cal C_c^1(\R_+): f' \text{ is $a$-H\"older continuous uniformly in $x$, } |f(x)|+|f'(x)|\lesssim (1+|x|)^{-(1+b)}\right\}. 
\end{align*}
Similar class has been used in \cite{Mesoscopic} for establishing the mesoscopic linear eigenvalue statistics. For $N^{-1+\tau}\le \eta \le 1$, $E\in \R_+$, $f\in \cal C^{1,a,b}(\R_+)$ and any deterministic vector $\bv\in \R^n$, we define
\be\label{Zvf}
\begin{split}
Z_{\eta, E}(\bv,f)&:=  \sqrt{ N /\eta}\int f\left(\eta^{-1} (x - E) \right)  \dd\left( F_{\bv}(x)- F_{1c,\bv}(x)\right) \\
&= \sqrt{   N \eta}  \left(\left\langle \bv, \eta^{-1}f\left( \eta^{-1}({\mathcal Q}_1 - E)  \right) \bv\right\rangle -\int_{\lambda_-}^{\lambda_+}  \eta^{-1}f\left(\eta^{-1} (x - E) \right) \dd F_{1c,\bv}(x)\right).
\end{split}
\ee
 Before stating the main results on the weak convergence of the process $Z_{\eta, E}(\bv,f)$, we first give the main assumptions. 

\begin{assumption}\label{main_assm}
Fix a small constant $\tau>0$.
\begin{itemize}
\item[(i)] $X=(x_{ij})$ is an $n\times N$ real matrix whose entries are independent random variables satisfying \eqref{assm1}.

\item[(ii)] $\tau \le d_N \le \tau^{-1}$ and $|d_N - 1|\ge \tau$.

\item[(iii)] $\Sigma$ is a deterministic positive semi-definite matrix satisfying (\ref{assm3}). Moreover, all the edges of $\rho_{2c}$ are $\tau$-regular, and all the bulk components of $\rho_{2c}$ are regular in the sense of Definition \ref{def_regular}.
\end{itemize}
\end{assumption}
We also need to introduce several notations. First, we denote
\be\label{kappa4}\kappa_4(i,j):=\mathbb E|\sqrt{N}x_{ij}|^4 - 3,\ee
which is the fourth cumulant of the entry $\sqrt{N}x_{ij}$. Then, we define two functions $\al,\beta :\R^{2+2n}\to \R$ as 
\be
\begin{split}\label{defal}
&\al(x_1, x_2, \bv_1,\bv_2)\equiv \al^{(N)} (x_1, x_2, \bv_1,\bv_2)\\
&:=  \sum_{ i =1}^{n}  \frac{\sum_{j=1}^{N}\kappa_4(i,j)}{N} \im\left[ \frac{m_{2c}(x_1)}{x_1} \left(\frac{\Sigma^{1/2}}{1+m_{2c}(x_1)\Sigma}\bv_1\right)^2_i\right] \im\left[ \frac{m_{2c}(x_2)}{x_2} \left(\frac{\Sigma^{1/2}}{1+m_{2c}(x_2)\Sigma}\bv_2\right)^2_i \right] , 
\end{split}
\ee
and 
\be
\begin{split}\label{defbeta}
&\beta(x_1, x_2, \bv_1,\bv_2)\equiv \beta^{(N)}(x_1, x_2, \bv_1,\bv_2)\\
&:=\re\left[\frac{m_{2c}(x_1)-\overline m_{2c}(x_2)}{x_1x_2 }\left(\bv_1^\top \frac{\Sigma}{(1+m_{2c}(x_1)\Sigma)(1+\overline m_{2c}(x_2)\Sigma)}\bv_2\right)^2  \right] \\
&\ -\re\left[\frac{m_{2c}(x_1)-m_{2c}(x_2)}{x_1x_2} \left(\bv_1^\top\frac{\Sigma}{(1+m_{2c}(x_1)\Sigma)(1+m_{2c}(x_2)\Sigma)}\bv_2 \right)^2\right] ,
\end{split}
\ee
for $x_1,x_2\in \R_+$ and $\bv_1,\bv_2\in \R^n$, where we abbreviated $m_{2c}(x)\equiv m_{2c}(x+\ii 0_+)$ for $x\in \R$. It is complex with $\im m_{2c}(x)=\pi \rho_{2c}(x)$ if $x\in \supp(\rho_{2c})$ (see \eqref{ST_inverse}); otherwise $m_{2c}(x)$ is real if $x\notin \supp(\rho_{2c})$.

We are now ready to state the main results. We first consider the convergence of the process $Z_{\eta, E}(\bv,f)$ with $\eta=1$, i.e., the linear eigenvector statistics on the global scale.  

\begin{theorem}\label{main_thm}
Suppose $d_N$, $X$ and $\Sigma$ satisfy Assumption \ref{main_assm}, and there exists a constant $c_0>0$ such that 
\begin{equation}\label{size_condition}
\max_{1\le i \le n, 1\le j \le N} \mathbb E|\sqrt{N}x_{ij}|^{8+c_0} \le c_0^{-1}.
\end{equation} 
Fix any $k\in \N$ and constants $a,b>0$. For any sequences of deterministic unit vectors $\bv_1\equiv \bv_1^{(n)},\ldots, \bv_k\equiv \bv_k^{(n)}\in \R^n$, and functions $f_1,\ldots, f_k \in \cal C^{1,a,b}(\R_+)$, the random vector 
\be\label{Z10}  (Z_{1,0}(\bv_i, f_i))_{1\le i \le k}=\left( \sqrt{N}  \left(\left\langle \bv_i, f_i\left( {\mathcal Q}_1 \right) \bv_i\right\rangle -\int_{\lambda_-}^{\lambda_+}  f_i\left( x  \right) \dd F_{1c,\bv_i}(x)\right)\right)_{1\le i \le k}\ee
converges weakly to a Gaussian vector $(\mathscr G_1, \ldots, \mathscr G_k)$ with mean zero and covariance function 
\be\label{complicated}
\begin{split}
 \mathbb E\left(\mathscr G_i \mathscr G_j \right) &=  \frac{1}{\pi^2}\iint_{x_1, x_2} f_i\left( x_1\right)f_j\left(x_2\right) \lim_{N\to \infty}\al^{(N)}(x_1, x_2, \bv_i,\bv_j) \dd x_1  \dd x_2  \\
&+\frac{1}{\pi^2}PV\iint_{x_1, x_2 } \frac{f_i\left( x_1\right)f_j\left(x_2\right)}{x_1-x_2} \lim_{N\to \infty}\beta^{(N)}(x_1, x_2, \bv_i,\bv_j)  \dd x_1  \dd x_2 \\
&+2\int f_i\left(x\right)f_j(x) \lim_{N\to \infty}\frac{\rho_{2c}^{(N)}(x)}{x^2}  \left(\bv_i^\top \frac{\Sigma}{(1+m^{(N)}_{2c}(x)\Sigma)(1+\overline m_{2c}^{(N)}(x)\Sigma)}\bv_j\right)^2\dd x ,
\end{split}
\ee
as long as all the limits in \eqref{complicated} converge. Here, $PV$ stands for ``principal value", that is, 
\be\nonumber PV\iint_{x_1, x_2 } \frac{g\left( x_1,x_2\right)}{x_1-x_2}  \dd x_1  \dd x_2:=\lim_{\delta\downarrow 0} \iint_{x_1, x_2 } \frac{g\left( x_1,x_2\right)(x_1-x_2)}{(x_1-x_2)^2 +\delta^2}  \dd x_1  \dd x_2\ee
for any function $g$ with sufficient regularity.
\end{theorem}
\begin{remark}\label{rem stronger}
Compared to the results in \cite{Bai2007,Xia2015}, our results are stronger in the following senses. 
\begin{itemize}
\item[(i)] We can deal with very general $\Sigma$ without assuming $\Sigma=I_n$ or \eqref{isointro}. 
\item[(ii)] We require weaker regularity of the functions $f_i$. 
\item[(iii)] It was assumed that the entries $x_{ij}$ are i.i.d.\,with $E|\sqrt{N}x_{ij}|^4=3$ in \cite{Bai2007}, while we obtain an extra term in \eqref{defal} that depends on the fourth cumulants of the $X$ entries.
\item[(iv)] We allow for different choices of vectors $\bv_i$ in the random vector \eqref{Z10}, while \cite{Bai2007,Xia2015} only considered the case with $\bv_i=\bv$ for all $i$. This generalization is important for applications, since if we want to estimate the difference, say $Z_{1,0}(\bv_1, f_1)- Z_{1,0}(\bv_2, f_2)$, then it is crucial to know the covariance between them.
\end{itemize}
We remark that \cite{Bai2007} only requires finite fourth moment for the entries of $X$, while we need the stronger moment assumption \eqref{size_condition}. However, we notice that the finite 8th moment condition is assumed in \cite{Xia2015}.
\end{remark}

Next, we consider the convergence of the process $Z_{\eta, E}(\bv,f)$ with $\eta\ll 1$,  i.e. the local linear eigenvector statistics.

\begin{theorem}\label{main_thm2}
Fix $E> 0$ and $N^{-1+c_1}\le \eta \ll 1$ for some constant $c_1>0$. Suppose $d_N$, $X$ and $\Sigma$ satisfy Assumption \ref{main_assm}, and there exist a constant $c_0>0$ such that 
\begin{equation}\label{size_condition2}
\max_{1\le i \le n, 1\le j \le N} \mathbb E|\sqrt{N}x_{ij}|^{a_\eta+c_0} \le c_0^{-1}, \quad a_\eta:= \frac{8}{1 - \log_N \eta}.
\end{equation} 
Fix any $k\in \N$ and constants $a,b>0$. For any sequences of deterministic unit vectors $\bv_1\equiv \bv_1^{(n)},\ldots, \bv_k\equiv \bv_k^{(n)}\in \R^n$, and functions $f_1,\ldots, f_k \in \cal C^{1,a,b}(\R_+)$, the random vector 
\be\nonumber  (Z_{\eta,E}(\bv_i, f_i))_{1\le i \le k}=\left(\sqrt{  \frac N \eta}  \left(\left\langle \bv_i, f\left( \eta^{-1}({\mathcal Q}_1 - E)  \right) \bv_i\right\rangle -\int_{\lambda_-}^{\lambda_+}  f\left(\eta^{-1} (x - E) \right) \dd F_{1c,\bv_i}(x)\right)\right)_{1\le i \le k}\ee
converges weakly to a Gaussian vector $(\mathscr G_1, \ldots, \mathscr G_k)$ with mean zero and covariance function 
\be
\begin{split}\label{simpfff}
&\mathbb E\left(\mathscr G_i \mathscr G_j \right)=\lim_{N\to \infty}\frac{2\rho_{2c}^{(N)}(E)}{E^2}  \left(\bv_i^\top \frac{\Sigma}{(1+m_{2c}^{(N)}(E)\Sigma)(1+\overline m_{2c}^{(N)}(E)\Sigma)}\bv_j\right)^2  \int f_i\left(x\right)f_j(x) \dd x 
\end{split}
\ee
as long as the limit in \eqref{simpfff} converges. 
\end{theorem}
\begin{remark}
Note that for $E$ outside $\supp(\rho_{2c})$, $(\mathscr G_1, \ldots, \mathscr G_k)$ converges to zero in probability. This is due to the fact that locally  there is no eigenvalue around $E$, and hence both $ f\left( \eta^{-1}(\lambda_i - E)  \right)$, $1\le i \le N\wedge n$, and $f\left(\eta^{-1} (x - E) \right)$, $x\in \supp(\rho_{2c})$, are of order $\oo(1)$. 
\end{remark}

\vspace{5pt}

We define the following process of resolvents
 \be\label{YetaE0}
 \cal Y_{\eta,E}(\bv,w):=\sqrt{N\eta} \bv^\top \left(R(E+w\eta)+\frac{(E+w\eta)^{-1}}{1+m_{2c}(E+w\eta)\Sigma}\right)\bv ,
 \ee
 where $R(z):= (\cal Q_1 -z)^{-1}= O^\top \cal G_1(z) O$ (recall \eqref{def_green}), $\bv$ is a deterministic vector in $\R^{n}$ and $w$ is a  fixed complex number in $\C$. Note that we have 
$ \cal Y_{\eta,E}(\bv,\overline w)=\overline {\cal Y}_{\eta,E}(\bv,w).$ 
To prove Theorems \ref{main_thm} and \ref{main_thm2}, we will first prove an intermediate CLT for the finite dimensional distribution of the process $\cal Y_{\eta,E}(\bv,w)$. We expect these results to be of independent interest.
To state them, we define the functions \smash{$\wh \al, \wh \beta: \C^{2}\times \R^{2n}\to \C$} as 
\be
\begin{split}\label{defal2}
&\wh \al (z_1, z_2, \bv_1,\bv_2)\equiv \wh \al^{(N)}(z_1, z_2, \bv_1,\bv_2)\\
&:=\frac{m_{2c}(z_1)m_{2c}(z_2)}{z_1 z_2}  \sum_{i=1}^n \frac{\sum_{j=1}^N\kappa_4(i,j) }{N} \left( \frac{\Sigma^{1/2}}{1+m_{2c}(z_1)\Sigma}\bv_1\right)^2_i \left(\frac{\Sigma^{1/2}}{1+m_{2c}(z_2)\Sigma}\bv_2\right)^2_i  , 
\end{split}
\ee
and
\be
\begin{split}\label{defbeta2}
&\wh \beta(z_1, z_2, \bv_1,\bv_2)\equiv \wh \beta^{(N)}(z_1, z_2, \bv_1,\bv_2)\\
&:= 2\frac{m_{2c}(z_1)-m_{2c}(z_2)}{z_1z_2(z_1- z_2)} \left(\bv_1^\top \frac{\Sigma}{(1+m_{2c}(z_1)\Sigma)(1+m_{2c}(z_2)\Sigma)}\bv_2 \right)^2 ,
\end{split}
\ee
for $z_{1},z_2\in \C$ and $\bv_{1},\bv_2\in \R^{n}$, where as a convention, $ (z_1- z_2)^{-1}(m_{2c}(z_1)-m_{2c}(z_2))$ is understood as $m'_{2c}(z_1)$
when $z_1=z_2$. Denote $\mathbb H:=\{z\in \C: \re z>0, z\notin \R\}$. Now, we state the CLT for $\cal Y_{1,0}(\bv,w)$.

 \begin{theorem}\label{main_thm3}
Suppose $d_N$, $X$ and $\Sigma$ satisfy Assumption \ref{main_assm}, and there exists a constant $c_0>0$ such that \eqref{size_condition} holds. Fix any $k\in \N$ and complex numbers $z_1, \ldots, z_{k}\in  \mathbb H$. For any sequence of deterministic unit vectors \smash{$\bv_1\equiv \bv_1^{(n)}$}, \smash{$\ldots, \bv_k\equiv \bv_k^{(n)}\in \R^n$}, the random vector  $  \left( \cal Y_{1,0}(\bv_1, z_1),\ldots, \cal Y_{1,0}(\bv_{k}, z_{k})\right) $
converges weakly to a complex Gaussian vector $(\Upsilon_1, \ldots, \Upsilon_{k})$ with mean zero and covariances 
\begin{align}\label{EYiYj}
\mathbb E \Upsilon_i \Upsilon_j &=\lim_{N\to \infty }\left[  \wh \al^{(N)}(z_i, z_j, \bv_i,\bv_j) +  \wh \beta^{(N)}(z_i, z_j, \bv_i,\bv_j)\right], \quad 1\le i , j \le k, 
\end{align}
as long as the limit in \eqref{EYiYj} converges. 
\end{theorem}

Then, we give the CLT for $ \cal Y_{\eta,E}(\bv,w)$ with $\eta\ll 1$. For $E$ outside the spectrum, that is,
\be\nonumber E\in S_{out}(\tau):=\left\{E: \text{dist}(E,  {\rm{supp}} \, \rho_{2c} )\ge \tau \right\},\ee
we will have a stronger result. 

\begin{theorem}\label{main_thm4}
Fix $E> 0$ and $N^{-1+c_1}\le \eta \ll 1$ for some constant $c_1>0$. Suppose $d_N$, $X$ and $\Sigma$ satisfy Assumption \ref{main_assm}, and there exists a constant $c_0>0$ such that \eqref{size_condition2} holds. Fix any $k \in \N$ and complex numbers $w_1, \ldots, w_{k}\in \mathbb H$. 
For any sequence of deterministic unit vectors \smash{$\bv_1\equiv \bv_1^{(n)},\ldots, \bv_k\equiv \bv_k^{(n)}\in \R^n$}, the random vector $ ( \cal Y_{\eta,E}(\bv_1, w_1),\ldots, \cal Y_{\eta,E}(\bv_{k}, w_{k}) )$ converges weakly to a complex Gaussian vector $(\Upsilon_1, \ldots, \Upsilon_{k})$ with mean zero and covariances 
\be\label{simpYYY}
\begin{split}
\mathbb E \Upsilon_i \Upsilon_j &=\mathbf 1(\im w_i \cdot \im w_j <0)  \lim_{N\to \infty}\frac{4\ii \cdot \im m_{2c}^{(N)}(E)}{E^2 (w_i-  w_j)} \left(\bv_i^\top\frac{\Sigma}{(1+m_{2c}^{(N)}(E)\Sigma)(1+\overline m_{2c}^{(N)}(E)\Sigma)} \bv_j\right)^2 ,
\end{split}
\ee
as long as the limit exists. In addition, if $E\in S_{out}(\tau)$ for some constant $\tau>0$ and \eqref{size_condition} holds, then for any $0<\eta\ll 1$ the random vector $ \eta^{-1/2}( \cal Y_{\eta,E}(\bv_1, w_1),\ldots, \cal Y_{\eta,E}(\bv_{k}, w_{k}) )$ converges weakly to a real Gaussian vector $(\Upsilon_1, \ldots, \Upsilon_{k})$ with mean zero and covariances
\be\label{simpYYYout}
\begin{split}
\mathbb E \Upsilon_i \Upsilon_j &=\lim_{N\to \infty }\left[ \wh \al^{(N)}(E, E, \bv_i,\bv_j) + \wh \beta^{(N)}(E, E, \bv_i,\bv_j)\right], 
\end{split}
\ee
as long as the limit exists.
\end{theorem}

\begin{remark}\label{positivevariance}
The reader may notice that given a vector $\bv\in \R^n$, the term $\wh \al^{(N)}(E, E, \bv,\bv)$ can be negative if the fourth cumulants $\kappa_4(i,j)$ are negative (e.g., for Rademacher entries). However, using $\kappa_4(i,j)\ge -2$, we have the simple bound
\be\nonumber
\begin{split} 
\wh \al (E,E, \bv,\bv) &\ge - \frac{2m_{2c}^2(E) }{E^2}  \sum_{i=1}^n \left( \frac{\Sigma^{1/2}}{1+m_{2c}(E)\Sigma}\bv\right)^4_i  \\
& \ge - \frac{2m_{2c}'(E) }{E^2} \left(\bv^\top \frac{\Sigma}{(1+m_{2c}(E)\Sigma)(1+m_{2c}(E)\Sigma)}\bv \right)^2 = -  \wh \beta(E, E, \bv,\bv), 
\end{split}
\ee
where in the second step we used that
\be\nonumber m_{2c}^2(E) = \left( \int \frac{\rho_{2c}(x)}{x-E}\dd x\right)^2 \le  \int \frac{\rho_{2c}(x)}{(x-E)^2}\dd x =m_{2c}'(E)  \ee
by Cauchy-Schwarz inequality. Hence, the sum $\wh \al^{(N)}(E, E, \bv,\bv) + \wh \beta^{(N)}(E, E, \bv,\bv)$ stays positive, as it should be because it is the asymptotic variance of  $\eta^{-1/2}\cal Y_{\eta,E}(\bv, w)$.
\end{remark}

\begin{remark}\label{rem localonly}
For the local statistics, Theorems \ref{main_thm2} and \ref{main_thm4}, to hold, we only need the spectrum $\rho_{2c}$ to behave well locally around $E$. In particular, the assumption $|d_N-1|\ge \tau$ in Assumption \ref{main_assm} is not needed as long as $E$ is away from zero. Moreover, the regularity of $\Sigma$ is not required to hold for the full spectrum---if $E$ is in the bulk, we only need that the density of $\rho_{2c}$ is of order 1 around $E$; if $E$ is near an edge, we only need that the corresponding edge is regular; if $E$ is outside the spectrum, we only need that $E$ is away from the spectrum by a distance of order 1. However, for simplicity of presentation, we do not attempt to find the weakest possible regularity assumption for Theorems \ref{main_thm2} and \ref{main_thm4}. 
\end{remark}

\begin{remark}
The results in Theorems \ref{main_thm}, \ref{main_thm2}, \ref{main_thm3} and \ref{main_thm4} can be used to give the CLT of more general quantities
$\left\langle \bu, f\left( \eta^{-1}({\mathcal Q}_1 - E)  \right) \bv\right\rangle$ or $\langle \bu, R(E+w\eta)\bv\rangle,$ by using the polarization identity
\be\nonumber \langle \bu , \cal M \bv\rangle= \frac12\langle (\bu+\bv), \cal M(\bu+\bv)\rangle- \frac12\langle (\bu-\bv), \cal M(\bu-\bv)\rangle \ee
for any symmetric matrix $\cal M$. Moreover, by considering real and imaginary parts separately, we can also extend the results to the case with complex test vectors $\bu$ and $\bv$.
\end{remark}

\begin{remark}\label{rem compare1}
Consider a special case where $f_i$'s are analytic functions on an open neighborhood of the real interval $[\lambda_-,\lambda_+]$, $d_N\to d\in (0,\infty)\setminus \{1\}$, and the $X$ entries are i.i.d. random variables satisfying \eqref{size_condition} and $\mathbb E|\sqrt{N}x_{ij}|^4 = 3$. Moreover, suppose that \eqref{isointro} holds for a sequence of deterministic unit vectors $\bv_n$. Then, by Theorem \ref{main_thm3}, we get that for fixed $z_1, z_2\in \mathbb H$, the covariance between $ \cal Y_{1,0}(\bv_n, z_1)$ and $\cal Y_{1,0}(\bv_n, z_{2})$ converges to
\be
\begin{split} \label{covYz}
\mathbb E \Upsilon_1 \Upsilon_2 &=\lim_{N\to \infty }  \frac{2\left[ \bv_n^\top (1+m_{2c}(z_1) \Sigma)^{-1}{\bv_n} - \bv_n^\top(1+m_{2c}(z_2)\Sigma)^{-1} {\bv_n} \right]^2}{z_1z_2(z_1- z_2)(m_{2c}(z_1)-m_{2c}(z_2))} \\
&=\lim_{N\to \infty }  \frac{2\left[ \int (1+m_{2c}(z_1) t)^{-1}\pi_{\Sigma}(\dd t)   - \int (1+m_{2c}(z_2) t)^{-1}\pi_{\Sigma}(\dd t)   \right]^2}{z_1z_2(z_1- z_2)(m_{2c}(z_1)-m_{2c}(z_2))}\\
&=\lim_{N\to \infty }  \frac{2(z_1m_{2c}(z_1)-z_2m_{2c}(z_2))^2}{d_N^{2} z_1z_2(z_1- z_2)(m_{2c}(z_1)-m_{2c}(z_2))} ,
\end{split}
\ee
where we used equation \eqref{deformed_MP21} in the last step. 

Now, we pick a contour $\mathcal C$ around $[\lambda_-,\lambda_+]$ in $\mathcal D$. Using Cauchy's integral formula, we get 
\be\nonumber Z_{1,0}(\bv_n,f_i)= \frac{-1}{2\pi \ii}\oint_{\mathcal C} f_i(z)\cal Y_{0,1}(\bv_n,z) \dd z.\ee
Then, using \eqref{covYz}, the covariance between $Z_{1,0}(\bv_n,f_i)$ and $Z_{1,0}(\bv_n,f_j)$ converges to
\begin{align}
\lim_{N\to \infty } \E Z_{1,0}(\bv_n,f_i) Z_{1,0}(\bv_n,f_j)= -\frac{1}{4\pi^2}\oint_{\mathcal C}\oint_{\mathcal C} f_i(z_1)f_j(z_2)\lim_{N\to \infty } \E \cal Y_{0,1}(\bv_n,z_1)\cal Y_{0,1}(\bv_n,z_2) \dd z_1\dd z_2 \nonumber\\
= -\frac{1}{2\pi^2}\oint_{\mathcal C}\oint_{\mathcal C} f_i(z_1)f_j(z_2)\lim_{N\to \infty }  \frac{(z_1m_{2c}(z_1)-z_2m_{2c}(z_2))^2}{d_N^{2} z_1z_2(z_1- z_2)(m_{2c}(z_1)-m_{2c}(z_2))} \dd z_1\dd z_2,\label{covFz}
\end{align}
if the function $m_{2c}$ converges as $N\to \infty$. Of course, there are some technical details missing in the above derivation, but it can be made rigorous readily. The formula \eqref{covFz} recovers the result in Theorem 2(b) of \cite{Bai2007}.
\end{remark}

\begin{remark}
Suppose the setting of Remark \ref{rem compare1} holds. In addition, we consider sample covariance matrices with trivial population covariance $\Sigma=I_{n}$, and assume that  the vectors $\bv_1,\ldots, \bv_k$ are all equal to a unit vector $\bv$. Then, the covariance function in \eqref{complicated} can be reduced to
\be\label{simple}
\begin{split}
\mathbb E\left(\mathscr G_i \mathscr G_j \right)&= \frac{2}{d} \left[ \int f_i\left(x\right)f_j(x)\rho_c(x) \dd x- \int f_i\left(x\right) \rho_c(x) \dd x \cdot \int f_j\left(x\right) \rho_c(x) \dd x\right] ,
\end{split}
\ee
where $\rho_c(x)$ is the MP density, 
\be\nonumber \rho_c(x)= \frac{\sqrt{(x-\lambda_-)(\lambda_+-x)}}{2\pi d x }\mathbf 1_{x\in [\lambda_-,\lambda_+]}, \quad \lambda_\pm:= (1\pm \sqrt{d})^2.\ee
In \cite{Bai2007}, a derivation of \eqref{simple} using \eqref{covFz} was given assuming that $f_i$ are analytic. Later in \cite{Xia2015}, \eqref{simple} was proved for more general $f_i$ with continuous third order derivatives. For the convenience of readers, we now give a derivation of \eqref{simple} from our result \eqref{complicated}. 

When $\Sigma=I_{n}$, the self-consistent equation \eqref{deformed_MP21} reduces to 
\begin{equation}\label{MPid}
\frac{1}{m_{2c}(z)} = - z + \frac{{d_N}}{1+m_{2c}(z) },
\end{equation}
and its solution is 
\be\label{MPid2}
m_{2c}(z)= \frac{-(z+1-d_N)+\sqrt{(z-\lambda^{(N)}_-)(z-\lambda^{(N)}_+)}}{2z}, \quad \lambda^{(N)}_\pm:= (1\pm \sqrt{d_N})^2.
\ee
Then, for $\bv_1=\bv_2=\bv$, using \eqref{defbeta} and \eqref{MPid}, we can obtain that
\be\label{betavv}
\begin{split} 
&\frac{\beta(x_1, x_2, \bv,\bv)}{x_1-x_2}=\frac{d_N^{-2}}{x_1x_2 (x_1-x_2)} \re\left[\frac{(x_1m_{2c}(x_1)-x_2\overline m_{2c}(x_2))^2}{m_{2c}(x_1)-\overline m_{2c}(x_2)} - \frac{(x_1m_{2c}(x_1)-x_2  m_{2c}(x_2))^2}{m_{2c}(x_1)- m_{2c}(x_2)}\right].
\end{split}
\ee
Combining the identity
\be\nonumber
\begin{split}
(x_1m_{2c}(x_1)-x_2\overline m_{2c}(x_2))^2 &=m_{2c}(x_1) \overline m_{2c}(x_2) (x_1-x_2)^2+ x_1x_2 ( m_{2c}(x_1)- \overline m_{2c}(x_2))^2 \\
&+(x_1m_{2c}(x_1)+x_2\overline m_{2c}(x_2)) (x_1-x_2)(m_{2c}(x_1)-\overline m_{2c}(x_2)) 
\end{split}
\ee
with a similar idenity for $(x_1m_{2c}(x_1)-x_2  m_{2c}(x_2))^2$, we can simplify \eqref{betavv} as 
\begin{align} 
\frac{\beta(x_1, x_2, \bv,\bv)}{x_1-x_2}&=\frac{x_1-x_2}{d_N^2 x_1 x_2 }\re\left[ \frac{m_{2c}(x_1) \overline m_{2c}(x_2)}{m_{2c}(x_1) - \overline m_{2c}(x_2)} -  \frac{m_{2c}(x_1)m_{2c}(x_2))}{m_{2c}(x_1)-m_{2c}(x_2))}\right] \nonumber\\
&=\frac{-1}{d_N^{3}x_1 x_2 }\re\left[  (1+x_1m_{2c}(x_1))x_2( \overline m_{2c}(x_2) - m_{2c}(x_2))\right] \nonumber\\
&= -2d_N^{-3} \im m_{2c}(x_1)\cdot \im m_{2c}(x_2) \to -\frac{2\pi^2}{d} \rho_c(x_1)\rho_c(x_2),\label{reducebeta}
\end{align}
where in the second step we used \eqref{MPid} to get
\be\nonumber
\begin{split} 
 \frac{(x_1-x_2)m_{2c}(x_1) \overline m_{2c}(x_2)}{m_{2c}(x_1) - \overline m_{2c}(x_2)} &= 1 - \frac{d_N m_{2c}(x_1) \overline m_{2c}(x_2)}{(1+m_{2c}(x_1))(1+\overline m_{2c}(x_2))} = 1 - d_N^{-1} (1+x_1m_{2c}(x_1))(1+x_2 \overline m_{2c}(x_2)) ,
\end{split}
\ee
and a similar identity with $\overline m_{2c}(x_2)$ replaced by $ m_{2c}(x_2)$. On the other hand, we can check that
\be\nonumber \frac{\rho^{(N)}_{2c} (x)}{x^2 |1+m_{2c}^{(N)}(x)|^4} = d_N^{-2} \rho^{(N)}_{2c} (x)  \to d^{-1}\rho_c(x).\ee
Together with \eqref{reducebeta}, this shows that \eqref{complicated} can be reduced to \eqref{simple}.
\end{remark}

\section{Basic tools}\label{sec_maintools}

In this section, we introduce some notations and collect some basic tools that will be used in the proof. With the notations in \eqref{def_green}, the Stieltjes transforms of $F_{\mathbf v}$ are equal to $\langle \mathbf u, \mathcal G_1(X,z)\mathbf u\rangle$, where $\bu:=O\bv$. One of the most basic tools for the proof is the following asymptotic estimate 
\begin{equation}\label{iso_law}
\langle \mathbf u, \mathcal G_1(X,z)\mathbf u\rangle \approx m_{1c,\mathbf u}(z), 
\end{equation}
which we shall refer to as the anisotropic local law.  More precisely, an {\it{anisotropic local law}} is an estimate of the form (\ref{iso_law}) for all $\Im\, z \gg N^{-1}$. Such local law has been established in \cite{isotropic,KY,Anisotropic,yang2018} for sample covariance matrices, assuming certain moment conditions on the matrix entries. 

The anisotropic local law can be stated in a simple and unified fashion using the following $(N+n)\times (N+n)$ symmetric matrix $H$:
 \begin{equation}\label{linearize_block}
   H : = \left( {\begin{array}{*{20}c}
   { 0 } & \Lambda^{1/2}OX  \\
   {(\Lambda^{1/2}OX)^{\top}} & {0}  \\
   \end{array}} \right).
 \end{equation}
We define the resolvent of $H$ as
 \begin{equation}\label{eqn_defG}
 G (X,z):= \left( {\begin{array}{*{20}c}
   { - I_{n}} & \Lambda^{1/2}OX  \\
   {(\Lambda^{1/2}OX)^{\top}} & { - zI_{N}}  \\
\end{array}} \right)^{-1}, \quad z\in \mathbb C_+ . 
\end{equation}
Using the Schur complement formula, it is easy to check that
\begin{equation} \label{green2}
G = \left( {\begin{array}{*{20}c}
   { z\mathcal G_1} & \mathcal G_1 (\Lambda^{1/2}OX)  \\
   {(\Lambda^{1/2}OX)^{\top}\mathcal G_1} & { \mathcal G_2 }  \\
\end{array}} \right)= \left( {\begin{array}{*{20}c}
   { z\mathcal G_1} & (\Lambda^{1/2}OX)\mathcal G_2   \\
   {\mathcal G_2}(\Lambda^{1/2}OX)^{\top} & { \mathcal G_2 }  \\
\end{array}} \right).
\end{equation}
Thus, a control of $G$ yields directly a control of the resolvents $\mathcal G_1$ and $\mathcal G_2$. For simplicity of notations, we define the index sets
$\mathcal I_1:=\{1,...,n\}$, $ \mathcal I_2:=\{n+1,...,n+N\}$ and $\mathcal I:=\mathcal I_1\cup\mathcal I_2$. We shall consistently use latin letters $i,j\in\mathcal I_1$, greek letters $\mu,\nu\in\mathcal I_2$, and $\fa,\fb\in\mathcal I$. Then, we label the indices of $X$ as $X= (X_{i\mu}:i\in \mathcal I_1, \mu \in \mathcal I_2).$
For simplicity, given a vector $\mathbf v\in \mathbb C^{\mathcal I_{1,2}}$, we always identify it with its natural embedding in $\C^{\cal I}$. For example, we shall identify $\mathbf v\in \mathbb C^{\mathcal I_1}$ with $\left( {\begin{array}{*{20}c}
   {\mathbf v}  \\
   \mathbf 0_{N} \\
\end{array}} \right)$.

Now, we introduce the spectral decomposition of $G$. Let $\Lambda^{1/2}O  X   = \sum_{k = 1}^{n\wedge N} {\sqrt {\lambda_k} \xi_k } \zeta _{k}^\top $ be a singular value decomposition of $\Lambda^{1/2}O  X$. Then, using (\ref{green2}), we can get that for $i,j\in \mathcal I_1$ and $\mu,\nu\in \mathcal I_2$,
\begin{align}
& G_{ij} = \sum_{k = 1}^{n} \frac{z\xi_k(i) \xi_k^\top(j)}{\lambda_k-z},\ \quad \ G_{\mu\nu} = \sum_{k = 1}^{N} \frac{\zeta_k(\mu) \zeta_k^\top(\nu)}{\lambda_k-z}, \quad G_{i\mu} = G_{\mu i} = \sum_{k = 1}^{n\wedge N} \frac{\sqrt{\lambda_k}\xi_k(i) \zeta_k^\top(\mu)}{\lambda_k-z} .\label{spectral1}
\end{align}
With these spectral decompositions, one can obtain the bound
\be\label{roughbound}\|G(z)\|\le C(\im z)^{-1}
\ee 
for some constant $C>0$. Furthermore, from \eqref{spectral1} it is also easy to derive the following identities, which we shall refer to as Ward's identities. For the proof, one can refer to Lemma 6.1 of \cite{yang2018}.

\begin{lemma}\label{lem_comp_gbound}
Let $\{\bu_i\}_{i\in \cal I_1}$ and $\{\bv_\mu\}_{\mu\in \cal I_2}$ be orthonormal basis vectors in $\R^{\cal I_1}$ and $\R^{\cal I_2}$, respectively. 
For any $\mathbf x \in \mathbb C^{\mathcal I_1}$ and $\mathbf y \in \mathbb C^{\mathcal I_2}$, we have 
  \begin{align}
 & \sum_{i \in \mathcal I_1 }  \left| {G_{\mathbf x \mathbf u_i} } \right|^2  =\sum_{i \in \mathcal I_1 }  \left| {G_{ \mathbf u_i \mathbf x} } \right|^2  = \frac{|z|^2}{\eta}\im\left(\frac{ G_{\mathbf x\mathbf x}}{z}\right) ,\quad  \sum_{\mu  \in \mathcal I_2 } {\left| {G_{\mathbf y \mathbf v_\mu } } \right|^2 }=\sum_{\mu  \in \mathcal I_2 } {\left| {G_{\mathbf v_\mu \mathbf y } } \right|^2 }  = \frac{{\im G_{\mathbf y\mathbf y} }}{\eta }, \label{eq_sgsq2} \\
& \sum_{i \in \mathcal I_1 } {\left| {G_{\mathbf y \mathbf u_i} } \right|^2 } =\sum_{i \in \mathcal I_1 } {\left| {G_{ \mathbf u_i \mathbf y} } \right|^2 } = {G}_{\mathbf y\mathbf y}  +\frac{\bar z}{\eta} \im G_{\mathbf y\mathbf y}  ,\quad \sum_{\mu \in \mathcal I_2 } {\left| {G_{\mathbf x \mathbf v_\mu} } \right|^2 }= \sum_{\mu \in \mathcal I_2 } {\left| {G_{\mathbf v_\mu \mathbf x } } \right|^2 }= \frac{G_{\mathbf x\mathbf x}}{z}  + \frac{\bar z}{\eta} \im \left(\frac{G_{\mathbf x\mathbf x}}{z}\right) . \label{eq_sgsq3} 
 \end{align}
\end{lemma}

We will use the following notion of stochastic domination, which was first introduced in \cite{Average_fluc} and subsequently used in many works on random matrix theory. 
It simplifies the presentation of the results and their proofs by systematizing statements of the form ``$\xi$ is bounded with high probability by $\zeta$ up to a small power of $N$".

\begin{definition}[Stochastic domination]\label{stoch_domination}
\begin{itemize}
\item[(i)] 
Let
\be\nonumber \xi=\left(\xi^{(N)}(u):N\in\mathbb N, u\in U^{(N)}\right),\hskip 10pt \zeta=\left(\zeta^{(N)}(u):N\in\mathbb N, u\in U^{(N)}\right) \ee
be two families of nonnegative random variables, where $U^{(N)}$ is a possibly $N$-dependent parameter set. We say $\xi$ is stochastically dominated by $\zeta$, uniformly in $u$, if for any small constant $\epsilon>0$ and large constant $D>0$, 
\be\nonumber  \sup_{u\in U^{(N)}}\mathbb P\left[\xi^{(N)}(u)>N^\epsilon\zeta^{(N)}(u)\right]\le N^{-D}\ee
for large enough $N\ge N_0(\epsilon, D)$, and we will use the notation $\xi\prec\zeta$. 

\item[(ii)]  If for some complex family $\xi$ we have $|\xi|\prec\zeta$, then we write $\xi \prec \zeta$ or $\xi=\OO_\prec(\zeta)$.

\item[(iii)] We say an event $\Xi$ holds with high probability if for any fixed $D>0$, $\mathbb P(\Xi)\ge 1- N^{-D}$ for large enough $N$.
\end{itemize}
\end{definition}

The next lemma collects basic properties of stochastic domination, which will be used tacitly throughout the proof .

\begin{lemma}[Lemma 3.2 in \cite{isotropic}]\label{lem_stodomin}
Let $\xi$ and $\zeta$ be two families of nonnegative random variables, and $C>0$ be a large constant.
\begin{itemize}
\item[(i)] Suppose that $\xi (u,v)\prec \zeta(u,v)$ uniformly in $u\in U$ and $v\in V$. If $|V|\le N^C$, then $\sum_{v\in V} \xi(u,v) \prec \sum_{v\in V} \zeta(u,v)$ uniformly in $u$.

\item[(ii)] If $ \xi_1 (u)\prec \zeta_1(u)$ and $ \xi_2 (u)\prec \zeta_2(u)$ uniformly in $u\in U$, then $\xi_1(u)\xi_2(u) \prec \zeta_1(u)\zeta_2(u)$ uniformly in $u\in U$.

\item[(iii)] Suppose that $\Psi(u)\ge N^{-C}$ is deterministic and $\xi(u)$ satisfies $\mathbb E\xi(u)^2 \le N^C$ for all $u$. Then, if $\xi(u)\prec \Psi(u)$ uniformly in $u$, we have $\mathbb E\xi(u) \prec \Psi(u)$ uniformly in $u$.
\end{itemize}
\end{lemma}

Throughout the rest of this paper, we will consistently use the notation $z=E+i\eta$ for the spectral parameter $z$. We define the spectral domain
\begin{equation}\label{eq_domainD}
\mathbf D \equiv \mathbf D(\omega,N) := \{z\in \mathbb C_+:  |z|\ge \omega 
, N^{-1+\omega} \le \eta\le \omega^{-1}\},
\end{equation}
for some small constant $\omega>0$. We will also consider a domain that is outside  $\supp(\rho_{2c})$:
\begin{equation}\label{eq_domainDout}
\mathbf D_{out}\equiv \mathbf D_{out}(\omega,N) := \{z\in \mathbb C_+:  |z|\ge \omega, 0 < \eta\le \omega^{-1}, \text{dist}(E, \supp(\rho_{2c}))\ge \omega \}.
\end{equation}
Recalling the condition (\ref{regular1}), we can take $\omega$ to be sufficiently small such that $\omega\le \lambda_-/2$. Define the distance to the spectral edges as
\begin{equation}
\kappa:= \min_{1\le k \le 2L}\vert E -a_k\vert.
 \label{KAPPA}
\end{equation}
Then, we have the following estimates for $m_{2c}$: for $z,z_1,z_2\in \mathbf D(\omega,N)\cup \mathbf D_{out}(\omega,N)$,
\begin{align}
&\vert m_{2c}(z) \vert \lesssim 1, \quad  \Im \, m_{2c}(z) \lesssim \begin{cases}
    {\eta}/{\sqrt{\kappa+\eta}}, & \text{ if } E \notin \text{supp}\, \rho_{2c}\\
    \sqrt{\kappa+\eta}, & \text{ if } E \in \text{supp}\, \rho_{2c}\\
  \end{cases};\label{Immc} \\
&|m_{2c}'(z)|\lesssim (\kappa+\eta)^{-1/2}, \quad |m_{2c}(z_1)-m_{2c}(z_2)|\lesssim \sqrt{|z_1-z_2|};\label{m'c}\\
&\max_{i\in \mathcal I_1} \vert (1 + m_{2c}(z)\sigma_i)^{-1} \vert = \OO(1).\label{Piii}
\end{align}
The reader can refer to \cite[Appendix A]{Anisotropic} and \cite[Lemma 4.5]{ding2019spiked} for the proof.

Our local law of resolvents will be stated under a bounded support condition. With a standard truncation argument, the moment assumption on $X$ entries will imply certain bounded support condition with probability $1-\oo(1)$.

\begin{definition}[Bounded support condition] \label{defn_support}
We say a matrix $X$ satisfies the {\it{bounded support condition}} with $q$, if
\begin{equation}
\max_{i\in \mathcal I_1, \mu \in \mathcal I_2}\vert X_{i\mu}\vert \le q. \label{eq_support000}
\end{equation}
Here, $q\equiv q_N$ is a deterministic parameter and usually satisfies $ N^{-{1}/{2}} \leq q \leq N^{- \phi} $ for some small constant $\phi>0$. Whenever (\ref{eq_support000}) holds, we say that $X$ has support $q$. 
\end{definition}

We define the deterministic limit of $G(z)$,
\begin{equation}\label{defn_pi}
\Pi (z): = \left( {\begin{array}{*{20}c}
   { -(1+m_{2c}(z)\Lambda)^{-1}} & 0  \\
   0 & {m_{2c}(z)I_{N}}  \\
\end{array}} \right) ,
\end{equation}
and the control parameter
\begin{equation}\label{eq_defpsi}
\Psi (z):= \sqrt {\frac{\Im \, m_{2c}(z)}{{N\eta }} } + \frac{1}{N\eta}.
\end{equation}
Now, we are ready to state some local laws for the resolvent $G(X,z)$, which have been proved in \cite{Anisotropic,yang2018}.

\begin{theorem}[Local laws]\label{lem_EG0}
Suppose $d_N$, $X$ and $\Sigma$ satisfy Assumption \ref{main_assm}. Suppose $X$ satisfies (\ref{eq_support000}) with $q\le N^{-\phi}$ for some constant $\phi>0$.
Then, the following estimates hold for $z\in \mathbf D $:
\begin{itemize}
\item {\bf the anisotropic local law}: for any deterministic unit vectors $\mathbf u, \mathbf v \in \mathbb C^{\mathcal I}$,
\begin{equation}\label{aniso_law}
\left| \langle \mathbf u, G(X,z) \mathbf v\rangle - \langle \mathbf u, \Pi (z)\mathbf v\rangle \right| \prec q + \Psi(z);
\end{equation}

\item {\bf the averaged local law}:
\begin{equation}\label{aver_law}
| m_2(X,z)-m_{2c}(z) |  \prec ({N\eta})^{-1}.
\end{equation}
\end{itemize}
For $z\in \mathbf D_{out}$, we have the following stronger estimates:
\begin{itemize}
\item {\bf the anisotropic local law}: for any deterministic unit vectors $\mathbf u, \mathbf v \in \mathbb C^{\mathcal I}$,
\begin{equation}\label{aniso_lawout}
\left| \langle \mathbf u, G(X,z) \mathbf v\rangle - \langle \mathbf u, \Pi (z)\mathbf v\rangle \right| \prec q + N^{-1/2};
\end{equation}

\item {\bf the averaged local law}:
\begin{equation}\label{aver_lawout}
| m_2(X,z)-m_{2c}(z) | \prec N^{-1}.
\end{equation}

\end{itemize}
All of the above estimates are uniform in the spectral parameter $z$. 
\end{theorem}
\begin{proof}
Under the high moment assumption with $q\prec N^{-1/2}$, the estimates \eqref{aniso_law}--\eqref{aniso_lawout} were proved in Theorem 3.6 of \cite{Anisotropic}. For more general $q$, they were proved in Theorems 3.6 and 3.8 of \cite{yang2018}. It remains to show \eqref{aver_lawout}. We shall use the following rigidity result for the eigenvalues, which is a corollary of (\ref{aver_law}). 

For any $1\le k\le 2L$, we define
\begin{equation}\nonumber 
N_k := \sum_{2l\le k} N\int_{a_{2l}}^{a_{2l-1}} \rho_{2c}(x)\dd x,
\end{equation}
which is the classical number of eigenvalues in $[a_{2k},\lambda_+]$. Then, we define the classical locations $\gamma_j$ for the eigenvalues of $\mathcal Q_2$ through
\begin{equation}\label{gammaj}
 1 - F_{2c}(\gamma_j) = \frac{j-1/2}{N},  \ \ 1\le j \le n\wedge N.
\end{equation}
Note that (\ref{gammaj}) is well-defined since the $N_k$'s are integers by Lemma \ref{Structure_lem}. For convenience, we denote $\gamma_0:=+\infty$ and $\gamma_{n\wedge N+1}:=0$. 

\begin{lemma}[Theorem 3.12 of \cite{Anisotropic}] \label{thm_largerigidity}
Suppose \eqref{aver_law} and the regularity conditions in Definition \ref{def_regular} hold. Then, for $\gamma_j \in [a_{2k},a_{2k-1}]$, we have that
\begin{equation}\label{rigidity2} 
| \lambda_j - \gamma_j| \prec [(N_{2k}+1-j)\wedge (j+1-N_{2k-1})]^{-1/3}N^{-{2}/{3}}.
\end{equation}
\end{lemma} 

For $z\in \mathbf D_{out}$, using definition \eqref{gammaj}, we get
\be\nonumber \left| \left(\frac1N\sum_{j=1}^{N\wedge n} \frac{1}{\gamma_j-z} - \frac{N-N\wedge n}{z}\right) - m_{2c}(z) \right|\prec N^{-1}, \ee
and using \eqref{rigidity2}, we get
\be\nonumber \left| \left(\frac1N\sum_{j=1}^{N\wedge n} \frac{1}{\gamma_j-z} - \frac{N-N\wedge n}{z}\right) - m_{2}(z) \right|= \left|  \frac1N\sum_{j=1}^{N\wedge n} \left(\frac{1}{\gamma_j-z} -\frac{1}{\lambda_j-z} \right) \right|\prec N^{-1}.\ee
There two estimates together imply \eqref{aver_lawout}.
\end{proof}

Another ingredient of the proof is the following cumulant expansion formula, whose proof is given in \cite[Proposition 3.1]{Cumulant1} and \cite[Section II]{Cumulant2}.

\begin{lemma}\label{cumulant lem}
Fix any $l\in \N$ and let $f\in \cal C^{l+1}(\R)$. Let $h$ be a real valued random variable with finite moments up to order $l+2$. Then, we have 
\be\nonumber \mathbb E [f(h) h]=\sum_{k=0}^{l}\frac1{k!}\kappa_{k+1}(h)\mathbb Ef^{(k)}(h)+ R_{l+1},\ee
where $\kappa_{k}(h)$ is the $k$-th cumulant of $h$ and $R_{l+1}$ satisfies that for any constant $\e>0$,
\be\nonumber R_{l+1}\lesssim \mathbb E\left| h^{l+2} \mathbf 1_{|h|>N^{\e-1/2}}\right|\cdot \|f^{(l+1)}\|_{\infty} + \mathbb E\left| h\right|^{l+2}\cdot \sup_{|x|\le N^{\e-1/2}}|f^{(l+1)}(x)|. \ee
\end{lemma}

Finally, we introduce the Helffer-Sj\"ostrand formula \cite{functional_calc}, which relates the convergence of the process $Z_{\eta,E}(\bv,f)$ to the CLT of the resolvents $\sqrt{N\eta} (G-\Pi)_{\bu\bu}$ with $\bu:=O\bv$. It was used to obtain (almost) sharp convergence rates for ESD (see e.g. \cite{Wigner_Bern,PY}) and VESD (see e.g. \cite{XYY_VESD}) of random matrices, and was applied to the study of mesoscopic eigenvalue statistics (see e.g. \cite{Mesoscopic,Mesoscopicsample,LX_Bernoulli}).

\begin{lemma}[Helffer-Sj{\"o}strand formula] \label{HSLemma}
Let $f\in \cal C^{1,a,b}$ for some fixed $a,b>0$. Let $\wt f$ be the almost analytic extension of $f$ defined by $\wt f(x+\ii y)= f(x)+ \ii (f(x+y)- f(x)).$ Let $\chi\in \cal C_c^\infty(\R)$ be a smooth cutoff function satisfying $\chi(0)=1$. Then, for any $E\in \R$, we have that
\be\nonumber f(E)=\frac{1}{\pi}\int_{\mathbb R^2} \frac{\partial_{\overline z}(\wt f(z)\chi (y))}{E-x-\ii y}\dd x\dd y,
\ee
where $\partial_{\overline z}:= \frac12(\partial_x + \ii \partial_y)$ is the antiholomorphic derivative.
\end{lemma}

\section{Overview of the proof} \label{sec_overview}

In this section, we give a brief overview of the proof of the main results. We first explain the basic strategy for the proof of Theorems \ref{main_thm3} and \ref{main_thm4}. To show the random vector $ ( \cal Y_{\eta,E}(\bv_1, w_1),\ldots, \cal Y_{\eta,E}(\bv_{k}, w_{k}) )$ converges weakly to a centered Gaussian vector $(\Upsilon_1, \ldots, \Upsilon_{k})$ for $N^{-1}\ll \eta \le 1$, we will show that the joint moments of $\cal Y_{\eta,E}(\bv_i, w_i)$, $1\le i \le k$, match those of $\Upsilon_i$, $1\le i \le k$, asymptotically up to arbitrary high order.  That is, for any fixed $\ell \in \N$ and $\ell$-tuple $(s_1, s_2, \ldots,  s_{\ell}) \in \{1,\ldots, k\}^{\ell}$ (where it is possible that $s_i=s_j$ for $i\ne j$), we want to show that 
\be\label{intro_match} \E\prod_{i=1}^\ell \cal Y_{\eta,E}(\bv_{s_i}, w_{s_i})-  \E\prod_{i=1}^\ell \Upsilon_{s_i} \to 0 .\ee
By the Wick's theorem (or Gaussian integration by parts), it suffices to show that $\E \cal Y_{\eta,E}(\bv_{s_1}, w_{s_1})\to 0$ and for $\ell\ge 2$,
\be\label{eq_strat_Wick} 
\E\prod_{i=1}^\ell \cal Y_{\eta,E}(\bv_{s_i}, w_{s_i}) = \sum_{i=2}^\ell \left[\E\left(\Upsilon_{s_1}\Upsilon_{s_i}\right)+\oo(1)\right]\cdot  \E \prod_{j\notin \{1,i\}}  \cal Y_{\eta,E}(\bv_{s_j}, w_{s_j})+\oo(1).
\ee
For simplicity of presentation, to explain the basic strategy for the proof of \eqref{eq_strat_Wick}, we consider a special case with $s_i\equiv 1$, $1\le i \le \ell$, in the discussion below. Then, we abbreviate $\bv_{1}$, $w_{1}$, $\cal Y_{\eta,E}(\bv_{1}, w_{1})$ and $\Upsilon_1$ as $\bv$, $w$, $Y$ and $\Upsilon$, respectively. Now, the problem is reduced to showing that  for any fixed $\ell\in \N$,
\be\label{eq_strat_simple} 
\E Y^{\ell} = (\ell-1) \left( \E \Upsilon^2 +\oo(1)\right)\cdot  \E Y^{\ell-2}+\oo(1).
\ee

With \eqref{green2} and \eqref{defn_pi}, we first rewrite \eqref{YetaE0} as 
 \be\label{YetaE}
 \begin{split}
Y(\bu,w) \equiv \cal Y_{\eta,E}(\bv,w) &=\sqrt{N\eta} \bu^\top \left(\cal G_1(z)- z^{-1}\Pi(z)\right)\bu  = {z}^{-1} {\sqrt{N\eta}}\bu^\top \left( G(z)- \Pi(z)\right)\bu,
\end{split}
 \ee
where $z:=E+w\eta$, $\bu:=O\mathbf v$ and $T:=\Lambda^{1/2}O$. Using the definitions of $G$ in \eqref{eqn_defG} and $\Pi$ in \eqref{defn_pi}, we obtain the simple identity  
\begin{align}\label{simple_id}
 G(z)-\Pi(z)&=G(z)\left[\Pi^{-1}(z)-G^{-1}(z)\right]\Pi(z) =G(z)\left( {\begin{array}{*{20}c}
   { - m_{2c}(z)\Lambda} & -TX  \\
   -(TX)^{\top} & {(m_{2c}^{-1}(z)+z)I_{n}}  \\
\end{array}} \right)\Pi(z),
\end{align}
which, together with \eqref{YetaE}, yields that 
\begin{align}\label{EYl}
\E Y^{\ell} &= z^{-1}\sqrt{N\eta} \E Y^{\ell-1}  \left[ \bu^\top  G(z) \left( {\begin{array}{*{20}c}
		{ - m_{2c}(z)\Lambda } & 0  \\
		0 & {0}  \\
\end{array}} \right) \Pi (z)\bu - \bu^\top G(z)\left( {\begin{array}{*{20}c}
		{0} & 0  \\
		{(TX)^\top} & {0}  \\
\end{array}}\right) \Pi (z)\bu \right].
\end{align}
The key to the proof is to evaluate the second term, i.e., the expectation \smash{$\mathbb E  Y ^{\ell-1} \sum_{i\in \cal I_1,\mu \in \cal I_2}G_{\mathbf u\mu} X_{i\mu} \bw (i)$}, where $\bw:=T^\top \Pi (z)\bu$. For this purpose, we adopt a strategy based on cumulant expansions as in some previous works on linear eigenvalue statistics of Wigner or sample covariance matrices \cite{Mesoscopic,Mesoscopicsample,LX_Bernoulli,Cumulant1}. Roughly speaking, with Lemma \ref{cumulant lem}, we need to estimate terms of the form
\be\label{intro_Er} 
-z^{-1}\sqrt{N\eta} \sum_{i,\mu}\frac{1}{r!} \kappa_{r+1}(X_{i\mu}) \E \frac{\partial^r \left(Y ^{\ell-1} G_{\mathbf u\mu} \right)}{\partial (X_{i\mu})^r} \bw (i),\quad 1\le r \le l,
\ee
plus an ``error term", say $R_{l+1}$, for some properly chosen $l\in \N$.  By definition of $G$, its derivative with respect to $X_{i\mu}$ is given by 
$\partial_{X_{i\mu}} G_{\fa\fb} = - G_{\fa  \mathbf t_i}G_{\mu \fb} - G_{\fa \mu}G_{ \mathbf t_i \fb}$, where we define the vector $\mathbf t_i:= T \mathbf e_i \in \R^{\cal I_1}$. 
We will use this identity to expand \eqref{intro_Er} and $R_{l+1}$ into a summation of polynomials of resolvent entries, each of which can be evaluated using the local laws in Theorem \ref{lem_EG0} above. For example, taking $r=1$ in \eqref{intro_Er} gives that
\begin{align}\label{est_1stderv}
\frac{\sqrt{N\eta}}{Nz}\sum_{i,\mu} \left[\E Y ^{\ell-1} G_{\mathbf u \mathbf t_i}G_{\mu\mu}  \bw (i)  + \E Y ^{\ell-1} G_{\mathbf u \mu}G_{\mathbf t_i\mu}  \bw (i)\right] + (\ell-1) \frac{2\eta}{z}\sum_{i,\mu}  \E Y ^{\ell-2} G_{\mathbf u\mathbf t_i}(G_{\mu \mathbf u})^2 \bw (i).
\end{align}
Notice that the first term contains the factor $N^{-1}\sum_\mu G_{\mu\mu}=m_2(z)$, which will cancels the first term in \eqref{EYl} up to a negligible error of order \smash{$(N\eta)^{-1/2}$} by the averaged local law \eqref{aver_law}. The factor $\sum_i G_{\mathbf u\mathbf t_i}\mathbf w (i)$ in the third term can be approximated by $\sum_i \Pi_{\mathbf u\mathbf t_i}\mathbf w (i)$ due to the anisotropic local law \eqref{aniso_law}. To estimate the second and third terms in \eqref{est_1stderv}, we still need to have an estimate for $\sum_{\mu}G_{\mathbf u\mu} G_{\mathbf v\mu} $ for arbitrary deterministic unit vectors $\bu$ and $\bv$. This can be obtain from the anisotropic local law for $G(z)$ by taking the derivative with respect to $z$, i.e., $\sum_{\mu}G_{\mathbf u\mu} G_{\mathbf v\mu} =\partial_z G_{\mathbf u \mathbf v} \approx \partial_z \Pi_{\mathbf u \mathbf v}$. With the above arguments, we find that the second term is an error of order \smash{$N^{-1/2}$}, while the third term will contribute to the first term on the right-hand side of \eqref{eq_strat_simple}. With a similar but more technical argument, we will show that the $r=3$ case of \eqref{intro_Er} gives a fourth cumulant dependent term that also contributes to the first term on the right-hand side of \eqref{eq_strat_simple}, while all the other cases lead to a negligible error. Combining all these cases together concludes \eqref{eq_strat_simple}.

However, in implementing the above strategy, there are some technical difficulties to deal with. A key issue is that under the finite 8th moment condition, we can only apply the cumulant expansion in Lemma \ref{cumulant lem} with $l$ as large as 7, in which case the error term $R_{l+1}$ will diverge when we estimate \smash{$\E Y^{\ell}$} for large $\ell$. In addition, a standard truncation argument (see \eqref{phiN} below) gives a truncated random matrix with bounded support of order $q=N^{-c}(N\eta)^{-1/4}$ for a small constant $c>0$. In this case, the anisotropic local law \eqref{aniso_law} is too weak so that the terms \eqref{intro_Er} are also out of control. To circumvent the above issue, we first assume a stronger moment condition that $X_{i\mu}$ has finite moments up to arbitrary high order (see \eqref{high_moments} below). Then, we can apply Lemma \ref{cumulant lem} with a sufficiently large $l$ so that $R_{l+1}$ can be bounded easily. In this case, another challenging task is to estimate \eqref{intro_Er} for arbitrary large $r$, where the polynomials of resolvent entries coming from high-order derivatives with respect to $X_{i\mu}$ will have some intricate algebraic structures. We will show that in each polynomial, there are sufficiently many small resolvent entries due to the anisotropic local law \eqref{aniso_law} and some $|G_{\mathbf u \mathbf t_i}|^2$ or $|G_{\mathbf u \mu}|^2$ factors, whose sum over $i$ or $\mu$ can be controlled using Ward's identities in Lemma \ref{lem_comp_gbound}. (In fact, without exploring the effect of Ward's identities, we cannot get good enough error bounds by using the anisotropic local law only.) The above argument will conclude the proof of \eqref{eq_strat_simple} under the stronger moment condition. After that, we use a comparison argument to extend it to the case with a weaker finite 8th moment condition. More precisely, given a random matrix $X$ satisfying \eqref{size_condition} or \eqref{size_condition2}, we can construct another random matrix ensemble $\wt X$ whose entries have finite moments up to arbitrary high order and have the same first four moments as those of $X$. With the four moment matching condition, we will adopt a Green's function comparison method developed in \cite{Anisotropic,yang2018} to show that $\E Y(X)^{\ell}$ matches \smash{$\E Y(\wt X)^{\ell}$} asymptotically, which completes the proof of \eqref{eq_strat_simple}. Extending the above argument allows us to establish the more general equation \eqref{eq_strat_Wick}, and thus conclude Theorems \ref{main_thm3} and \ref{main_thm4}. 

Finally, given Theorems \ref{main_thm3} and \ref{main_thm4}, we can derive Theorems \ref{main_thm} and \ref{main_thm2} through a direct application of the Helffer-Sj{\"o}strand formula in Lemma \ref{HSLemma}. More precisely, as in \eqref{intro_match}, we need to show that 
\be\label{intro_match2} \E\prod_{i=1}^\ell Z_{\eta,E}(\bv_{s_i}, f_{s_i})-  \E\prod_{i=1}^\ell \mathscr G_{s_i} \to 0. \ee
Then, similar to the argument in \cite{Mesoscopic}, the Helffer-Sj{\"o}strand formula allows us to reduce this problem to showing \eqref{intro_match}, although many technical details are required to establish this connection and to control all the errors. In particular, the anisotropic local law \eqref{aniso_law} under the finite 8th moment condition is not good enough for this purpose.  Hence, we again prove \eqref{intro_match2} under the stronger finite high moment condition \eqref{high_moments} first and then use the Green's function comparison argument to extend it to the general case in Theorems \ref{main_thm3} and \ref{main_thm4}.

Part of our proof is inspired by previous works on linear eigenvalue statistics of Wigner matrices and sample covariance matrices in \cite{Mesoscopic,Mesoscopicsample,LX_Bernoulli,Cumulant1}. In particular, similar to these works, our proof is also based on a cumulant expansion method as discussed above. On the other hand, our proof has the following novelties. First, we handle both global and local eigenvector statistics at the same time, while \cite{Cumulant1} only considered global statistics and  \cite{Mesoscopic,Mesoscopicsample,LX_Bernoulli} considered local statistics where the dependence on the fourth cumulant of the random matrix entries does not appear. Second, estimating error terms for linear eigenvector statistics is slightly harder than that for linear eigenvalue statistics (partly because the anisotropic local law is weaker than the averaged local law). In addition, we have considered the most general sample covariance model with non-diagonal $\Sigma$, while the previous works \cite{Mesoscopic,Mesoscopicsample,LX_Bernoulli} studied either Wigner matrices or sample covariance matrices with diagonal $\Sigma$. Thus, these works only use entrywise local laws (i.e., a special case of \eqref{aniso_law} with $\mathbf u$ and $\mathbf v$ being standard basis vectors), where all off-diagonal entries are small. 
In our case, however, the behavior of the generalized resolvent entry $G_{\mathbf u\mathbf v}$ is more complicated since the size of $\Pi_{\mathbf u\mathbf v}$ depends critically on the directions of $\mathbf u$ and $\mathbf v$. To deal with this issue, in the proof, we develop a systematic argument to estimate terms of the form \eqref{intro_Er} for any fixed $r$ by applying the anisotropic local law in a proper way. Third, the comparison argument that treats the extension to the finite 8th moment condition is also new. In fact, \cite{Mesoscopicsample,LX_Bernoulli} both assumed the finite high moment condition, while \cite{Mesoscopic} used a comparison argument based on a standard Lindeberg replacement trick and the four-moment matching condition. However, for linear eigenvector statistics, the comparison argument in \cite{Mesoscopic} fails due to intricate behaviors of generalized resolvent entries. Our proof is instead based on a continuous interpolation introduced in \cite{Anisotropic} and we develop a systematic way to bound the errors in the comparison argument.

\section{CLT for resolvents }\label{sec resolvent}

As discussed in Section \ref{sec_overview}, we first prove Theorem \ref{main_thm3} and Theorem \ref{main_thm4} under a stronger moment assumption: for any fixed $p\in \mathbb N$, there is a constant $C_p$ such that 
\begin{equation}\label{high_moments}
\max_{i,\mu}\mathbb E|\sqrt{N}X_{i\mu}|^p \le C_p.
\end{equation}
By Markov's inequality, $X$ has bounded support $q\prec N^{-1/2}$. In Section \ref{sec relax}, we will discuss how to relax it to  \eqref{size_condition} or \eqref{size_condition2} using a Green's function comparison argument.

 \begin{proposition}\label{main_prop}
Theorems \ref{main_thm3} and \ref{main_thm4} hold under the moment assumption \eqref{high_moments}.
\end{proposition}


Recalling the notation in \eqref{YetaE}, Proposition \ref{main_prop} follows from the following lemma on the convergence of moments. 

\begin{lemma}\label{lem moment}
Suppose $d_N$, $X$ and $\Sigma$ satisfy Assumption \ref{main_assm}, $N^{-1+c_1}\le \eta \le 1$,  and \eqref{high_moments} holds. Fix any $E>0$ and $k\in \N$. For any deterministic unit vectors $\bv_1,\ldots, \bv_{k}\in \R^n$ and fixed $w_1, \ldots, w_{k}\in \mathbb H$, we have 
\begin{align}\label{moments part}
&\mathbb E\left[ \prod_{s=1}^kY(\bu_{s},w_{s})\right]=\begin{cases} 
\sum \prod \eta \gamma(z_{s},z_{t},\bv_{s},\bv_{t})+ \OO_\prec\left( (N\eta)^{-1/2}\right), \ &\text{if $k\in 2\mathbb N$} \\ 
 \OO_\prec\left( (N\eta)^{-1/2}\right), \ &\text{otherwise}
\end{cases},
\end{align}
where we denoted $\bu_i:=O\bv_i$, $z_i : = E+w_i\eta$ and $\gamma(z_{s},z_{t},\bv_{s},\bv_{t}):=\wh \al(z_{s},z_{t},\bv_{s},\bv_{t}) + \wh \beta(z_{s},z_{t},\bv_{s},\bv_{t})$, and $\sum\prod$ means summing over all distinct ways of partitioning indices into pairs. 
In addition, if $ N^{-C}\le \eta\ll 1$ for some constant $C>1$ and $E\in S_{out}(\tau)$, we have the stronger estimate
\begin{align}\label{moments partout}
&\mathbb E\left[\prod_{s=1}^k \frac{Y(\bu_{s},w_{s})}{\sqrt{\eta}}\right]=\begin{cases} 
\sum \prod  \gamma(z_{s},z_{t},\bv_{s},\bv_{t})+ \OO_\prec\left( N^{-1/2}\right), \ &\text{if $k\in 2\mathbb N$} \\ 
 \OO_\prec\left(N^{-1/2}\right), \ &\text{otherwise}
\end{cases}.
\end{align}
\end{lemma}
\begin{remark}
In the statement of this lemma, we allow that $\bu_s=\bu_t$ and $z_s=z_t$ for $s\ne t$. In other words, we are calculating the multivariate moments
\be\nonumber \mathbb E\left[Y^{r_1}(\bu_{i_1},w_{i_1})\cdots Y^{r_k}(\bu_{i_{k}},w_{i_{k}})\right], \quad r_1,\ldots, r_k\in \N,\ee
if we combine identical terms. 
\end{remark}

\begin{proof}[Proof of Proposition \ref{main_prop}]
By Wick's theorem, \eqref{moments part} with $E=0$ and $\eta=1$ shows that the convergence in Theorem \ref{main_thm3} holds in the sense of moments, which further implies the weak convergence. Similarly, under the setting of Theorem \ref{main_thm4}, \eqref{moments part} shows that the random vector $ ( \cal Y_{\eta,E}(\bv_1, w_1),\ldots, \cal Y_{\eta,E}(\bv_{k}, w_{k}) )$ converges weakly to a complex centered Gaussian vector $(\Upsilon_1, \ldots, \Upsilon_{k})$ with covariances 
\begin{align*} 
\mathbb E \Upsilon_i \Upsilon_j =\lim_{N\to \infty }\left[\eta \wh \al^{(N)}(z_i, z_j, \bv_i,\bv_j) + \eta\wh \beta^{(N)}(z_i, z_j, \bv_i,\bv_j)\right].
\end{align*}
When $\eta\ll 1$, this expression can be simplified to \eqref{simpYYY}.

Finally, under the setting of Theorem \ref{main_thm4}, suppose $E\in S_{out}(\tau)$ and $ N^{-4} \le \eta \ll 1$. By Wick's theorem, \eqref{moments partout} shows that the random vector $ \eta^{-1/2}( \cal Y_{\eta,E}(\bv_1, w_1),\ldots, \cal Y_{\eta,E}(\bv_{k}, w_{k}) )$ converges weakly to a real centered Gaussian vector $(\Upsilon_1, \ldots, \Upsilon_{k})$ with covariances  
\begin{align*} 
\mathbb E \Upsilon_i \Upsilon_j =\lim_{N\to \infty }\left[ \wh \al^{(N)}(E, E, \bv_i,\bv_j) + \wh \beta^{(N)}(E,E, \bv_i,\bv_j)\right].
\end{align*}
Finally, if $E\in S_{out}(\tau)$ and $\eta\le N^{-4}$, we can show that the random vector $ \eta^{-1/2}( \cal Y_{\eta,E}(\bv_1, w_1),\ldots, \cal Y_{\eta,E}(\bv_{k}, w_{k}) )$ has the same asymptotic distribution as $( \eta_0^{-1/2}\cal Y_{\eta_0,E}(\bv_1, w_1),\ldots,  \eta_0^{-1/2}\cal Y_{\eta_0,E}(\bv_{k},w_{k}))$, $\eta_0:=N^{-4},$ using the bound
\be\nonumber \|G(E+w_i\eta)- G(E+w_i\eta_0)\| \lesssim |\eta-\eta_0|\|G(E+w_i\eta)\|\cdot \|G(E+w_i\eta_0)\| \lesssim N^{-4} \quad \text{with high probability}.\ee
Here, we used that by the rigidity estimate \eqref{rigidity2}, $\|G(z)\|=\OO(1)$ with high probability for $z\in \mathbf D_{out}$.
\end{proof}

In the rest of this section, we mostly focus on the proof of \eqref{moments part}. We will discuss how to extend the argument to \eqref{moments partout} at the end of this section. For simplicity of presentation, the bulk of the proof is devoted to the calculation of moments 
\be\label{k1k2}\mathbb E \left[ Y^{k_1}(\bu_{1},w_{1})Y^{k_2}(\bu_{2},\overline w_{2})\right] ,\quad k_1, k_2 \in \N, \quad \bu_{1},\bu_2\in \R^n,\quad w_1,w_2\in \C_+.\ee
The proof for the more general expression in \eqref{moments part} is almost the same, except for some immaterial changes of notations. 

In the following calculation, we write $Y(\bu_{2},\overline w_{2})$ as $\overline Y(\bu_{2}, w_{2})$ and abbreviate $z_1:=E+w_1\eta$, $z_2:=E+w_2\eta$, $G^{(1)}:=G(z_1)$, $G^{(2)}:=G(z_2)$ and $T=\Lambda^{1/2}O$. Moreover, 
we denote 
\be\label{Y12}
\begin{split}
& Y_1:=z_1Y(\bu_{1},w_{1})= \sqrt{N\eta} (G(z_1)-\Pi(z_1))_{ \bu_1 \bu_1} ,\quad Y_2:=z_2Y(\bu_{2},w_{2})= \sqrt{N\eta} (G(z_2)-\Pi(z_2))_{\bu_2 \bu_2} ,
\end{split}
\ee
and $\mathfrak G:=Y_1^{k_1}\overline Y_2^{k_2} .$ In the following proof, we focus on calculating $\E \mathfrak G$. Note that by the assumptions of Lemma \ref{lem moment}, we have $|z_{1}|\sim |z_2|\sim 1$. 
Hence, we can easily derive the estimates on \eqref{k1k2} from that on $\E \mathfrak G$ by using the trivial identity
\smash{$z_1^{-k_1}\overline z_2^{-k_2} \mathfrak G=Y^{k_1}(\bu_{1},w_{1})\overline Y^{k_2}(\bu_{2}, w_{2})$}.

Without loss of generality, we assume that $k_1\ge k_2$ and $k_1+k_2\ge 1$. Under the assumption \eqref{high_moments},  $X$ has bounded support $q\prec N^{-1/2}$. Then, by \eqref{aniso_law}, we have 
\be\nonumber |Y_1|+ |Y_2|\prec \sqrt{N\eta}\Psi(z_1)+ \sqrt{N\eta}\Psi(z_2) =\OO(1).\ee
Then, using Lemma \ref{lem_stodomin} (iii), we get that for any fixed $n_1,n_2\in \N$,
\be\nonumber  \mathbb E|Y_1|^{n_1}|Y_2|^{n_2}\prec 1, \ee
where the second moment bound on $|Y_1|^{n_1}|Y_2|^{n_2}$ required by Lemma \ref{lem_stodomin} (iii) follows immediately from \eqref{roughbound}. We will use this bound tacitly in the proof. 

Using the identity \eqref{simple_id}, for $\bu_1\in \R^{\cal I_1}$, we get
\be\label{EG0}
\begin{split}
\mathbb E\mathfrak G&=\mathbb E  \sqrt{N\eta}\left\langle \bu_1,   G^{(1)} \left( {\begin{array}{*{20}c}
   { - m_{2c}(z_1)\Lambda } & 0  \\
   0 & {0}  \\
\end{array}} \right) \Pi (z_1)\bu_1\right\rangle Y_1^{k_1-1}\overline Y_2^{k_2}\\
&-\mathbb E  \sqrt{N\eta}\left\langle \bu_1,G^{(1)}\left( {\begin{array}{*{20}c}
   {0} & 0  \\
   {(TX)^\top} & {0}  \\
\end{array}}\right) \Pi (z_1)\bu_1\right\rangle Y_1^{k_1-1}\overline Y_2^{k_2}=: \cal M_1+\cal M_2.
\end{split}
\ee
Similar as in \eqref{kappa4}, we denote by $\kappa_k(i,\mu)$ the $k$-th cumulant of $\sqrt{N}X_{i\mu}$. Then, using Lemma \ref{cumulant lem} with $h=X_{i\mu}$, we can express $\cal M_2$ as
\begin{align}
&\cal M_2=-\sqrt{N\eta}\mathbb E  \sum_{i\in \cal I_1,\mu \in \cal I_2}G_{\bu_1\mu}^{(1)}X_{i\mu} \bw_1(i)Y_1^{k_1-1} \overline Y_2^{k_2} = \sum_{k=1}^l \mathfrak G_k+ \cal E, \label{cumulant1}
\end{align}
 where we denoted $\bw_1:=T^\top \Pi (z_1)\bu_1$. The terms on the right-hand side of \eqref{cumulant1} are defined as
 \be\label{G123}
 \begin{split}
 \mathfrak G_k :=-\frac{\sqrt{N\eta} }{k! N^{(k+1)/2}}  \sum_{i\in \cal I_1,\mu \in \cal I_2}  \bw_1(i) \mathcal \kappa_{k+1}(i,\mu)\mathbb E \frac{\partial^k  (G_{\bu_1\mu}^{(1)} Y_1^{k_1-1} \overline Y_2^{k_2} )}{\partial (X_{i \mu })^k } , \end{split}
\ee
and 
 \be\label{GE}\mathcal E: =-\sqrt{N\eta}  \sum_{i\in \cal I_1,\mu \in \cal I_2}  \bw_1(i) R_{l+1}(i\mu),\ee
where $R_{l+1}(i\mu)$ satisfies the bound
\begin{align*}
R_{l+1}(i\mu)\lesssim \mathbb E\left| X_{i\mu}^{l+2} \mathbf 1_{|X_{i\mu}|>N^{\e-1/2}}\right|\cdot \left\|\partial^{l+1}_{i\mu}  f_{i\mu} \right\|_{\infty} + \mathbb E\left| X_{i\mu}\right|^{l+2}\cdot \mathbb E\sup_{|x|\le N^{\e-1/2}}\left| \partial_{i\mu}^{l+1}  f_{i\mu} (H^{(i\mu)}+x\Delta_{i\mu})\right| . 
\end{align*}
Here, we abbreviated $f_{i\mu}:=G^{(1)}_{\bu_1\mu} Y_1^{k_1-1} \overline Y_2^{k_2}$, $\partial_{i\mu}:=\partial/\partial X_{i \mu }$, 
$\Delta_{i\mu}:=\begin{pmatrix} 0 & \bt_i \mathbf e_\mu^\top \\  \mathbf e_\mu \bt_i^\top & 0 \end{pmatrix}$ with $\mathbf t_i= T \mathbf e_i ,$
and $ H^{(i\mu)}:= H- X_{i\mu}\Delta_{i\mu}$ such that $H^{(i\mu)}$ is independent of $X_{i\mu}$. 
We next estimate the right-hand side of \eqref{cumulant1} term by term using the formula
\be\label{dervG}\frac{\partial^r G }{\partial (X_{i\mu})^r} = (-1)^r r! G (\Delta_{i\mu}G)^r.\ee
This can be derived from the following resolvent expansion: for any $x,x'\in \mathbb R$ and $k\in \mathbb N$,
\begin{equation} \label{eq_comp_expansion}
\begin{split}
G_{(i \mu)}^{x'} = G_{(i\mu)}^{x} &+\sum_{r=1}^{k}  (x-x')^k G_{(i\mu)}^{x}\left( \Delta_{i\mu} G_{(i\mu)}^{x}\right)^r +(x-x')^{k+1} G_{(i\mu)}^{x'}\left(\Delta_{i\mu} G_{(i\mu)}^{x}\right)^{k+1},
\end{split}
\end{equation}
where we abbreviated $G_{(i \mu)}^{x}:= G(H^{(i\mu)}+x\Delta_{i\mu}). $

\subsection{The leading term $\mathfrak G_1$} 
 We expand $\mathfrak G_1$ as 
  \be\label{G1}
 \begin{split}
 \mathfrak G_1 =&-\sqrt{ \frac{\eta}{N}}  \mathbb E  \sum_{i\in \cal I_1,\mu \in \cal I_2}\frac{\partial G_{\bu_1\mu}^{(1)}}{\partial X_{i\mu}} \bw_1(i)Y_1^{k_1-1} \overline Y_2^{k_2}  -\sqrt{ \frac{\eta}{N}}  \mathbb E  \sum_{i\in \cal I_1,\mu \in \cal I_2}G_{\bu_1\mu}^{(1)}\bw_1(i)\frac{\partial (Y_1^{k_1-1} \overline Y_2^{k_2})}{\partial X_{i\mu}} .
\end{split}
\ee
For the first term in \eqref{G1}, we have
\begin{align}
&-\sqrt{ \frac{\eta}{N}}  \mathbb E  \sum_{i\in \cal I_1,\mu \in \cal I_2}\frac{\partial G_{\bu_1\mu}^{(1)}}{\partial X_{i\mu}} \bw_1(i)Y_1^{k_1-1} \overline Y_2^{k_2} \nonumber\\
& = \sqrt{ \frac{\eta}{N}}  \mathbb E  \sum_{i\in \cal I_1,\mu \in \cal I_2}G^{(1)}_{\bu_1\mu} G^{(1)}_{\bt_i \mu}  \bw_1(i)Y_1^{k_1-1} \overline Y_2^{k_2} + \sqrt{ \frac{\eta}{N}}  \mathbb E  \sum_{i\in \cal I_1,\mu \in \cal I_2} G^{(1)}_{\mu\mu}  G^{(1)}_{\bu_1\bt_i} \bw_1(i)Y_1^{k_1-1} \overline Y_2^{k_2} \nonumber\\
& = \sqrt{ \frac{\eta}{N}}  \mathbb E  (G^{(1)}J_2G^{(1)} )_{\bu_1 \wt\bu_1} Y_1^{k_1-1} \overline Y_2^{k_2}+\sqrt{ N\eta}  \mathbb E \left(  m_2(z_1) G^{(1)}_{\bu_1\wt\bu_1}Y_1^{k_1-1} \overline Y_2^{k_2}\right),\label{G1_1+2}
\end{align}
where we denoted $J_2:=\begin{pmatrix}0 & 0 \\ 0 & I_{N} \end{pmatrix}$ and $\wt\bu_1:= \sum_{i\in \cal I_1} \bw_1(i)\bt_i = \wt\Lambda \Pi(z_1)\bu_1$ 
 with $\wt\Lambda:=\begin{pmatrix}\Lambda & 0 \\ 0 & 0 \end{pmatrix}.$
For the first term in \eqref{G1_1+2}, using the Ward's identities in Lemma \ref{lem_comp_gbound}, we can bound it by
\be\label{G1111}
\begin{split}
 & \sqrt{ \frac{\eta}{N}}\mathbb E\left[\left|G_{\bu_1\bu_1}\right|+\eta^{-1}\left|\im \left(z^{-1} G^{(1)}_{\bu_1\bu_1} \right)\right|\right]^{1/2} \left[\left|G_{\wt\bu_1\wt\bu_1}\right|+\eta^{-1} \left| \im \left( z^{-1}  G^{(1)}_{\wt\bu_1\wt\bu_1} \right)\right|\right]^{1/2}   \prec (N\eta)^{-1/2}, 
\end{split}
\ee
where in the second step we used \eqref{aniso_law} to bound $|G_{\bu_1\bu_1}|\prec 1 $ and $|G_{\wt\bu_1\wt\bu_1}|\prec 1$. On the other hand, using \eqref{aver_law}, we can estimate the second term in \eqref{G1_1+2} as
\be\label{G1222}
\begin{split}
 &\sqrt{ N\eta}  \mathbb E \left(  m_{2c}(z_1) (G^{(1)} \wt\Lambda \Pi(z_1))_{\bu_1\bu_1}Y_1^{k_1-1} \overline Y_2^{k_2}\right) + \OO_\prec\left((N\eta)^{-1/2}\right)  = -\cal M_1 + \OO_\prec\left((N\eta)^{-1/2}\right).
 \end{split}
\ee

Next, for the second term in \eqref{G1}, using \eqref{dervG}, we calculate that
\begin{align}
&-\sqrt{ \frac{\eta}{N}}  \mathbb E  \sum_{i\in \cal I_1,\mu \in \cal I_2}G_{\bu_1\mu}^{(1)}\bw_1(i)\frac{\partial (Y_1^{k_1-1} \overline Y_2^{k_2})}{\partial X_{i\mu}} \nonumber\\
&=2(k_1-1)\eta \mathbb E  \sum_{i\in \cal I_1,\mu \in \cal I_2}G^{(1)}_{\bu_1\mu}\bw_1(i)  G^{(1)}_{\bu_1 \mu}G^{(1)}_{\bt_i \bu_1}   Y_1^{k_1-2}  \overline Y_2^{k_2} \label{main1G}\\
&+2k_2\eta \mathbb E  \sum_{i\in \cal I_1,\mu \in \cal I_2}G^{(1)}_{\bu_1\mu}\bw_1(i)  \overline G^{(2)}_{\bu_2 \mu}\overline G^{(2)}_{\bt_i \bu_2}     Y_1^{k_1-1}  \overline Y_2^{k_2-1},\label{main2G}
\end{align}
where as a convention, the first term is zero if $k_1=1$ and the second term is zero if $k_2=0$. For the two terms \eqref{main1G} and \eqref{main2G}, we shall apply the identity 
\be\label{keyidentity}
\sum_{\mu\in \cal I_2}G_{\bu \mu}(z)G_{\bu' \mu}(z') = \frac{G_{\bu\bu'}(z) - G_{\bu\bu'}(z')}{z-z'},\quad z,z'\in \C, \quad \bu,\bu'\in \R^{\cal I},
\ee
which follows directly from the definition \eqref{eqn_defG}. Applying this identity to \eqref{main2G}, we can write that
\be\label{main2G2}
\begin{split}
\eta \sum_{i\in \cal I_1,\mu \in \cal I_2}G^{(1)}_{\bu_1\mu}\bw_1(i)  \overline G^{(2)}_{\bu_2 \mu}\overline G^{(2)}_{\bt_i \bu_2} = \eta \left(  \frac{G_{\bu_1\bu_2}(z_1) - G_{\bu_1\bu_2}(\overline z_2)}{z_1-\overline z_2}\right) \left(\Pi(z_1)\wt\Lambda \overline G^{(2)}\right)_{\bu_1\bu_2}\\
=    \eta \left(  \frac{\Pi_{\bu_1\bu_2}(z_1) - \Pi_{\bu_1\bu_2}(\overline z_2)}{z_1-\overline z_2}\right) \left(\Pi(z_1)\wt\Lambda \overline \Pi(z_2)\right)_{\bu_1\bu_2} +  \OO_\prec\left((N\eta)^{-1/2}\right),
\end{split}
\ee
where in the last step we used \eqref{aniso_law} and that $|z_1-\overline z_2|\gtrsim \eta$. 

On the other hand, for \eqref{main1G}, we develop another version of the identity \eqref{keyidentity} in order to deal with the case where $z$ is very close to $z'$ (or even $z=z'$). Suppose $z,z'\in \C_+$ satisfy that $\im z\gtrsim \eta$ and $\im z'\gtrsim \eta$. Then, we define the contour 
$\Gamma=\partial B_{c\eta}(z)\cup \partial B_{c\eta}(z')$ for some constant $c>0$, where for any $\xi\in \C$ and $r>0$, $\partial B_{r}(\xi)$ denotes the boundary of the disk around $\xi$ with radius $r$. We can choose $c>0$ small enough such that $\Gamma\subset \C_+$ and $ \min_{\xi\in \Gamma}\im \xi \gtrsim \eta$. Then, by Cauchy's integral formula and \eqref{aniso_law}, we get that
\be\label{keyidentity2}
\begin{split}
 \sum_{\mu\in \cal I_2}G_{\bu \mu}(z)G_{\bu' \mu}(z') &=\frac1{2\pi \ii}\int_{\Gamma} \frac{G_{\bu \bu'}(\xi)}{(\xi-z)(\xi-z')}\dd\xi  =\frac1{2\pi \ii}\int_{\Gamma} \frac{\Pi_{\bu \bu'}(\xi) + \OO_\prec((N\eta)^{-1/2})}{(\xi-z)(\xi-z')}\dd\xi \\
&=  \frac{\Pi_{\bu\bu'}(z) - \Pi_{\bu\bu'}(z')}{z-z'} + \OO_\prec\left(\eta^{-1}(N\eta)^{-1/2}\right).\
\end{split}
\ee
Applying it to \eqref{main1G}, we can write that
\be\label{main2G1}
\begin{split}
&\eta  \sum_{i\in \cal I_1,\mu \in \cal I_2}G^{(1)}_{\bu_1\mu}\bw_1(i)  G^{(1)}_{\bu_1 \mu}G^{(1)}_{\bt_i \bu_1} =\eta  \Pi_{\bu_1\bu_1}'(z_1)\left(\Pi(z_1)\wt\Lambda \Pi(z_1)\right)_{\bu_1\bu_1} +  \OO_\prec\left((N\eta)^{-1/2}\right) . 
\end{split}
\ee
Plugging \eqref{main2G2} and \eqref{main2G1} into \eqref{main1G} and \eqref{main2G}, we obtain that 
\be\label{G1333}
\begin{split}
-\sqrt{ \frac{\eta}{N}}  \mathbb E  \sum_{i\in \cal I_1,\mu \in \cal I_2}G_{\bu_1\mu}^{(1)}\bw_1(i)\frac{\partial (Y_1^{k_1-1} \overline Y_2^{k_2})}{\partial X_{i\mu}}  = 2(k_1-1)\eta \Pi_{\bu_1\bu_1}'(z_1)\left(\Pi(z_1)\wt\Lambda \Pi(z_1)\right)_{\bu_1\bu_1}   \mathbb E Y_1^{k_1-2} \overline Y_2^{k_2} \\
+2k_2\eta \left(  \frac{\Pi_{\bu_1\bu_2}(z_1) - \Pi_{\bu_1\bu_2}(\overline z_2)}{z_1-\overline z_2}\right) \left(\Pi(z_1)\wt\Lambda \overline \Pi(z_2)\right)_{\bu_1\bu_2}   \mathbb E     Y_1^{k_1-1}  \overline Y_2^{k_2-1}  + \OO_\prec\left( (N\eta)^{-1/2}\right).
\end{split}
\ee

In sum, combining \eqref{G1_1+2}--\eqref{G1222} and \eqref{G1333}, we obtain that 
\be\label{finalG1}
\begin{split}
&\cal M_1 + \mathfrak G_1 = 2(k_1-1)\eta \Pi_{\bu_1\bu_1}'(z_1)\left(\Pi(z_1)\wt\Lambda \Pi(z_1)\right)_{\bu_1\bu_1}   \mathbb E Y_1^{k_1-2} \overline Y_2^{k_2} \\
&\qquad \qquad +2k_2\eta \left(  \frac{\Pi_{\bu_1\bu_2}(z_1) - \Pi_{\bu_1\bu_2}(\overline z_2)}{z_1-\overline z_2}\right) \left(\Pi(z_1)\wt\Lambda \overline \Pi(z_2)\right)_{\bu_1\bu_2}   \mathbb E     Y_1^{k_1-1}  \overline Y_2^{k_2-1}  + \OO_\prec\left( (N\eta)^{-1/2}\right)\\
&=(k_1-1)z_1^2\eta \wh \beta(z_1, z_1, \bv_1,\bv_1) \mathbb E Y_1^{k_1-2} \overline Y_2^{k_2}  + k_2 z_1\overline z_2\eta \wh \beta(z_1, \overline z_2, \bv_1,\bv_2) \mathbb E     Y_1^{k_1-1}  \overline Y_2^{k_2-1} + \OO_\prec\left(  (N\eta)^{-\frac12}\right) , 
\end{split}
\ee
where we used \eqref{defn_pi} to rewrite the coefficients into \eqref{defbeta2} and recall that $\bv_i=O^\top \bu_i$.

\subsection{The error term $\mathfrak G_2$}  

For the term
\begin{align*}
\mathfrak G_2 :=-\frac{\sqrt{\eta} }{2 N}  \sum_{i\in \cal I_1,\mu \in \cal I_2}  \bw_1(i) \mathcal \kappa_{3}(i,\mu)\mathbb E \frac{\partial^2  (G_{\bu_1\mu}^{(1)} Y_1^{k_1-1} \overline Y_2^{k_2} )}{\partial (X_{i \mu })^2 },
\end{align*}
we consider the following cases. We first assume that the two derivatives act on $G^{(1)}_{\bu_1\mu} $:  
\begin{align*}
\frac{\partial^2  G^{(1)}_{\bu_1\mu}  }{\partial  X_{i \mu }^2 }= 4 G^{(1)}_{\bu_1\bt_i}G^{(1)}_{\mu\bt_i}G^{(1)}_{\mu\mu}+2 G^{(1)}_{\bu_1\mu}(G^{(1)}_{\mu\bt_i})^2+2G^{(1)}_{\bu_1 \mu}G^{(1)}_{\bt_i\bt_i}G^{(1)}_{\mu\mu}.
\end{align*}
Inserting these three terms into $\mathfrak G_2 $, we can bound the resulting expressions as follows. First, we have
\begin{align}\label{frakG21}
&  2\frac{\sqrt{\eta}}{N}   \sum_{i\in \cal I_1,\mu \in \cal I_2}\left| \bw_1(i)  G^{(1)}_{\bu_1\bt_i}G^{(1)}_{\mu\bt_i}G^{(1)}_{\mu\mu} \right|\prec \frac1{N^{3/2}}   \sum_{i\in \cal I_1,\mu \in \cal I_2} | \bw_1(i)|  |G^{(1)}_{\bu\bt_i}|\prec \frac{1}{\sqrt{N\eta}},
\end{align}
where we used \eqref{aniso_law} in the first step to bound $G^{(1)}_{\mu\bt_i} \prec (N\eta)^{-1/2}$, and in the the second step we used Lemma \ref{lem_comp_gbound} and \eqref{aniso_law} to bound that
\be\label{wardapp1}  \sum_{i\in \cal I_1} |\bw_1(i)|  |G^{(1)}_{\bu\bt_i}| \le \Big(\sum_{i\in \cal I_1} |\bw_1(i)|^2 \Big)^{1/2} \Big(\sum_{i\in \cal I_1} |G^{(1)}_{\bu\bt_i}|^2\Big)^{1/2} \prec \eta^{-1/2}.\ee
Similarly, we can bound that
\begin{align}\label{frakG22}
&  \frac{\sqrt{\eta}}{N}     \sum_{i\in \cal I_1,\mu \in \cal I_2}\left| \bw_1(i)G^{(1)}_{\bu_1\mu}(G^{(1)}_{\mu\bt_i})^2\right|\prec \frac1{N^{5/2}\eta}   \sum_{i\in \cal I_1,\mu \in \cal I_2} | \bw_1(i)| \prec \frac{1}{N\eta}.
\end{align}
Finally,  we have
\be\label{frakG23}
\begin{split}
&\quad -\frac{\sqrt{\eta} }{N}  \sum_{i\in \cal I_1,\mu \in \cal I_2}  \bw_1(i) \mathcal \kappa_{3}(i,\mu)G^{(1)}_{\bu_1 \mu}G^{(1)}_{\bt_i\bt_i}G^{(1)}_{\mu\mu}\\
& =-\frac{\sqrt{\eta} }{N}  \sum_{i\in \cal I_1,\mu \in \cal I_2}  \bw_1(i) \mathcal \kappa_{3}(i,\mu)G^{(1)}_{\bu_1 \mu}\Pi_{\bt_i\bt_i}(z_1)\Pi_{\mu\mu}(z_1)  +\OO_\prec\left( \frac{1}{N^{2}\sqrt{\eta}}  \sum_{i\in \cal I_1,\mu \in \cal I_2} | \bw_1(i)|  \right) \\
& \prec  \frac{\sqrt{\eta} }{N}  \sum_{i\in \cal I_1}  |\bw_1(i)| \eta^{-1/2}+ \frac1{\sqrt{N\eta}} \prec \frac1{\sqrt{N\eta}},
\end{split}
\ee
where in the second step we applied \eqref{aniso_law} to $G^{(1)}$ to get that
\be
\sum_{\mu \in \cal I_2} \mathcal \kappa_{3}(i,\mu)G^{(1)}_{\bu_1 \mu} \Pi_{\mu\mu}(z_1)=G^{(1)}_{\bu_1 \wt\bw_i} \prec \frac{\sqrt{N}}{\sqrt{N\eta}} = \eta^{-1/2}.\label{harderaniso}  
\ee
Here, we have used the fact that $\wt\bw_i :=\sum_\mu \kappa_{3}(i,\mu)\Pi_{\mu\mu}(z_1)\mathbf e_\mu$ has $l^2$-norm $\OO(\sqrt{N})$. 

Next, we consider the case that one derivative acts on $G^{(1)}_{\bu_1\mu} $ and the other acts on $ Y_1^{k_1-1}\overline Y_2^{k_2} $. Suppose the other derivative acts on a $Y_1$ factor, then we need to estimate
\begin{align*}
& -\frac{\eta}{\sqrt{N}} \sum_{i\in \cal I_1,\mu \in \cal I_2}  \bw_1(i) \kappa_3(i,\mu) \left(G^{(1)}_{\bu_1 \mu}G^{(1)}_{\bt_i \mu}+G^{(1)}_{\bu_1 \bt_i}G^{(1)}_{\mu \mu}\right) G^{(1)}_{\bu_1 \bt_i}G^{(1)}_{\bu_1 \mu}  .
\end{align*}
For the first term, we can bound it using \eqref{aniso_law} as 
\begin{align}\label{frakG24}
& \frac{\eta}{\sqrt{N}} \sum_{i\in \cal I_1,\mu \in \cal I_2}  |\bw_1(i)||G^{(1)}_{\bu_1 \mu}G^{(1)}_{\bt_i \mu} G^{(1)}_{\bu_1 \bt_i}G^{(1)}_{\bu_1 \mu}| \prec \frac{1}{N^2\sqrt{\eta}}\sum_{i\in \cal I_1,\mu \in \cal I_2}|\bw_1(i)| \prec \frac{1}{\sqrt{N\eta}}.
\end{align}
For the second term, we can apply similar argument as in \eqref{frakG23} to get that
\be\label{frakG25}
\begin{split}
&\quad -\frac{\eta}{\sqrt{N}} \sum_{i\in \cal I_1,\mu \in \cal I_2}  \bw_1(i) \kappa_3(i,\mu) G^{(1)}_{\bu_1 \bt_i}G^{(1)}_{\mu \mu} G^{(1)}_{\bu_1 \bt_i}G^{(1)}_{\bu_1 \mu}  \\
&=   -\frac{\eta}{\sqrt{N}} \sum_{i\in \cal I_1,\mu \in \cal I_2}  \bw_1(i) \kappa_3(i,\mu) \Pi_{\mu \mu}(z_1) (G^{(1)}_{\bu_1 \bt_i})^2 G^{(1)}_{\bu_1 \mu}  +\OO_\prec \left(\frac{1}{N^{3/2}} \sum_{i\in \cal I_1,\mu \in \cal I_2} | \bw_1(i)| |G^{(1)}_{\bu_1 \bt_i}|  \right) \\
& \prec \frac{1}{\sqrt{N}} \sum_{i\in \cal I_1} |\bw_1(i)||G^{(1)}_{\bu_1 \bt_i}| \prec (N\eta)^{-1/2},
\end{split}
\ee
where in the second step we applied (\ref{harderaniso}) to the first term, and in the last step we used \eqref{wardapp1}. If the other derivative acts on a $\overline Y_2$ factor, then we have similar estimates. 

Finally, we consider the case that there are two derivatives acting on $ Y^{k_1-1}\overline Y^{k_2}$. 
\vspace{5pt}

\noindent {\bf Case 1:} Suppose that the two derivatives act on two different $Y$ factors. If they are both $Y_1$ factors, then we have
\be\label{frakG26}
\eta^{3/2}\sum_{i\in \cal I_1,\mu \in \cal I_2}  |\bw_1(i)| |G_{\bu_1\mu}^{(1)}|  |G^{(1)}_{\bu_1 \bt_i}|^2|G^{(1)}_{\bu_1 \mu}|^2  \prec \frac{N\eta^{3/2}}{(N\eta)^{3/2}}\sum_{i\in \cal I_1} |\bw_1(i)||G^{(1)}_{\bu_1 \bt_i}|  \prec \frac{1}{ \sqrt{N\eta}},
\ee
where we used \eqref{wardapp1} in the second step. We have similar estimates if the two derivatives act on two $\overline Y_2$ factors or on a $Y_1$ factor and a $\overline Y_2$ factor. 

\vspace{5pt}

\noindent {\bf Case 2:} Suppose that the two derivatives act on one single $Y$ factor. If this is a $Y_1$ factor, then we need to bound 
\begin{align*}
-\frac{\eta }{ \sqrt{N}}  \sum_{i\in \cal I_1,\mu \in \cal I_2}  \bw_1(i) \mathcal \kappa_{3}(i,\mu)G_{\bu_1\mu}^{(1)} \left( (G^{(1)}_{\bu_1 \bt_i})^2G^{(1)}_{\mu\mu}+(G^{(1)}_{\bu_1 \mu})^2G^{(1)}_{\bt_i\bt_i}+2G^{(1)}_{\bu_1 \bt_i}G^{(1)}_{\bu_1 \mu}G^{(1)}_{\bt_i \mu}\right).
\end{align*}
The first term has been estimated in \eqref{frakG25}, and the third term has been estimated in \eqref{frakG24}. For the second term, using \eqref{aniso_law}, we get that
\begin{align}\label{frakG27}
\frac{\eta }{ \sqrt{N}}  \sum_{i\in \cal I_1,\mu \in \cal I_2} | \bw_1(i)| |G_{\bu_1\mu}^{(1)}|^3 |G^{(1)}_{\bt_i\bt_i}|\prec \frac{1}{ N^2\sqrt{\eta}}  \sum_{i\in \cal I_1,\mu \in \cal I_2} | \bw_1(i)| \prec \frac{1}{ \sqrt{N\eta}}. \end{align}
If the two derivatives act on a $\overline Y_2$ factor, then we have a similar estimate.

Combining \eqref{frakG21}--\eqref{frakG23} and \eqref{frakG24}--\eqref{frakG27}, we obtain that
\be\label{finalG2}
\begin{split}
&\mathfrak G_2 \prec \frac{1}{ \sqrt{N\eta}}.
\end{split}
\ee

\subsection{Terms $\mathfrak G_k$ with $k\ge 3$}

For the terms $\mathfrak G_k$ with $k\ge 3$, the expressions begin to become rather complicated. In order to exploit the structures of them in a systematical way, we introduce the following algebraic object.

\begin{definition}[Words]\label{def_comp_words}
 Given $i\in \mathcal I_1$ and $\mu\in \mathcal I_2$, let $\sW$ be the set of words of even length in two letters $\{\mathbf i, \bm{\mu}\}$. We denote the length of a word $w\in\sW$ by $2{\bm l}(w)$ with ${\bm l}(w)\in \mathbb N$. We use bold symbols to denote the letters of words. For instance, $w=\mathbf a_1\mathbf b_2\mathbf a_2\mathbf b_3\cdots\mathbf a_r\mathbf b_{r+1}$ denotes a word of length $2r$. Let $\sW_r:=\{w\in \mathcal W: {\bm l}(w)=r\}$ be the set of words of length $2r$, and such that
each word $w\in \sW_r$ satisfies that $\mathbf a_l\mathbf b_{l+1}\in\{\mathbf i\bm{\mu},\bm{\mu}\mathbf i\}$ for all $1\le l\le r$.

Next, we assign to each letter a value $[\cdot]$ through $[\mathbf i]:=\bt_i$ and $[\bm {\mu}]:=\mathbf e_\mu$. 
It is important to distinguish the abstract letter from its value, which is a vector (or can be regarded as a summation index). To each word $w$ we assign two types of random variables $A^{(1)}_{i, \mu}(w)$ and $A^{(2)}_{i, \mu}(w)$ as follows. If ${\bm l}(w)=0$, we define
 \be\nonumber A^{(1)}_{i, \mu}(w):=G^{(1)}_{\mathbf u_1\mathbf u_1}-\Pi_{\mathbf u_1\mathbf u_1}(z_1), \quad A^{(2)}_{i, \mu}(w):=G^{(2)}_{\mathbf u_2\mathbf u_2}-\Pi_{\mathbf u_2\mathbf u_2}(z_2).\ee
 If ${\bm l}(w)\ge 1$, say $w=\mathbf a_1\mathbf b_2\mathbf a_2\mathbf b_3\cdots\mathbf a_r\mathbf b_{r+1}$, we define
 \begin{equation}\nonumber
 \begin{split}
 & A^{(1)}_{i, \mu}(w):=G^{(1)}_{\bu_1[\mathbf a_1]} G^{(1)}_{[\mathbf b_2][\mathbf a_2]}\cdots G^{(1)}_{[\mathbf b_r][\mathbf a_r]} G^{(1)}_{[\mathbf b_{r+1}]\bu_1},\quad A^{(2)}_{i, \mu}(w):=\overline G^{(2)}_{\bu_2[\mathbf a_1]} \overline G^{(2)}_{[\mathbf b_2][\mathbf a_2]}\cdots \overline G^{(2)}_{[\mathbf b_r][\mathbf a_r]} \overline G^{(2)}_{[\mathbf b_{r+1}]\bu_2}.
 \end{split}
 \end{equation}
 Finally, for $w=\mathbf a_1\mathbf b_2\mathbf a_2\mathbf b_3\cdots\mathbf a_r\mathbf b_{r+1}$, we define another type of word as
  \begin{equation}\label{eq_comp_A(W)2}
 \wt A_{i, \mu}(w):=G^{(1)}_{\bu_1[\mathbf a_1]} G^{(1)}_{[\mathbf b_2][\mathbf a_2]}\cdots G^{(1)}_{[\mathbf b_r][\mathbf a_r]} G^{(1)}_{[\mathbf b_{r+1}]\mu}.
 \end{equation}
\end{definition}

Notice these words are constructed in a way such that, by \eqref{dervG},  
\[\left(\frac{\partial}{\partial X_{i\mu}}\right)^r Y_1=(-1)^r r! \sqrt{N\eta}\sum_{w\in \mathcal W_r} A^{(1)}_{i, \mu}(w),\quad r\in \mathbb N.\]
Similarly, $A^{(2)}_{i, \mu}(w)$ is related to the derivatives of $\overline Y_2$, and $\wt A_{i, \mu}(w)$ is related to the derivatives of $G^{(1)}_{\bu_1\mu}$. Thus, we have
\be\label{dervWord}
\begin{split}
\frac{\partial^k  (G_{\bu_1\mu}^{(1)} Y_1^{k_1-1} \overline Y_2^{k_2} )}{\partial (X_{i \mu })^k } &= (-1)^k (N\eta)^{\frac12(k_1+k_2-1)}  \sum_{l_1+\cdots+l_{k_1+k_2}=k} \Big[l_1! \sum_{w_1\in \cal W_{l_1}} \wt A_{i, \mu}(w_1)\Big] \\
&\times \prod_{s=2}^{k_1} \Big[l_s! \sum_{w_s\in\sW_{l_s}} A^{(1)}_{i, \mu}(w_s)\Big]\prod_{s=k_1+1}^{k_1+k_2} \Big[l_s! \sum_{w_s\in\sW_{l_s}} A^{(2)}_{i, \mu}(w_s)\Big].
\end{split}
\ee
In the following proof, for simplicity, we shall abbreviate
 \be\nonumber A_{i, \mu}(w_s)\equiv \begin{cases}A^{(1)}_{i, \mu}(w_s), \ &\text{if } \ 2\le s\le k_1\\
 A^{(2)}_{i, \mu}(w_s), \ &\text{if }\  k_1+1\le s\le k_1+k_2
 \end{cases}.\ee
Moreover, we introduce the notations
\be\nonumber a:=  \#\{2\le s\le k_1+k_2:l_i\ge 1\},\quad a_1:= \#\{2\le s\le k_1+k_2:l_i=1\}.\ee
Without loss of generality, we assume that the words with nonzero length are $w_{s_1},\ldots, w_{s_a}$, and the words with length 1 are $w_{s_1},\ldots, w_{s_{a_1}}$. Then, we have 
\be\label{sumS}l_{s_{1}} + \cdots + l_{s_a}=k-l_1 \ \  \Rightarrow \ \ 2a\le k-l_1+ a_1.\ee
By definition, it is easy to see that 
\be\label{Rimu2} |A_{i, \mu}(w_s)| \prec R_i^2 + R_\mu^2,\quad \text{if }\ l_s\ge 1, s\ge 2, \ee
where we used the notations
\be\nonumber R_i:= |G^{(1)}_{\bu_1 \bt_i}|+|G^{(2)}_{\bu_2 \bt_i}|,\quad R_\mu:=  |G^{(1)}_{\bu_1 \mu}|+|G^{(2)}_{\bu_2 \mu}|+|G^{(1)}_{\bt_i \mu}| +|G^{(2)}_{\bt_i \mu}| \prec (N\eta)^{-1/2}.\ee
If $l_s=1$ for some $s\ge 2$, we have the better bound
\be\label{Rimu1} |A_{i, \mu}(w_s)| \prec R_i  R_\mu \prec \frac{R_i}{\sqrt{N\eta}}.\ee
Similarly, we have 
\be\label{Rimu3} |\wt A_{i, \mu}(w_1)|\prec \mathbf 1(l_1 \ge 1)R_i + R_\mu \prec \mathbf 1(l_1 \ge 1)R_i  + (N\eta)^{-1/2}. \ee
Finally, using Lemma \ref{lem_comp_gbound} and \eqref{aniso_law}, we can bound that
\be\label{Rimu4}
\sum_{i\in \cal I_1}R_i^2 + \sum_{\mu\in \cal I_2}R_\mu^2 \prec \eta^{-1}, \quad \sum_{i\in \cal I_1}|\bw_1(i)|R_i\prec \eta^{-1/2}.
\ee
We will use these bounds tacitly in the following proof.


 Now, we study the $k=3$ case using the above tools. In this case, we will obtain a leading term that depends on the fourth cumulants of the $X$ entries.

\vspace{5pt}

\noindent{\bf The leading term $\mathfrak G_3$.} We insert \eqref{dervWord} into the term
\be\nonumber \mathfrak G_3 :=-\frac{\sqrt{\eta} }{6 N^{3/2}}  \sum_{i\in \cal I_1,\mu \in \cal I_2}  \bw_1(i) \mathcal \kappa_{4}(i,\mu)\mathbb E \frac{\partial^3  (G_{\bu_1\mu}^{(1)} Y_1^{k_1-1} \overline Y_2^{k_2} )}{\partial (X_{i \mu })^3 }. \ee
Then, applying \eqref{Rimu2}--\eqref{Rimu3} to $\mathfrak G_3$, we see that it suffices to bound 
\begin{align*}
&\frac{(N\eta)^{a/2}\sqrt{\eta}}{N^{3/2}}  \mathbf 1(l_1\ge 1)\sum_{i\in \cal I_1,\mu \in \cal I_2}  |\bw_1(i)| R_i (R_iR_\mu)^{a_1} (R_i^2 + R_\mu^2)^{a-a_1}\\
+& \frac{(N\eta)^{a/2}\sqrt{\eta}}{N^{3/2}}  \sum_{i\in \cal I_1,\mu \in \cal I_2}  |\bw_1(i)|  R_\mu(R_iR_\mu)^{a_1} (R_i^2 + R_\mu^2)^{a-a_1}=:\cal K_1+\cal  K_2.
\end{align*}

For $\cal K_2$, we first consider the case $a_1=0$. Then, $a$ can only be $0$ or $1$, and we have
\begin{align*}
\cal K_2&\prec  \frac{\mathbf 1(a=0)}{N^{2}} \sum_{i\in \cal I_1,\mu \in \cal I_2}  |\bw_1(i)|  + \mathbf 1(a= 1) \frac{(N\eta)^{1/2} }{N^{2}} \sum_{i\in \cal I_1,\mu \in \cal I_2}  |\bw_1(i)| \left(R_i^2 + R_\mu^2\right)\\
&\prec  N^{-1/2}  +   \frac{(N\eta)^{1/2} }{N^{2}} \left(\frac{N}{\sqrt{\eta}}+\frac{\sqrt{N}}{\eta}\right) \prec N^{-1/2}.
\end{align*}
Then, in the $a_1=1$ case, $a$ can only be $1$ or $2$, and we have
\begin{align*}
\cal K_2 &\prec \mathbf 1(a=1) \frac{(N\eta)^{1/2} }{N^{2}}  \sum_{i\in \cal I_1,\mu \in \cal I_2}  | \bw_1(i)|  R_iR_\mu  + \mathbf 1(a= 2)\frac{N\eta }{N^{2}}  \sum_{i\in \cal I_1,\mu \in \cal I_2}  | \bw_1(i)|  (R_iR_\mu) (R_i^2 + R_\mu^2) \\
&\prec  \frac{(N\eta)^{1/2} }{N^{2}} \frac{\sqrt{N}}{\eta}  +  \frac{N\eta }{N^{2}}\left( \frac{\sqrt{N}}{\eta}+ \frac{1}{\sqrt{N}\eta^2 }\right)   \prec  N^{-1/2}.
\end{align*}
Finally, for the $a_1\ge 2$ case, we have $a=a_1$ and 
\begin{align*}
&\cal K_2\prec \frac{(N\eta)^{a_1/2} }{N^{2} }  \sum_{i\in \cal I_1,\mu \in \cal I_2}  |\bw_1(i)|  (R_iR_\mu)^{a_1}\prec  \frac{1}{N^{2} }  \sum_{i\in \cal I_1,\mu \in \cal I_2}  |\bw_1(i)|  R_i \prec \frac{1}{N\sqrt{\eta}} \le N^{-1/2}.
\end{align*}

Next, we estimate $\cal K_1$. If $a_1=0$ and $l_1\ge 2$, then $a$ can only be $0$, and we have that
\begin{align*}
\cal K_1&\prec \frac{\sqrt{\eta}}{N^{3/2}}  \sum_{i\in \cal I_1,\mu \in \cal I_2}  |\bw_1(i)| R_i \prec N^{-1/2} .\end{align*}
If $a_1=1$ and $l_1\ge 1$, then $a$ can only be $1$, and we have that
\begin{align*}
\cal K_1 &\prec \frac{(N\eta)^{1/2}\sqrt{\eta}}{N^{3/2}}  \sum_{i\in \cal I_1,\mu \in \cal I_2}  |\bw_1(i)| R_i (R_iR_\mu) \prec \frac{(N\eta)^{1/2}\sqrt{\eta}}{N^{3/2}}  \frac{\sqrt{N}}{\eta} =  N^{-1/2}.
\end{align*}
If $a_1\ge 2$ and $l_1\ge 1$, then $a=a_1=2$, and we have that
\begin{align*}
&\cal K_1\prec \frac{(N\eta)\sqrt{\eta}}{N^{3/2}} \sum_{i\in \cal I_1,\mu \in \cal I_2}  |\bw_1(i)| R_i (R_iR_\mu)^{2} \prec \frac{(N\eta)\sqrt{\eta}}{N^{3/2}}\frac1{\eta^{3/2}} = N^{-1/2}. 
\end{align*}

Finally, we are left with the case $a_1=0$ and $l_1=1$, which will provide a leading term. In this case, we have that one derivative acts on $G^{(1)}_{\bu_1\mu}$ and two other derivatives act on a $Y_1$ or $\overline Y_2$ factor, i.e., 
\be
\begin{split}\label{G3temp}
\mathfrak G_3 &=(k_1-1)\frac{\sqrt{\eta} }{2 N^{3/2}}  \sum_{i\in \cal I_1,\mu \in \cal I_2}  \bw_1(i) \mathcal \kappa_{4}(i,\mu) \mathbb E\left( G_{\bu_1\bt_i}^{(1)}G_{\mu\mu}^{(1)} +G_{\bu_1\mu}^{(1)}G_{\bt_i\mu}^{(1)} \right) \frac{\partial^2  Y_1}{\partial (X_{i \mu })^2 }Y_1^{k_1-2} \overline Y_2^{k_2} \\
&+k_2\frac{\sqrt{\eta} }{2 N^{3/2}}  \sum_{i\in \cal I_1,\mu \in \cal I_2}  \bw_1(i) \mathcal \kappa_{4}(i,\mu) \mathbb E\left( G_{\bu_1\bt_i}^{(1)}G_{\mu\mu}^{(1)} +G_{\bu_1\mu}^{(1)}G_{\bt_i\mu}^{(1)} \right) \frac{\partial^2  \overline Y_2}{\partial (X_{i \mu })^2 }Y_1^{k_1-1} \overline Y_2^{k_2-1} + \OO_\prec(N^{-1/2}) \\
&=(k_1-1)\frac{\sqrt{\eta} }{2 N^{3/2}}  \sum_{i\in \cal I_1,\mu \in \cal I_2}  \bw_1(i) \mathcal \kappa_{4}(i,\mu) \mathbb E G_{\bu_1\bt_i}^{(1)}G_{\mu\mu}^{(1)}  \frac{\partial^2  Y_1}{\partial (X_{i \mu })^2 }Y_1^{k_1-2} \overline Y_2^{k_2} \\
&+k_2\frac{\sqrt{\eta} }{2 N^{3/2}}  \sum_{i\in \cal I_1,\mu \in \cal I_2}  \bw_1(i) \mathcal \kappa_{4}(i,\mu) \mathbb E G_{\bu_1\bt_i}^{(1)}G_{\mu\mu}^{(1)}  \frac{\partial^2  \overline Y_2}{\partial (X_{i \mu })^2 }Y_1^{k_1-1} \overline Y_2^{k_2-1} + \OO_\prec(N^{-1/2}),
\end{split}
\ee
where in the second step we used that the $G_{\bu_1\mu}^{(1)}G_{\bt_i\mu}^{(1)}$ terms have been bounded as $\cal K_2$ in the above proof. We now calculate the first term on the right-hand side of \eqref{G3temp}, which takes the form 
\begin{align*}
& (k_1-1)\frac{\eta }{ N}  \mathbb E \sum_{i\in \cal I_1,\mu \in \cal I_2}  \bw_1(i) \mathcal \kappa_{4}(i,\mu)G_{\bu_1\bt_i}^{(1)}G_{\mu\mu}^{(1)} \left( (G^{(1)}_{\bu_1 \bt_i})^2G^{(1)}_{\mu\mu}+(G^{(1)}_{\bu_1 \mu})^2G^{(1)}_{\bt_i\bt_i}+2G^{(1)}_{\bu_1 \bt_i}G^{(1)}_{\bu_1 \mu}G^{(1)}_{\bt_i \mu}\right)Y_1^{k_1-2} \overline Y_2^{k_2}\\
& =:\mathbb E\cal K_1 + \mathbb E\cal K_2 + \mathbb E\cal K_3,
\end{align*}
where we have slightly abused the notations $\cal K_{1}$ and $\cal K_{2}$. We can bound that
\be\nonumber \cal K_2 + \cal K_3 \prec \frac{\eta}{N}\sum_{i\in \cal I_1,\mu \in \cal I_2}|\bw_1(i)| R_i R_\mu^{2} \prec \frac{\eta}{N} \frac{1}{\eta^{3/2}} \le N^{-1/2} .\ee
For $\cal K_1$, we have that
\begin{align}
\mathbb E\cal K_1 &= (k_1-1)\frac{\eta}{N}  \sum_{i\in \cal I_1,\mu \in \cal I_2}  \bw_1(i) \mathcal \kappa_{4}(i,\mu)\left[\Pi_{\mu\mu}(z_1)\right]^2  [\Pi_{\bu_1 \bt_i}(z_1)]^3 \mathbb E  Y_1^{k_1-2} \overline Y_2^{k_2} \nonumber\\
& + \OO_\prec\left(\frac{\sqrt{\eta} }{N^{3/2}}  \mathbb E \sum_{i\in \cal I_1,\mu \in \cal I_2}  |\bw_1(i)|(R_i + |\Pi_{\bu_1 \bt_i}(z_1)|) \right) \nonumber\\
&= (k_1-1)\frac{\eta m_{2c}^2(z_1)}{N}  \sum_{i\in \cal I_1,\mu \in \cal I_2}   \mathcal \kappa_{4}(i,\mu)  [\Pi_{\bu_1 \bt_i}(z_1)]^4 \mathbb E  Y_1^{k_1-2} \overline Y_2^{k_2} +\OO_\prec\left(N^{-1/2} \right),\nonumber
\end{align}
where we used that $\bw_1(i)=\Pi_{\bu_1 \bt_i}(z_1)$ by the definition of $\bw_1$. We have a similar estimate for the second term on the right-hand side of \eqref{G3temp}. In sum, we obtain that  
\begin{align}
\mathfrak G_3 &= (k_1-1)\frac{\eta m_{2c}^2(z_1)}{N}  \sum_{i\in \cal I_1,\mu \in \cal I_2}   \mathcal \kappa_{4}(i,\mu)  [\Pi_{\bu_1 \bt_i}(z_1)]^4 \mathbb E  Y_1^{k_1-2} \overline Y_2^{k_2} \label{finalG3}\\
&+k_2\frac{\eta m_{2c}(z_1)\overline m_{2c}(z_2)}{N}  \sum_{i\in \cal I_1,\mu \in \cal I_2}   \mathcal \kappa_{4}(i,\mu)  [\Pi_{\bu_1 \bt_i}(z_1)]^2 [\overline \Pi_{\bu_2 \bt_i}(z_2)]^2  \mathbb E  Y_1^{k_1-1} \overline Y_2^{k_2-1} + \OO_\prec(N^{-1/2})\nonumber\\
&= (k_1-1)z_1^2 \eta \wh \al(z_1, z_1, \bv_1,\bv_1)  \mathbb E     Y_1^{k_1-2}  \overline Y_2^{k_2} + k_2z_1\overline z_2\eta \wh \al(z_1, \overline z_2, \bv_1,\bv_2)  \mathbb E     Y_1^{k_1-1}  \overline Y_2^{k_2-1} + \OO_\prec(N^{-1/2}),\nonumber
\end{align}
where we used \eqref{defn_pi} to rewrite the coefficients with \eqref{defal2}.

\vspace{5pt}

Next, we deal with cases with $k\ge 4$, which only contain error terms.

\vspace{5pt}

\noindent{\bf The error terms $\mathfrak G_k$, $k\ge 4$.} The terms $\mathfrak G_k$, $k\ge 4$, can be estimated in similar ways as $\mathfrak G_3$. We insert \eqref{dervWord} into \eqref{G123}, and apply \eqref{Rimu2}--\eqref{Rimu3} to get that
\begin{align*}
\mathfrak G_k& \prec \frac{(N\eta)^{a/2}\sqrt{\eta}}{N^{k/2}}  \mathbf 1(l_1\ge 1)\sum_{i\in \cal I_1,\mu \in \cal I_2}  |\bw_1(i)| R_i (R_iR_\mu)^{a_1} (R_i^2 + R_\mu^2)^{a-a_1}\\
&+ \frac{(N\eta)^{a/2}\sqrt{\eta}}{N^{k/2}}  \sum_{i\in \cal I_1,\mu \in \cal I_2}  |\bw_1(i)|  R_\mu(R_iR_\mu)^{a_1} (R_i^2 + R_\mu^2)^{a-a_1}=:\cal K_1+\cal  K_2.
\end{align*}
For the term $\cal K_1$, 
we have
\begin{align*}
\cal K_1&\prec  \mathbf 1(l_1\ge 1)\frac{(N\eta)^{(a-a_1)/2}\sqrt{\eta}}{N^{k/2}}  \sum_{i\in \cal I_1,\mu \in \cal I_2}  | \bw_1(i)| R_i   \prec \mathbf 1(l_1\ge 1)\frac{ (N\eta)^{(a-a_1)/2}\sqrt{\eta}}{N^{k/2}} \frac{N}{\sqrt{\eta}} \le N^{-(k-a+a_1)/{2}+1},
\end{align*}
where in the second step we used \eqref{Rimu4}. With \eqref{sumS}, we obtain that
\be\label{katemp}\frac{k-a+a_1}{2}-1\ge  \frac{1}{2}\left(k+a_1-\frac{k-l_1+a_1}{2}\right)-1\ge \frac12 \ \ \Rightarrow \ \ \cal K_1\prec N^{-1/2},\ee
if $k+a_1+l_1\ge 6$. It remains to consider the case $k=4$, $l_1=1$ and $a_1=0$. In this case, $a$ can only be $1$ and we still have 
\be\nonumber \frac{k-a+a_1}{2}-1= \frac{1}{2} \ \ \Rightarrow \ \ \cal K_1\prec N^{-1/2}.\ee

Then, we bound $\cal K_2$. If $a_1=0$, we have
\begin{align*}
\cal K_2 &\prec  \frac{(N\eta)^{a/2} }{N^{(k+1)/2}}  \sum_{i\in \cal I_1,\mu \in \cal I_2}  | \bw_1(i)|  (R_i^2 + R_\mu^2)^{a} \\
&\prec \frac{\mathbf 1(a=0) }{N^{(k+1)/2}} \sum_{i\in \cal I_1,\mu \in \cal I_2}  | \bw_1(i)|  + \mathbf 1(a\ge 1) \frac{(N\eta)^{a/2} }{N^{(k+1)/2}} \sum_{i\in \cal I_1,\mu \in \cal I_2}  | \bw_1(i)| (R_i^2 + R_\mu^2)\\
&\prec \frac{\mathbf 1(a=0) }{N^{k/2-1}}    + \mathbf 1(a\ge 1)\frac{(N\eta)^{a/2} }{N^{(k+1)/2}} \left(\frac{N}{\sqrt{\eta}} +\frac{\sqrt N}{\eta}\right)\prec \frac{1 }{N}   + \frac{\mathbf 1(a\ge 1) }{N^{(k-a-1)/2}}  \prec N^{-1/2},
\end{align*}
where we used \eqref{Rimu4} in the third step, $k\ge 4$ in the fourth step, and a similar estimate as in \eqref{katemp} in the last step:
\be\nonumber  \frac{k-a-1}2 \ge \frac{k+l_1}{4}-\frac12\ge \frac12.\ee
If $a_1\ge 1$, we have
\begin{align*}
&\cal K_2\prec \frac{(N\eta)^{a/2} }{N^{(k+1)/2}} \frac{1}{(N\eta)^{a_1/2}}  \sum_{i\in \cal I_1,\mu \in \cal I_2}  | \bw_1(i)|  R_i  \prec \frac{(N\eta)^{(a-a_1)/2} }{N^{(k+1)/2}} \frac{N}{\sqrt{\eta}} \prec N^{-1/2},
\end{align*}
where in the last step we used \eqref{sumS} to get that
\be\nonumber
\frac{(N\eta)^{(a-a_1)/2} }{N^{(k+1)/2}} \frac{N}{\sqrt{\eta}} \prec \begin{cases} N^{-(k-1)/2}\eta^{-1/2} \le N^{-(k-2)/2}\le N^{-1}, & \text{ if }a=a_1\\
N^{-(k+a_1-a-1)/2} \le N^{-(k+a_1-2)/4}\le N^{-1/2}, & \text{ if }a>a_1
\end{cases}.
\ee

In sum, we obtain that
\be\label{finalG4}
\mathfrak G_k\prec N^{-1/2},\quad k\ge 4.
\ee

\subsection{The error term $\mathcal E$} 

Finally, we show that the term $\cal E$ in \eqref{GE} is sufficiently small as long as $l$ is large enough. We first bound 
\be\nonumber \cal K_1:=\sqrt{N\eta}  \sum_{i\in \cal I_1,\mu \in \cal I_2}  |\bw_1(i)| \mathbb E\left| X_{i\mu}\right|^{l+2}\cdot \mathbb E\sup_{|x|\le N^{\e-1/2}}\left| \partial_{i\mu}^{l+1}  f_{i\mu} (H^{(i\mu)}+x\Delta_{i\mu})\right|.\ee
We claim that for any deterministic unit vectors $\bu,\bv\in \R^{\cal I}$,
\be\label{supbdd}
\sup_{|x|\le N^{\e-1/2}}\left( |G^{(1)}_{\bu\bv} (H^{(i\mu)}+x\Delta_{i\mu}) | + |G^{(2)}_{\bu\bv} (H^{(i\mu)}+x\Delta_{i\mu}) |\right) =\OO(1)
\ee
with high probability. In fact, for $z\in \{z_1,z_2\}$ and $|x|\le N^{\e-1/2}$, we have the following resolvent expansion by \eqref{eq_comp_expansion}:
\begin{equation}
G(H^{(i\mu)}+x\Delta_{i\mu}) = G(z) - (x-X_{i\mu}) G (z)\Delta_{i\mu} G(z) +(x-X_{i\mu})^{2} G(H^{(i\mu)}+x\Delta_{i\mu})\left(\Delta_{i\mu} G(z)\right)^{2}.\nonumber
\end{equation}
Using  $|X_{i\mu}|\prec N^{-1/2}$, $|x|\le N^{\e-1/2}$, \eqref{aniso_law} for $G(z)$, and the rough bound \eqref{roughbound} for $G(H^{(i\mu)}+x\Delta_{i\mu})$, we obtain from the above expansion that
\begin{equation}\nonumber
\begin{split}
G_{\mathbf u\mathbf v}(H^{(i\mu)}+x\Delta_{i\mu}) &\prec 1+\eta^{-1}N^{-(1-2\e)} \le 2,
\end{split}
\end{equation}
as long as $\e$ is small enough such that $2\e<c_1$ (recall that $\eta\ge N^{-1+c_1}$). This implies \eqref{supbdd}. With \eqref{supbdd} and \eqref{high_moments}, we can bound $\left| \partial_{i\mu}^{l+1}  f_{i\mu} (H^{(i\mu)}+x\Delta_{i\mu})\right|\prec (N\eta)^{(k_1+k_2-1)/2}$ and
\be\nonumber \cal K_1\prec (N\eta)^{(k_1+k_2)/2}N^{3/2}N^{-(l+2)/2} \le N^{-1/2}\ee
as long as $l\ge k_1+k_2+2$. 

Now, fix an $l\ge k_1+k_2+2$, we bound the term
\be\nonumber \cal K_2:=\sqrt{N\eta}  \sum_{i\in \cal I_1,\mu \in \cal I_2}  |\bw_1(i)| \mathbb E\left| X_{i\mu}^{l+2} \mathbf 1_{|X_{i\mu}|>N^{\e-1/2}}\right|\cdot \left\|\partial^{l+1}_{i\mu}  f_{i\mu} \right\|_{\infty} .\ee
Recall that the derivatives take the form \eqref{dervWord}. Then, using \eqref{roughbound}, we can obtain that
\be\nonumber \left\|\partial^{l+1}_{i\mu}  f_{i\mu} \right\|_{\infty} \lesssim (N\eta)^{(k_1+k_2-1)/2}\eta^{-(k_1+k_2+l+1)} .\ee
On the other hand, by \eqref{high_moments}, we have $ \mathbb E\big| X_{i\mu}^{l+2} \mathbf 1_{|X_{i\mu}|>N^{\e-1/2}}\big| \le N^{-D}$ for any fixed constant $D>0$. Hence, we have 
\be\nonumber \cal K_2\prec (N\eta)^{(k_1+k_2)/2}\eta^{-(k_1+k_2+l+1)} N^{3/2} N^{-D} \prec  N^{-1/2}\ee
as long as $D$ is taken large enough. 


In sum, we obtain that
\be\label{finalG5}
\mathcal E\prec N^{-1/2}.\ee 


Combining the estimates \eqref{finalG1}, \eqref{finalG2}, \eqref{finalG3}, \eqref{finalG4} and \eqref{finalG5}, we conclude that 
\be
\begin{split} \label{inducoriginal}
\mathbb EY_1^{k_1} \overline Y^{k_2}_2 & =(k_1-1) z_1^2 \eta\gamma(z_1, z_1, \bv_1,\bv_1)  \mathbb E     Y_1^{k_1-2}  \overline Y_2^{k_2} \\
&+ k_2   z_1\overline z_2 \eta \gamma(z_1, \overline z_2, \bv_1,\bv_2)  \mathbb E     Y_1^{k_1-1}  \overline Y_2^{k_2-1}  + \OO_\prec\left( (N\eta)^{-1/2}\right)  .
\end{split}
\ee
As a special case, if $k_1=1$ and $k_2=0$, we obtain that 
\be\label{meanzeroY} \mathbb EY_1 \prec (N\eta)^{-1/2} , \ee
which verifies the mean zero condition in Proposition \ref{main_prop}. Finally, applying the induction relation \eqref{inducoriginal} repeatedly and using \eqref{meanzeroY}, we can conclude \eqref{moments part} for the expression in \eqref{k1k2}. 

We can extend the above proof to the general expression on the left-hand side of \eqref{moments part}. 

\begin{proof}[Proof of Lemma \ref{lem moment}]
We calculate $\mathbb E\left[ Y(\bu_{1},w_{1})\cdots Y(\bu_{k},w_{k})\right]$ using the cumulant expansion formula as in \eqref{EG0} and \eqref{cumulant1}. All the leading terms and error terms can be estimated in exactly the same way. For example, if we expand $Y(\bu_{1},w_{1})$ as in \eqref{EG0}, we can obtain that
\be
\begin{split} \label{induct}
\mathbb E\left[ Y(\bu_{1},w_{1})\cdots Y(\bu_{k},w_{k})\right] & =\sum_{s=2}^k \eta\gamma(z_{1},z_{s},\bv_{1},\bv_{s}) \mathbb E \prod_{t\notin \{1,s\}}Y(\bu_{t},w_{t})+ \OO_\prec\left( (N\eta)^{-1/2}\right)  .
\end{split}
\ee
Using this induction relation and \eqref{meanzeroY}, we can conclude \eqref{moments part}.

The proof of \eqref{moments partout} is similar and we only explain the key differences. First, the local laws \eqref{aniso_law} and \eqref{aver_law} can be replaced with the stronger ones \eqref{aniso_lawout} and \eqref{aver_lawout}. Moreover, by the eigenvalue rigidity estimate \eqref{rigidity2}, we have $\|G(z)\|=\OO(1)$ with high probability for $z\in \mathbf D_{out}$. Thus, for all the estimates that used the Ward's identities in Lemma \ref{lem_comp_gbound}, we can replace them with a simpler bound: for any deterministic unit vector $\bu\in \R^{\cal I}$,
\be\label{simpleGG}
\sum_{\fa\in \cal I}|G_{\mathbf u\fa}|^2 = (\overline GG)_{\mathbf u\mathbf u} =\OO(1)\quad \text{with high probability}.
\ee
Finally, in calculating the moments, we need a rough bound 
\be\label{roughhhh}
\E \left| \sqrt{N}\langle \bu,(G(z)-\Pi(z))\bu\rangle\right|^k \prec 1,
\ee
for any fixed $k\in \N$ and deterministic unit vector $\bu \in \R^{\cal I_1}$. For $z \in \mathbf D_{out}$ with $\im z\ge N^{-C}$, this follows from \eqref{aniso_lawout} and Lemma \ref{lem_stodomin} (iii), where the second moment bound on \smash{$\left| \sqrt{N}\langle \bu,(G(z)-\Pi(z))\bu\rangle\right|^k$} follows from the trivial bound \eqref{roughbound}. (This is the only place where we need the condition $\eta\ge N^{-C}$.) Now, plugging \eqref{aniso_lawout}, \eqref{aver_lawout}, \eqref{simpleGG} and \eqref{roughhhh} into the arguments between \eqref{EG0} and \eqref{finalG5}, we can conclude \eqref{moments partout}. 
\end{proof}

 \section{CLT for general functions}\label{sec generalf}

In this section, we prove the following weaker version of Theorem \ref{main_thm} and Theorem \ref{main_thm2} under \eqref{high_moments}. 

 \begin{proposition}\label{main_prop2}
 Theorems \ref{main_thm} and \ref{main_thm2} hold under the moment assumption \eqref{high_moments}.
\end{proposition} 

As for Proposition \ref{main_prop}, our proof of Proposition \ref{main_prop2} is also based on a moment calculation. More precisely, we will prove the following counterpart of Lemma \ref{lem moment}.
 
 \begin{lemma}\label{lem moment2}
Suppose $d_N$, $X$ and $\Sigma$ satisfy Assumption \ref{main_assm}, $N^{-1+c_1}\le \eta \le 1$, and \eqref{high_moments} holds. Fix any $E>0$, $k\in \N$ and constants $a,b>0$. Then, for any deterministic unit vectors $\bv_1,\ldots, \bv_k\in \R^n$ and functions $f_1,\ldots, f_k \in \cal C^{1,a,b}(\R_+)$, we have 
\begin{align}\label{moments partf}
&\mathbb E\left[ \prod_{s=1}^kZ_{\eta,E}(\bv_{s}, f_{s})\right]=\begin{cases} 
\sum \prod \varpi(f_{s},f_{t},\bv_{s},\bv_{t})+ \OO_\prec\left( N^{-c}\right), \ &\text{if} \ \ l\in 2\mathbb N \\ 
 \OO_\prec\left( N^{-c}\right), \ &\text{otherwise}
\end{cases},
\end{align}
for some constant $c>0$, where $ \varpi(f_i,f_j, \bv_i,\bv_j)\equiv  \varpi^{(N)}(f_i,f_j, \bv_i,\bv_j)$ is defined as
\begin{align*}
  \varpi(f_i,f_j, \bv_i,\bv_j) &:=  \frac{\eta}{\pi^2}\iint_{x_1, x_2} f_i\left( x_1\right)f_j\left(x_2\right)\al(E+x_1\eta, E+x_2\eta, \bv_i,\bv_j) \dd x_1  \dd x_2  \\
&+\frac{1}{\pi^2}PV\iint_{x_1, x_2 } \frac{f_i\left( x_1\right)f_j\left(x_2\right)}{x_1-x_2} \beta(E+x_1\eta, E+x_2\eta, \bv_i,\bv_j)  \dd x_1  \dd x_2 \\
&+2\int f_i\left(x\right)f_j(x) \frac{\rho_{2c}(E+x\eta) }{(E+x\eta)^2} \left(\bv_i^\top \frac{\Sigma}{(1+m_{2c}(E+x\eta)\Sigma)(1+\overline m_{2c}(E+x\eta)\Sigma)}\bv_j\right)^2\dd x ,
\end{align*}
and $\sum\prod$ means summing over all distinct ways of partitions of indices.
\end{lemma}

\begin{proof}[Proof of Proposition \ref{main_prop2}]
By Wick's theorem, \eqref{moments partf} with $E=0$ and $\eta=1$ shows that the convergence in Theorem \ref{main_thm} holds in the sense of moments, which further implies the weak convergence. The reader may be worried that in Theorem \ref{main_thm}, $E$ is taken to be 0, which does not satisfy the setting in Lemma \ref{lem moment2}. However, this is not an issue, because $\supp(f_i)\subset \R_+$, i.e., there exists a constant $c>0$ such that $f_i(x) \equiv 0 $ for all $1\le i \le k$ and $0\le x\le c$. Hence, we can take $E=c/2$ and apply Lemma \ref{lem moment2} with $\eta=1$ to the functions $g_i (x) \in \cal C^{1,a,b}(\R_+)$ defined through $g_i(x)= f_i(x+E)$.

Under the setting of Theorem \ref{main_thm2}, by Wick's theorem, \eqref{moments partf} shows that the random vector $(Z_{\eta,E}(\bv_i, f_i))_{1\le i \le k}$ converges weakly to a Gaussian vector. Moreover, the covariance function can be simplified if we take $\eta=\oo(1)$ in $\varpi(f_i,f_j, \bv_i,\bv_j)$ and use \eqref{Immc}--\eqref{Piii}:
\begin{align*}
 \varpi(f_i,f_j, \bv_i,\bv_j)&= \frac{1}{\pi^2}PV\iint_{x_1, x_2 } \frac{f_i\left( x_1\right)f_j\left(x_2\right)}{x_1-x_2} \beta(E, E, \bv_i,\bv_j)  \dd x_1  \dd x_2 \\
&+2\int f_i\left(x\right)f_j(x) \frac{\rho_{2c}(E) }{ E^2} \left(\bv_i^\top \frac{\Sigma}{(1+m_{2c}(E)\Sigma)(1+\overline m_{2c}(E)\Sigma)}\bv_j\right)^2\dd x +\OO(\sqrt{\eta})\\
&=2\int f_i\left(x\right)f_j(x) \frac{\rho_{2c}(E) }{ E^2} \left(\bv_i^\top \frac{\Sigma}{(1+m_{2c}(E)\Sigma)(1+\overline m_{2c}(E)\Sigma)}\bv_j\right)^2\dd x +\OO(\sqrt{\eta}),
\end{align*}
where in the second step  we used $\beta(E,E, \bv_i,\bv_j)=0$. Taking $N\to\infty$, we get \eqref{simpfff}.
\end{proof}

The proof of Lemma \ref{lem moment2} is based on the proof of Lemma \ref{lem moment}. More precisely, we will use the Helffer-Sj{\"o}strand formula in Lemma \ref{HSLemma} to reduce the problem to the study of the CLT for the process $Y(\bu,w)$. Denote $\wt\eta=N^{-\e_0}\eta$ for some small constant $\e_0>0$ and abbreviate
 \be\nonumber f_\eta(x):=f\left(\frac{x-E}{\eta}\right), \quad \wt f_\eta(x+\ii y)= f_\eta(x)+\ii\left( f_\eta(x+y) - f_\eta(x)\right).\ee
Let $\chi\in \cal C_c^\infty(\R)$ be a smooth cutoff function as in Lemma \ref{HSLemma} satisfying that (i) $\chi(y)=1$ for $|y| \le 1$, (ii) $\chi(y)=0$ for $|y|\ge 2$, and (iii) $\|\chi^{(k)}\|_\infty=\OO(1)$ for any fixed $k\in \N$. Then, using Lemma \ref{HSLemma}, we obtain that
 \be\label{useHS}  \left\langle \bu, f_\eta(\wt {\cal Q}_1)\bu\right\rangle=  \frac1{\pi}\int_{\mathbb C}   \bu^T\frac{\partial_{\overline z} (\wt f_\eta(z)\chi(y/\wt\eta) )}{\wt {\cal Q}_1-z}\bu \dd^2 z
 =\int_{\C} \phi_f(z)  (\cal G_1)_{\mathbf u\mathbf u}(z)\dd^2 z,\ee
 where we used \eqref{def_green} in the second step, and $\phi_f$ is defined as
 \begin{align*}
 \phi_f(x+\ii y)&:= \frac{1}{2\pi} \left[ (i-1)(f'_\eta(x+y)-f'_\eta(x))\chi(y/\wt\eta) - \frac1{\wt\eta} (f_\eta(x+y)-f_\eta(x))\chi'(y/\wt\eta)\right]+ \frac{\ii}{2\pi \wt\eta}f_{\eta}(x) \chi'(y/\wt\eta) .
 \end{align*}

For simplicity, the bulk of the proof is devoted to the calculation of the moments 
\be\nonumber\mathbb E \left[ Z^k_{\eta,E}(\bv, f)\right] ,\quad k \in \N, \ \ \bv\in \R^{\cal I_1}, \ \ f\in \cal C^{1,a,b}(\R_+). \ee
The proof for the more general expression in \eqref{moments partf} is exactly the same, except for some immaterial changes of notations. We will describe it briefly at the end of the proof. 
Denoting $\bu:=O\bv$, we have
 \be\nonumber Z(f)\equiv Z_{\eta,E}(\bv, f)=  \sqrt{  \frac N \eta}  \left(\left\langle \bu, f\left( \eta^{-1}(\wt{\mathcal Q}_1 - E)  \right) \bu\right\rangle -\int_{\lambda_-}^{\lambda_+}  f\left( \frac{x - E}{\eta} \right) \dd F_{1c,\bv}(x)\right).\ee
With \eqref{useHS}, we can write that
\be
 \begin{split}\label{Zfk}
&\E\left[ Z(f)\right]^{k}=\frac1{\eta^{k/2}}\int \frac{\phi_f(z_1)\cdots\phi_f(z_k) }{\sqrt{|y_1|\cdots |y_k|}}\E\left[Y(z_1)\cdots Y(z_k)\right]  \dd^2 z_1 \cdots \dd^2 z_k,
 \end{split}
 \ee
 where we have used the simplified notation 
  \be\nonumber  Y(z_i) \equiv Y(\bu,z_i):= \sqrt{N|y_i|} \langle \bu, (\cal G_1 - z_i^{-1}\Pi(z_i))\bu\rangle,\quad z_i:=x_i+\ii y_i, \ \ 1\le i \le k.\ee

Recall that with the anisotropic local law \eqref{aniso_law}, we only have the estimate $Y(z)\prec 1$ for $\im z \gg N^{-1}$. In the next lemma, we generalize this bound to $z$ with smaller imaginary part. 

 \begin{lemma}
 Suppose \eqref{aniso_law} holds for all $z\in \mathbf D$ with $q\prec N^{-1/2}$. For any deterministic unit vectors $\mathbf u, \mathbf v \in \mathbb R^{\mathcal I_1} $, we have 
\begin{equation}\label{aniso_law0}
\left| \langle \mathbf u, G(X,z) \mathbf v\rangle - \langle \mathbf u, \Pi (z)\mathbf v\rangle \right| \prec (N\eta)^{-1/2} + (N\eta)^{-1},
\end{equation}
for all $z\in \mathbf S(\omega,N) := \{z\in \mathbb C_+:   |z|\ge \omega , 0< \eta\le \omega^{-1}\}.$
\end{lemma}
 \begin{proof}
By Theorem \ref{lem_EG0}, we know that \eqref{aniso_law0} holds for $z\in \mathbf S(\omega,N)$ with $\eta \ge N^{-1+\e}$ for any small constant $\e>0$. It remains to show that \eqref{aniso_law0} holds for $z\in \mathbf S(\omega,N)$ with $\eta\le \eta_0:=N^{-1+\e}$. For $z=E+\ii \eta \in \mathbf S(\omega,N)$ with $\eta\le \eta_0$, we denote $z_0:=E+\ii \eta_0$. Then, using the spectral decomposition \eqref{spectral1}, we get 
\begin{align}
&  |G_{\bu\bv}(z) - G_{\bu\bv}(z_0)|\lesssim \sum_{k = 1}^{n} \frac{ \eta_0|\langle \bu,\xi_k\rangle||\langle \bv,\xi_k\rangle|}{|(\lambda_k-E-\ii\eta) (\lambda_k-E-\ii\eta_0)| } \nonumber \\
&\le \eta_0\left( \sum_{k = 1}^{n} \frac{|\langle \bu,\xi_k\rangle|^2}{|\lambda_k-E-\ii\eta|^2 } \right)^{1/2}\left(\sum_{k = 1}^{n} \frac{ |\langle \bv,\xi_k\rangle|^2}{|\lambda_k-E-\ii\eta_0|^2 } \right)^{1/2} \nonumber\\
 &\le \eta_0\left( \frac{\eta_0^2}{\eta^2} \sum_{k = 1}^{n} \frac{ |\langle \bu,\xi_k\rangle|^2}{|\lambda_k-E-\ii\eta_0|^2 }\right)^{1/2}\sqrt{\frac{\im [z_0^{-1} G_{\bv\bv}(z_0)]}{\eta_0}} = \frac{\eta_0}{\eta} \sqrt{\im  \frac{G_{\bu\bu}(z_0)}{z_0}\cdot\im \frac{G_{\bv\bv}(z_0)}{z_0}} \prec \frac{N^\e}{N\eta}, \label{downlaw}
   \end{align}
where in the third and fourth steps we used the identity 
\be\nonumber   \sum_{k = 1}^{n} \frac{ |\langle \bv,\xi_k\rangle|^2}{|\lambda_k-E-\ii\eta_0|^2 }  = \frac{\im [z_0^{-1}G_{\bv\bv}(z_0)]}{\eta_0},\ee
and in the last step we applied \eqref{aniso_law0} to $G(z_0)$. On the other hand, using \eqref{Immc}, we get $|\Pi(z)-\Pi(z_0)|=\OO(1)$. Together with \eqref{aniso_law0} for $G(z_0)$ and the bound \eqref{downlaw}, it gives that 
 \be\nonumber |G_{\bu\bv}(z)- \Pi_{\bu\bv}(z)|\prec  1+ \frac{N^\e}{N\eta}+\frac1{\sqrt{N\eta_0}} ,\quad  \eta \le N^{-1+\e}.\ee
Since $\e$ is arbitrary, we conclude \eqref{aniso_law0}. 
 \end{proof}

With the above lemma, we obtain the following a priori estimates on $Y(z)$:
\be\label{Yzsmall}
|Y(z)|\prec 1+(Ny)^{-1/2},\quad z= x+ \ii y, \ \  |z| \ge \omega, \ \ 0<y \le \omega^{-1}.
\ee
Moreover, by the rough bound \eqref{roughbound}, we have the deterministic bound $|y||Y(z)| =\OO(1)$. Hence, combining \eqref{Yzsmall} with Lemma \ref{lem_stodomin} (iii), we obtain that for any fixed $k\in \N$ and $y>0$,
 \be\nonumber \E|Y(z)|^k= |y|^{-k}\E|yY(z)|^k \prec \left( 1+(Ny)^{-1/2}\right)^k.\ee
We will use this bound tacitly in the following proof. 
 
\subsection{The bad region}

The following argument is an extension of the one in Section 5 of \cite{Mesoscopic}. Let $\sigma:=N^{-\e_1} \eta $ for some constant $\e_1>\e_0$, which we will choose later. We  define the ``good" region
 \be\nonumber \cal R:=\{z_1, z_2 ,\ldots, z_k \in \C: |y_1|, \ldots, |y_k| \in [\sigma,2\wt\eta]\}.\ee
In this subsection, we show that the integral in \eqref{Zfk} over the ``bad" region $\cal R^c$ is negligible. For this purpose, we need to bound the following two integrals 
 \be\nonumber \int_{|y|\le \sigma}\left|\phi_f(z)\left( \frac{1}{\sqrt{\eta |y|}} + \frac{1}{|y|\sqrt{N\eta}}\right)\right|\dd^2 z , \quad \int_{  \sigma \le |y| \le 2\wt\eta}\left|\phi_f(z)\left( \frac{1}{\sqrt{\eta |y|}} + \frac{1}{|y|\sqrt{N\eta}}\right)\right|\dd^2 z .\ee
Note that by definition, we have $\phi_f(z)=0$ for $|y|\ge 2\wt\eta$.

Since $\chi'(y/\wt\eta)=0$ for $|y|\le \wt\eta$, we get that  
 \begin{align}
 &\int\limits_{|y|\le \sigma}\left|\phi_f(z)\left( \frac{1}{\sqrt{\eta |y|}} + \frac{1}{|y|\sqrt{N\eta}}\right)\right|\dd^2 z 
 \lesssim \int\limits_{|y|\le \sigma}\frac{\left|f'_\eta(x+y)-f'_\eta(x)\right|} {\sqrt{\eta |y|}} \dd^2 z+\frac1{\sqrt{N\eta}}\int\limits_{|y|\le \sigma}\frac{\left|f'_\eta(x+y)-f'_\eta(x)\right|} {|y|} \dd^2 z \nonumber \\
& =\sqrt{\frac{\sigma}{ \eta }}  \int_{|\wt y|\le 1}\frac{\left|f'(\wt x+\wt y N^{-\e_1})-f'(\wt x)\right|}{\sqrt{|\wt y|}}  \dd \wt x\dd \wt y  +\frac1{\sqrt{N\eta}}\int_{|\wt y|\le 1}\frac{\left|f'(\wt x+\wt y N^{-\e_1})-f'(\wt x)\right|} {|\wt y|} \dd \wt x\dd \wt y,\label{bddsigma}
 \end{align}
 where in the second step we applied the change of variables $\wt x=(x-E)/{\eta}$ and $\wt y:={y}/{\sigma}.$ By the H\"older continuity and decay of $f'$, we know
 \be\label{Holderdecay}
 \begin{split}
 \left|f'(\wt x+\wt y N^{-\e_1})-f'(\wt x)\right| &\le C\min\{(|\wt y|N^{-\e_1})^a,(1+|\wt x|)^{-1-b}\} \le C\frac{(|\wt y|N^{-\e_1})^{pa}}{(1+|\wt x|)^{(1-p)(1+b)}},
 \end{split}
 \ee
for all $p\in [0,1]$. Choosing $p= \frac{b}{2(1+b)}$, we have $(1-p)(1+b)=1+b/2>1$. Then, the integrals in \eqref{bddsigma} are bounded as
 \begin{align*}
   \int_{|\wt y|\le 1}\frac{\left|f'(\wt x+\wt y N^{-\e_1})-f'(\wt x)\right|}{\sqrt{|\wt y|}}  \dd \wt x\dd \wt y &\le \int_{|\wt y|\le 1}\frac{\left|f'(\wt x+\wt y N^{-\e_1})-f'(\wt x)\right|} {|\wt y|} \dd \wt x\dd \wt y\\
 &\le CN^{-pa\e_1}\int_{|\wt y|\le 1}\frac{|\wt y|^{-1+pa}}{(1+|x|)^{1+b/2}} \dd \wt x\dd \wt y\le CN^{-pa\e_1}.
 \end{align*}
 Thus, \eqref{bddsigma} gives (recall that $\eta\ge N^{-1+c_1}$)
  \begin{align}
 &\int_{|y|\le \sigma}\left|\phi_f(z)\left( \frac{1}{\sqrt{\eta |y|}} + \frac{1}{|y|\sqrt{N\eta}}\right)\right|\dd^2 z \le CN^{-pa\e_1}\left( N^{-\e_1/2} + N^{-c_1/2}\right). \label{bddsigma2}
 \end{align}

 Similarly, we can show that 
 \begin{align*}
 \int_{  \sigma \le |y| \le \wt\eta}\left|\phi_f(z)\right|\left|\left( \frac{1}{\sqrt{\eta |y|}} + \frac{1}{|y|\sqrt{N\eta}}\right)\right|\dd^2 z \le CN^{-pa\e_0}\left( N^{-\e_0/2} + N^{-c_1/2}\right).
 \end{align*}
 On the other hand, we have
 \begin{align*}
 \int_{  \wt \eta \le |y| \le 2 \wt\eta}\left|\phi_f(z)\right| \left|\left( \frac{1}{\sqrt{\eta |y|}} + \frac{1}{|y|\sqrt{N\eta}}\right)\right|\dd^2 z &=  \int_{ 1 \le |\wt y| \le 2}\left|\psi_f(\wt x,\wt y)\right|\left|\left( \frac{N^{\e_0/2}}{\sqrt{\wt y}} + \frac{N^{\e_0}}{|\wt y|\sqrt{N\eta}}\right)\right|\dd \wt x\dd \wt y \\
 &\lesssim N^{\e_0/2} + N^{\e_0-c_1/2},  
 \end{align*}
 where
  \begin{align*}
 \psi_f(\wt x, \wt y)&:= \frac{1}{2\pi} \left[ N^{-\e_0} (i-1)(f'(\wt x+N^{-\e_0}\wt y)-f'(\wt x))\chi(\wt y) -  (f(\wt x+N^{-\e_0}\wt y)-f(\wt x))\chi'(\wt y)\right] + \frac{\ii}{2\pi} f (\wt x) \chi'(\wt y).
 \end{align*}
 Combining the above two estimates with \eqref{bddsigma2}, we get
   \begin{align}
 &\int \left|\phi_f(z)\left( \frac{1}{\sqrt{\eta |y|}} + \frac{1}{|y|\sqrt{N\eta}}\right)\right|\dd^2 z \le C N^{\e_0/2} , \label{bddsigma3}
 \end{align}
 as long as we choose $\e_0<c_1$. 
 
 Now, with \eqref{Yzsmall}, \eqref{bddsigma2} and \eqref{bddsigma3}, we obtain that 
 \begin{align*} 
& \frac1{\eta^{k/2}}\int_{\cal R^c} \phi_f(z_1)\cdots\phi_f(z_k) \frac{1}{\sqrt{|y_1|\cdots |y_k|}}\mathbb E\left[Y(z_1)\cdots Y(z_k)\right] \dd^2 z_1 \cdots \dd^2 z_k \\
&\prec \sum_{s=1}^k \int_{|y_s|\le \sigma} \prod_{i=1}^k \left|\phi_f(z_i)\left( \frac{1}{\sqrt{\eta |y_i|}} + \frac{1}{|y_i|\sqrt{N\eta}}\right)\right|  \dd^2 z_1 \cdots \dd^2 z_k \lesssim N^{-\e_1/2}\cdot N^{(k-1)\e_0/2}\le N^{-\e_0},
 \end{align*}
 as long as we choose the constants $\e_0$ and $\e_1$ such that 
 \be\label{choose constants}(k+1)\e_0 < \e_1 <c_1/2.\ee
 
 \subsection{The good region}

To estimate \eqref{Zfk}, it remains to deal with the integral over the good region $\cal R$, that is, 
 \be\label{functionE1}
 \begin{split}
&\mathbb E\left[Z(f)\right]^{k}= \frac1{\eta^{k/2}}\int_{\cal R} \phi_f(z_1)\cdots\phi_f(z_k) \frac{\mathbb E\mathfrak G}{\sqrt{|y_1|\cdots |y_k|}} \dd^2 z_1 \cdots \dd^2 z_k + \OO_\prec(N^{-\e_0/2}),
 \end{split}
 \ee
 where we have abbreviated $\mathfrak G:=Y(z_1)\cdots Y(z_k)$. For $\mathfrak G$, we can apply the results in Lemma \ref{lem moment}. Note that on $\cal R$, with \eqref{choose constants}, we can simplify \eqref{Yzsmall} as
 \be\label{Yzsmall1}
|Y(z)|\prec 1 ,\quad z= x+ \ii y, \ \  |z| \ge \omega, \ \ \sigma \le y \le 2\wt\eta .
\ee

We can perform the same calculations between \eqref{EG0} and \eqref{finalG5} for $ \E\mathfrak G$. The only difference is that for $\mathfrak G$ in \eqref{EG0}, the imaginary parts of the spectral parameters are all of a fixed scale $\eta$, while for $\mathfrak G$ in the current case, the imaginary parts of the spectral parameters are in the range $\sigma\le y_i \le 2\wt\eta$. However, the calculations after \eqref{EG0} can be easily adapted to the current setting, and gives a similar expression as in \eqref{induct}:
\begin{align} \label{induct2}
\mathbb E\mathfrak G & =\sum_{s=2}^k \sqrt{|y_1y_s|}\gamma(z_{1},z_{s},\bu,\bu) \mathbb E \prod_{t\notin \{1,s\}}Y(z_{t})+ \OO_\prec\left((N\sigma)^{-1/2}\right)  .
\end{align}
The $\eta$ factor in \eqref{induct} is replaced with $\sqrt{|y_1y_s|}$ because the scaling $\sqrt{N\eta}$ in $Y({\bu_{t},w_{t}})$ of \eqref{induct} is replaced with $\sqrt{N|y_{s}|}$ in $Y(z_{s})$ here. In case the reader is worried about the real parts of $z_i$'s, we remark that due to the fact $\supp(f)\subset \R_+$, the integral in \eqref{functionE1} is nonzero only when 
\be\label{E=0}\frac{x_i + y_i - E}{\eta} \ge 0 \ \ \ \text{and} \ \ \ \frac{x_i - E}{\eta} \ge 0 \ \ \Rightarrow \ \ x_i \ge E  - 2\wt\eta,\ee
for all $1\le i \le k$. Thus, we have $x_i\gtrsim 1$ for $1\le i \le k$, which is required in the calculations leading to \eqref{induct}.

Plugging \eqref{induct2} into \eqref{functionE1} and using \eqref{bddsigma3}, we obtain that for $k\ge 2$,
\be\label{functionE2}
 \begin{split}
\mathbb E\left[Z(f)\right]^{k}&=(k-1)\left( \frac{1}{\eta}\int_{\sigma\le |y_1|,|y_s|\le 2\wt\eta} \phi_f(z_1) \phi_f(z_s) \gamma(z_{1},z_{s},\bu,\bu) \dd^2 z_1 \dd^2 z_s\right) \mathbb E\left[Z(f)\right]^{k-2} \\
&+\OO_\prec\left(N^{-\e_0/2}+ N^{k\e_0/2}(N\sigma)^{-1/2}\right) ,
 \end{split}
 \ee
where $\sigma=N^{-\e_1}\eta\ge N^{-1+c_1-\e_1}$. Recall that we have chosen the constants as in \eqref{choose constants}, so $N^{k\e_0/2}(N\sigma)^{-1/2}\le N^{-\e_0/2}$. On the other hand, when $k=1$, by \eqref{meanzeroY} we have
$ \mathbb E\mathfrak G \prec (N\sigma)^{-1/2} .$ 
Together with \eqref{functionE1}, we get 
\be\label{K=1f}\E Z(f)\prec N^{-\e_0/2} + N^{\e_0/2}(N\sigma)^{-1/2} \lesssim N^{-\e_0/2},\ee
which verifies the mean zero condition in Proposition \ref{main_prop2}.

For \eqref{functionE2}, it remains to study the expression 
\be\nonumber \cal F(z_1,z_2):=\frac{1}{\eta}\int_{\sigma\le |y_1|,|y_2|\le 2\wt\eta} \phi_f(z_1) \phi_f(z_2) \gamma(z_{1},z_{2}) \dd^2 z_1 \dd^2 z_2,\ee
where we have taken $s=2$ and abbreviated $ \gamma(z_{1},z_{2})\equiv  \gamma(z_{1},z_{2},\bu,\bu)= \wh \al(z_1, z_2, \bu,\bu) + \wh \beta(z_1, z_2, \bu,\bu)$. Here, we recall \eqref{defal2} and \eqref{defbeta2}: 
\begin{align*} 
&\wh \al(z_1, z_2)\equiv\wh \al(z_1, z_2, \bu,\bu) =\frac{m_{2c}(z_1)m_{2c}(z_2)}{N z_1 z_2}  \sum_{i \in \cal I_1,\mu\in \cal I_2} \kappa_4(i,\mu)  \left(O^\top\frac{\Lambda^{1/2}}{1+m_{2c}(z_1)\Lambda}\bu\right)^2_i \left(O^\top\frac{\Lambda^{1/2}}{1+m_{2c}(z_2)\Lambda}\bu\right)^2_i  , \\ 
&\wh \beta(z_1, z_2)\equiv \wh \beta(z_1, z_2, \bu,\bu) = 2\frac{m_{2c}(z_1)-m_{2c}(z_2)}{z_1z_2(z_1- z_2)} \left(\bu^\top \frac{\Lambda}{(1+m_{2c}(z_1)\Lambda)(1+m_{2c}(z_2)\Lambda)} \bu\right)^2  .
\end{align*}
We decompose $\phi_f$ as $ \phi_f(z)= \phi_1 + \phi_2 + \phi_3 ,$
where
\begin{align*}\phi_1:= \frac{i-1}{2\pi}  (f'_\eta(x+y)-f'_\eta(x))&\chi(y/\wt\eta),  \quad \phi_2:= - \frac1{2\pi \wt\eta} (f_\eta(x+y)-f_\eta(x))\chi'(y/\wt\eta),\quad  \phi_3:= \frac{\ii}{2\pi\wt\eta}f_{\eta}(x) \chi'(y/\wt\eta).
\end{align*}
Correspondingly, we decompose
$\cal F(z_1,z_2) = \sum_{i,j=1}^3\cal F_{ij}(z_1,z_2) ,$
where
\be\nonumber \cal F_{ij}=\cal F_{ji}:=\frac{1}{\eta}\int_{\sigma\le |y_1|,|y_2|\le 2\wt\eta} \phi_i(z_1) \phi_j(z_2) \gamma(z_{1},z_{2}) \dd^2 z_1 \dd^2 z_2. \ee
We will show that $\cal F_{33}$ is the main term, while all the other $\cal F_{ij}$ are error terms.

\subsubsection{The error terms} 

By \eqref{E=0}, we have $|z_1|\gtrsim1$ and $|z_2|\gtrsim1$. Then, we can bound $\gamma(z_1,z_2)$ in the following two cases. If $|z_1-z_2|\ge |y_1|/2$, using \eqref{Immc} and \eqref{Piii}, we get 
\be\label{gamma111}
|\gamma(z_1,z_2)|\lesssim |y_1|^{-1}.
\ee
If $|z_1-z_2|< |y_1|/2$, using \eqref{Immc}, \eqref{m'c} and \eqref{Piii}, we get 
\be\label{gamma222}
\begin{split}
|\gamma(z_1,z_2)| &\lesssim \frac{1}{\min_{1\le k \le 2L} |z_1-a_k|^{1/2}} \le \min\left\{\frac{1}{\min_{1\le k \le 2L} |x_1-a_k|^{1/2}}, |y_1|^{-1/2}\right\}\lesssim |y_1|^{-1} .
\end{split}
\ee
Now, using \eqref{gamma111} and \eqref{gamma222}, we can bound $\cal F_{11}$ as
\begin{align*}
|\cal F_{11}| & \lesssim \frac1{\eta}\int_{|y_1|\le 2\wt\eta, |y_2|\le 2\wt\eta}\left|\frac{f'_\eta(x_1+y_1)-f'_\eta(x_1)}{y_1}\right| \left|f'_\eta(x_2+y_2)-f'(x_2)\right| \dd^2 z_1 \dd^2 z_2 \\
&= N^{-\e_0}\int_{|\wt y_1|\le 2 , |\wt y_2|\le 2}\left|\frac{f'(\wt x_1+\wt y_1N^{-\e_0})-f'(\wt x_1)}{\wt y_1}\right| \left|f'(\wt x_2+\wt y_2 N^{-\e_0})-f'(\wt x_2)\right| \dd \wt x_1\dd \wt y_1\dd \wt x_2\dd \wt y_2\\
&\lesssim N^{-\e_0}\int_{|\wt y_1|\le 2, |\wt y_2|\le 2}\frac{(|\wt y_1|N^{-\e_0})^{pa}}{|\wt y_1|(1+|\wt x_1|)^{(1-p)(1+b)}}\frac{1}{(1+|\wt x_2|)^{1+b}}  \dd \wt x_1\dd \wt y_1\dd \wt x_2\dd \wt y_2 \lesssim N^{-(1+pa)\e_0}, 
\end{align*}
 where in the second step we applied the change of variables
 $\wt x_i= (x_i-E)/{\eta}$ and $\wt y:={y_i}/{\wt\eta}$, $ i\in\{1,2\},$
 and in the third step we used \eqref{Holderdecay} with $p= \frac{b}{2(1+b)}$. 
Similarly, we can bound $\cal F_{12}$, $\cal F_{13}$ and $\cal F_{22}$ as follows:
\begin{align*}
|\cal F_{12}| &\lesssim \frac1{\eta\wt\eta}\int_{|y_1|\le 2\wt\eta, |y_2|\le 2\wt\eta}\left|\frac{f'_\eta(x_1+y_1)-f'_\eta(x_1)}{y_1}\right| \left|f_\eta(x_2+y_2)-f_\eta(x_2)\right| \dd^2 z_1 \dd^2 z_2 \\
&= \int_{|\wt y_1|\le 2 , |\wt y_2|\le 2}\left|\frac{f'(\wt x_1+\wt y_1N^{-\e_0})-f'_\eta(\wt x_1)}{\wt y_1}\right| \left|f(\wt x_2+\wt y_2N^{-\e_0})-f(\wt x_2)\right| \dd \wt x_1\dd \wt y_1\dd \wt x_2\dd \wt y_2 \\
&\lesssim \int_{|\wt y_1|\le 2, |\wt y_2|\le 2}\frac{(|\wt y_1|N^{-\e_0})^{pa}}{|\wt y_1|(1+|\wt x_1|)^{(1-p)(1+b)}}\frac{1}{(1+|\wt x_2|)^{1+b}}  \dd \wt x_1\dd \wt y_1\dd \wt x_2\dd \wt y_2 \lesssim N^{-pa\e_0} ,
\end{align*}
\begin{align*}
 |\cal F_{13}| & \lesssim \frac1{\eta\wt\eta}\int_{|y_1|\le 2\wt\eta, |y_2|\le 2\wt\eta}\left|\frac{f'_\eta(x_1+y_1)-f'_\eta(x_1)}{y_1}\right| \left|f_\eta(x_2)\right| \dd^2 z_1 \dd^2 z_2 \\
&= \int_{|\wt y_1|\le 2 , |\wt y_2|\le 2}\left|\frac{f'(\wt x_1+\wt y_1N^{-\e_0})-f'_\eta(\wt x_1)}{\wt y_1}\right| \left|f(\wt x_2)\right| \dd \wt x_1\dd \wt y_1\dd \wt x_2\dd \wt y_2 \\
&\lesssim \int_{|\wt y_1|\le 2, |\wt y_2|\le 2}\frac{(|\wt y_1|N^{-\e_0})^{pa}}{|\wt y_1|(1+|\wt x_1|)^{(1-p)(1+b)}}\frac{1}{(1+|\wt x_2|)^{1+b}}  \dd \wt x_1\dd \wt y_1\dd \wt x_2\dd \wt y_2 \lesssim N^{-pa\e_0} ,
\end{align*}
and
\begin{align*}
 |\cal F_{22}| &\lesssim \frac1{\eta\wt\eta^2}\int_{|y_1|\le 2\wt\eta, |y_2|\le 2\wt\eta}\left| \frac{f_\eta(x_1+y_1)-f_\eta(x_1)}{y_1}\right| \left|f_\eta(x_2+y_2)-f_\eta(x_2)\right| \dd^2 z_1 \dd^2 z_2 \\
&=N^{\e_0} \int_{|\wt y_1|\le 2 , |\wt y_2|\le 2}\left|\frac{f(\wt x_1+\wt y_1N^{-\e_0})-f(\wt x_1)}{\wt y_1}\right| \left|f(\wt x_2+\wt y_2N^{-\e_0})-f(\wt x_2)\right| \dd \wt x_1\dd \wt y_1\dd \wt x_2\dd \wt y_2 \\
&\lesssim N^{-\e_0}\int_{|\wt y_1|\le 2, |\wt y_2|\le 2}\frac{1}{ (1+|\wt x_1|)^{1+b}}\frac{1}{(1+|\wt x_2|)^{1+b}}  \dd \wt x_1\dd \wt y_1\dd \wt x_2\dd \wt y_2 \lesssim N^{-\e_0} ,
\end{align*}
where in the third step we used 
\be\label{Holderdecaysimple}
\left|f(\wt x_1+\wt y_1N^{-\e_0})-f(\wt x_1)\right| \lesssim \frac{|\wt y_1 |N^{-\e_0}}{(1+|\wt x_1|)^{1+b}}. \ee

To bound $\cal F_{23}$, we need better bounds on $\gamma(z_1,z_2)$. We decompose the integral in $\cal F_{23}$ as
\begin{align*}
\cal F_{23}&=\frac{1}{\eta}\int_{\sigma\le |y_1|,|y_2|\le 2\wt\eta,|x_1-x_2|\ge \eta N^{-\e_0/2}} \phi_2(z_1) \phi_3(z_2) \gamma(z_{1},z_{2}) \dd^2 z_1 \dd^2 z_2 \\
&+\frac{1}{\eta}\int_{\sigma\le |y_1|,|y_2|\le 2\wt\eta,|x_1-x_2| < \eta N^{-\e_0/2}} \phi_2(z_1) \phi_3(z_2) \gamma(z_{1},z_{2}) \dd^2 z_1 \dd^2 z_2 
=:\cal F^{(1)}_{23}+ \cal F^{(2)}_{23} .
\end{align*}
For $\cal F^{(1)}_{23}$, we use the bound $|\gamma(z_1,z_2)|\lesssim \eta^{-1}N^{\e_0/2}$ when $|x_1-x_2|> \eta N^{-\e_0/2}$ to get that
\begin{align*}
|\cal F_{23}^{(1)}| &\lesssim \frac{N^{\e_0/2}}{\eta^2\wt\eta^2}\int_{|y_1|\le 2\wt\eta, |y_2|\le 2\wt\eta,|x_1-x_2|\ge \eta N^{-\e_0/2}}\left|f_\eta(x_1+y_1)-f_\eta(x_1) \right| \left|f_\eta(x_2)\right| \dd^2 z_1 \dd^2 z_2 \\
&\le N^{\e_0/2} \int_{|\wt y_1|\le 2 , |\wt y_2|\le 2}\left| f(\wt x_1+\wt y_1N^{-\e_0})-f(\wt x_1) \right| \left|f(\wt x_2)\right| \dd \wt x_1\dd \wt y_1\dd \wt x_2\dd \wt y_2 \\
&\lesssim N^{\e_0/2} \int_{|\wt y_1|\le 2, |\wt y_2|\le 2}\frac{|\wt y_1|N^{-\e_0} }{(1+|\wt x_1|)^{1+b}}\frac{1}{(1+|\wt x_2|)^{1+b}}  \dd \wt x_1\dd \wt y_1\dd \wt x_2\dd \wt y_2  \lesssim N^{-\e_0/2} .
\end{align*}
On the other hand, using \eqref{Holderdecaysimple} 
the term $\cal F^{(2)}_{23}$ can be bounded as
\begin{align*}
|\cal F_{23}^{(2)}| &\lesssim \frac{1}{\eta\wt\eta^2}\int_{|y_1|\le 2\wt\eta, |y_2|\le 2\wt\eta,|x_1-x_2|< \eta N^{-\e_0/2}}\left|\frac{f_\eta(x_1+y_1)-f_\eta(x_1)}{y_1} \right| \left|f_\eta(x_2)\right| \dd^2 z_1 \dd^2 z_2 \\
&= N^{\e_0} \int_{|\wt y_1|\le 2 , |\wt y_2|\le 2, |\wt x_1-\wt x_2|< N^{-\e_0/2}}\left| \frac{f(\wt x_1+\wt y_1N^{-\e_0})-f(\wt x_1)}{\wt y_1} \right| \left|f(\wt x_2)\right| \dd \wt x_1\dd \wt y_1\dd \wt x_2\dd \wt y_2 \\
&\lesssim  \int_{|\wt y_1|\le 2, |\wt y_2|\le 2, |\wt x_1-\wt x_2|< N^{-\e_0/2}}\frac{1}{(1+|\wt x_1|)^{1+b}}\frac{1}{(1+|\wt x_2|)^{1+b}}  \dd \wt x_1\dd \wt y_1\dd \wt x_2\dd \wt y_2  \lesssim N^{-\e_0/2} .
\end{align*}

\vspace{5pt}

In sum, we have obtained that
\be\label{errorF}\sum_{i=1}^2 \sum_{j=1}^3 |\cal F_{ij}| \lesssim N^{-\e_0/2} + N^{-pa\e_0}.\ee


\subsubsection{The main term}  
It remains to study the main term 
 \begin{align*}
 \cal F_{33}(z_1,z_2)=&-\frac{1}{4\pi^2 \eta\wt\eta^2}\int_{\wt\eta\le |y_1|,|y_2|\le 2\wt\eta} f_\eta(x_1) f_\eta(x_2)  \wh \al (z_{1},z_{2})   \chi'(y_1/\wt\eta)\chi'(y_2/\wt\eta)\dd^2 z_1 \dd^2 z_2\\
 &-\frac{1}{4\pi^2 \eta\wt\eta^2}\int_{\wt\eta\le |y_1|,|y_2|\le 2\wt\eta} f_\eta(x_1) f_\eta(x_2)  \wh\beta(z_1,z_2)  \chi'(y_1/\wt\eta)\chi'(y_2/\wt\eta)\dd^2 z_1 \dd^2 z_2 =:  \cal K_1+\cal K_2 .
 \end{align*}
For term $\cal K_1$, we first consider the integral over $R_{++}:=\{\wt\eta\le y_1\le 2\wt\eta, \wt\eta\le y_2 \le 2\wt\eta\}$,
\begin{align}
&(\cal K_1)_{++}:=-\frac{1}{4\pi^2 \eta\wt\eta^2}\int_{R_{++}} f_\eta(x_1) f_\eta(x_2) \wh \al (z_{1},z_{2}) \chi'(y_1/\wt\eta)\chi'(y_2/\wt\eta)\dd^2 z_1 \dd^2 z_2 \nonumber\\
&=-\frac{\eta}{4\pi^2}\iint_{1\le \wt y_{1},\wt y_{2} \le 2} f\left( \wt x_1\right)f\left(\wt x_2\right)\chi'(\wt y_1)\chi'(\wt y_2)  \al\left( (E+\wt x_1 \eta) + \ii \wt y_1 \wt\eta,(E+\wt x_2\eta) + \ii \wt y_2 \wt\eta\right)  \dd \wt x_1 \dd \wt y_1 \dd \wt x_2 \dd \wt y_2.\nonumber
\end{align}
With \eqref{m'c}, we can obtain that 
\be\label{sqrtal}\left|\wh\al\left( (E+\wt x_1 \eta) + \ii \wt y_1 \wt\eta,(E+\wt x_2\eta) + \ii \wt y_2 \wt\eta\right) - \wh\al_{++}\left( E+\wt x_1 \eta ,E+\wt x_2\eta \right)\right|\lesssim  {\wt \eta}^{1/2}.\ee
Here, for $x_1,x_2\in \R_+$ and $\fa,\fb\in \{+,-\}$, we denote
\be\nonumber
\begin{split} 
\wh \al_{\fa\fb}(x_1, x_2):= \sum_{i\in \cal I_1,\mu\in \cal I_2} \frac{ \kappa_4(i,\mu)}{N}  &\left[ \frac{m_{2c}^{\fa}(x_1)}{x_1}\left(O^\top\frac{\Lambda^{1/2}}{1+m_{2c}^{\fa}(x_1)\Lambda}\bu\right)^2_i\right] \left[\frac{m_{2c}^{\fb}(x_2)}{x_2}\left(O^\top\frac{\Lambda^{1/2}}{1+m_{2c}^{\fb}(x_2)\Lambda}\bu\right)^2_i\right]  ,
\end{split}
\ee
where for a complex number $z\in \C$, we used the notations $z^+:=z$ and $z^-:=\overline z$. Thus, $(\cal K_1)_{++}$ can be reduced to
\begin{align*}
(\cal K_1)_{++}&=-\frac{\eta}{4\pi^2}\int_{1\le \wt y_{1},\wt y_{2} \le 2} f\left( \wt x_1\right)f\left(\wt x_2\right)\chi'(\wt y_1)\chi'(\wt y_2)  \wh\al_{++}\left( E+\wt x_1 \eta  ,E+\wt x_2\eta \right)  \dd \wt x_1 \dd \wt y_1 \dd \wt x_2 \dd \wt y_2 +\OO ( {\wt\eta}^{1/2} ) \\
&= - \frac{\eta}{4\pi^2}\iint  f\left( \wt x_1\right)f\left(\wt x_2\right)\wh\al_{++}\left(E+\wt x_1 \eta ,E+\wt x_2\eta \right)  \dd \wt x_1  \dd \wt x_2  +\OO(N^{-\e_0/2}).
\end{align*}
Similarly, we can calculate the integrals over the other three regions: $(\cal K_1)_{+-}$ for $R_{+-}:=\{\wt\eta\le y_1\le 2\wt\eta, -2\wt\eta\le y_2 \le -\wt\eta\}$, $(\cal K_1)_{-+}$ for $R_{-+}:=\{-2\wt\eta\le y_1\le -\wt\eta, \wt\eta\le y_2 \le 2\wt\eta\}$, and $(\cal K_1)_{--}$ for $R_{--}:=\{-2\wt\eta\le y_1\le -\wt\eta, -2\wt\eta\le y_2 \le -\wt\eta\}$. Combining all these four terms, we obtain that
\be
\begin{split}\label{K1}
\cal K_1&=-\frac{\eta}{4\pi^2}\iint f\left( \wt x_1\right)f\left(\wt x_2\right)  (\wh\al_{++} + \wh\al_{--} - \wh\al_{+-} - \wh\al_{-+} )\left( E+\wt x_1 \eta  ,E+\wt x_2\eta \right)  \dd \wt x_1  \dd \wt x_2 +\OO (N^{-\e_0/2}) \\
&=  \frac{\eta}{\pi^2}\iint  f\left( \wt x_1\right)f\left(\wt x_2\right)\al \left(E+\wt x_1 \eta ,E+\wt x_2\eta ,  \bv,\bv\right)  \dd \wt x_1  \dd \wt x_2  +\OO(N^{-\e_0/2}),
\end{split}
\ee
where recall that for $x_1,x_2\in \R$ and $\bv=O^\top\bu$, $\al$ is defined in \eqref{defal}.

Next, we study  the term $\cal K_2$. We introduce the notations
\begin{align*} 
\wt\beta(z_1,z_2):= \frac{1}{z_1z_2}\left(\bu^\top\frac{1}{1+m_{2c}(z_1)\Lambda}\Lambda\frac{1}{1+m_{2c}(z_2)\Lambda}\bu \right)^2 , 
\end{align*}
and for $x_1,x_2\in \R_+$, 
\begin{align*} 
\wt\beta_{\fa\fb}(x_1,x_2):=\frac{1}{x_1x_2} \left(\bu^\top\frac{1}{1+m_{2c}^{\fa}(x_1)\Lambda}\Lambda\frac{1}{1+m^{\fb}_{2c}(x_2)\Lambda} \bu\right)^2 , \quad \fa,\fb\in \{+,-\}.
\end{align*}
Then, we can write that
\be\nonumber \wh \beta(z_1, z_2):= 2\frac{m_{2c}(z_1)-m_{2c}(z_2)}{z_1- z_2} \wt\beta(z_1,z_2).\ee
We first consider the integral over the region $R_{++}$:
\be\label{whbeta}
\begin{split}
&(\cal K_2)_{++}:= -\frac{1}{4\pi^2 \eta\wt\eta^2}\int_{R_{++}} f_\eta(x_1) f_\eta(x_2) \wh\beta(z_1,z_2) \chi'(y_1/\wt\eta)\chi'(y_2/\wt\eta)\dd^2 z_1 \dd^2 z_2 \\
&= -\frac{1}{2\pi^2}\int_{ 1\le \wt y_{1},\wt y_2\le 2} f\left( \wt x_1\right)f\left(\wt x_2\right)\chi'(\wt y_1)\chi'(\wt y_2) \frac{m_{2c}((E+\wt x_1 \eta) + \ii \wt y_1 \wt\eta)-m_{2c}((E+\wt x_2 \eta) + \ii \wt y_2 \wt\eta)}{(\wt x_1-\wt x_2) + \ii (\wt y_1-\wt y_2)N^{-\e_0}}  \\
&\quad \times\wt\beta\left( (E+\wt x_1 \eta) + \ii \wt y_1 \wt\eta,(E+\wt x_2\eta) + \ii \wt y_2 \wt\eta\right)  \dd \wt x_1 \dd \wt y_1 \dd \wt x_2 \dd \wt y_2  \\
&= -\frac{1}{2\pi^2}\int_{ 1\le \wt y_{1},\wt y_{2}\le 2} f\left( \wt x_1\right)f\left(\wt x_2\right)\chi'(\wt y_1)\chi'(\wt y_2) \frac{m_{2c}((E+\wt x_1 \eta) + \ii \wt y_1 \wt\eta)-m_{2c}((E+\wt x_2 \eta) + \ii \wt y_2 \wt\eta)}{(\wt x_1-\wt x_2) + \ii (\wt y_1-\wt y_2)N^{-\e_0}}   \\
&\quad \times\wt\beta_{++}\left( E+\wt x_1 \eta ,E+\wt x_2\eta \right)  \dd \wt x_1 \dd \wt y_1 \dd \wt x_2 \dd \wt y_2 +\OO(N^{-\e_0/2}),
\end{split}
\ee
where we used a similar bound for $\wt \beta$ as in \eqref{sqrtal}:
\be\label{sqrtbeta}\left|\wt\beta\left( (E+\wt x_1 \eta) + \ii \wt y_1 \wt\eta,(E+\wt x_2\eta) + \ii \wt y_2 \wt\eta\right) - \wt\beta_{++}\left( E+\wt x_1 \eta ,E+\wt x_2\eta \right)\right|\lesssim \sqrt{\wt \eta},\ee
and the following bound by \eqref{m'c}: 
\begin{align*}
&\int_{ 1\le \wt y_{1},\wt y_{2}\le 2} |f\left( \wt x_1\right)f\left(\wt x_2\right)| \left|\frac{m_{2c}((E+\wt x_1 \eta) + \ii \wt y_1 \wt\eta)-m_{2c}((E+\wt x_2 \eta) + \ii \wt y_2 \wt\eta)}{(\wt x_1-\wt x_2) + \ii (\wt y_1-\wt y_2)N^{-\e_0}}\right| \dd \wt x_1 \dd \wt y_1 \dd \wt x_2 \dd \wt y_2  \\
& \lesssim \int_{ 1\le \wt y_{1,2}\le 2} |f\left( \wt x_1\right)f\left(\wt x_2\right)|  \frac{\sqrt{\eta}}{|\wt x_1-\wt x_2|^{1/2} + |\wt y_1-\wt y_2|^{1/2}N^{-\e_0/2}} \dd \wt x_1 \dd \wt y_1 \dd \wt x_2 \dd \wt y_2=\OO(1).
\end{align*}

We decompose the integral on the right-hand side of \eqref{whbeta} as  \smash{$(\cal K_2)_{++}=(\cal K_2^{(1)})_{++}+ (\cal K_2^{(2)})_{++},$} where \smash{$(\cal K_2^{(1)})_{++}$} contains the integral over the region with $|\wt x_1-\wt x_2|\le N^{-\e}$ and {$(\cal K_2^{(2)})_{++}$} contains the integral over the region with $|\wt x_1-\wt x_2|> N^{-\e}$, with $\e$ being a sufficiently small constant such that $0<\e<\e_0/10.$ For $(\cal K_2^{(1)})_{++}$, we have that
\begin{align*}
 |(\cal K_2^{(1)})_{++}|&\lesssim  \int_{1\le \wt y_{1},\wt y_2\le 2, |\wt x_1-\wt x_2|\le N^{-\e}} \frac{\eta^{1/2}  \left| f\left( \wt x_1\right)f\left(\wt x_2\right)\right| \wt\beta_{++}\left( E+\wt x_1 \eta , E+\wt x_2\eta \right) }{|\wt x_1-\wt x_2|^{1/2} + |\wt y_1-\wt y_2|^{1/2}N^{-\e_0/2}} \dd \wt x_1 \dd \wt y_1 \dd \wt x_2 \dd \wt y_2+\OO(N^{-\e_0})\\
& \lesssim  \int_{1\le \wt y_{1},\wt y_2\le 2, |\wt x_1-\wt x_2|\le N^{-\e}}\frac{ \left| f\left( \wt x_1\right)f\left(\wt x_2\right)\right| }{|\wt x_1-\wt x_2|^{1/2}}\dd \wt x_1 \dd \wt y_1 \dd \wt x_2 \dd \wt y_2 \lesssim N^{-\e/2},
\end{align*}
where we used \eqref{m'c} in the first step. For $(\cal K_2^{(2)})_{++}$, we have that
\begin{align*}
&(\cal K_2^{(2)})_{++}= -\frac{1}{2\pi^2}\int_{ 1\le \wt y_{1},\wt y_{2}\le 2,|\wt x_1-\wt x_2|> N^{-\e}} f\left( \wt x_1\right)f\left(\wt x_2\right)\chi'(\wt y_1)\chi'(\wt y_2) \frac{m_{2c}(E+\wt x_1 \eta )-m_{2c}(E+\wt x_2 \eta)}{\wt x_1-\wt x_2 }  \nonumber\\
&\qquad \qquad \ \ \ \times\wt\beta_{++}\left( E+\wt x_1 \eta ,E+\wt x_2\eta \right)  \dd \wt x_1 \dd \wt y_1 \dd \wt x_2 \dd \wt y_2  + \OO(N^{-\e_0/2+\e})\\
&= -\frac{1}{2\pi^2}\iint_{|\wt x_1-\wt x_2|> N^{-\e}} f\left( \wt x_1\right)f\left(\wt x_2\right)\frac{m_{2c}(E+\wt x_1 \eta )-m_{2c}(E+\wt x_2 \eta)}{\wt x_1-\wt x_2 } \wt\beta_{++}\left( E+\wt x_1 \eta ,E+\wt x_2\eta \right)  \dd \wt x_1 \dd\wt x_2 + \OO(N^{-\e_0/4})\\
&= -\frac{1}{2\pi^2}\iint f\left( \wt x_1\right)f\left(\wt x_2\right)\frac{m_{2c}(E+\wt x_1 \eta )-m_{2c}(E+\wt x_2 \eta)}{\wt x_1-\wt x_2 } \wt\beta_{++}\left( E+\wt x_1 \eta ,E+\wt x_2\eta \right)  \dd \wt x_1 \dd\wt x_2  + \OO(N^{-\e/2}),
\end{align*}
where in the first step we used 
\be \nonumber \frac{1}{(\wt x_1-\wt x_2) + \ii (\wt y_1-\wt y_2)N^{-\e_0}}= \frac{1}{\wt x_1-\wt x_2} + \OO(N^{-\e_0+2\e}),\ee
and $m_{2c}((E+\wt x_i \eta) + \ii \wt y_i \wt\eta)-m_{2c}(E+\wt x_i \eta )=\OO(N^{-\e_0/2})$ by \eqref{m'c},  and in the last step we used 
\begin{align*}
&\iint_{|\wt x_1-\wt x_2|\le N^{-\e}} |f\left( \wt x_1\right)f\left(\wt x_2\right)|\left|\frac{m_{2c}(E+\wt x_1 \eta )-m_{2c}(E+\wt x_2 \eta)}{\wt x_1-\wt x_2 } \wt\beta_{++}\left( E+\wt x_1 \eta ,E+\wt x_2\eta \right)\right|  \dd \wt x_1 \dd\wt x_2  \\
& \lesssim \iint_{|\wt x_1-\wt x_2|\le N^{-\e}} \frac{|f\left( \wt x_1\right)f\left(\wt x_2\right)|}{|\wt x_1-\wt x_2|^{1/2}}  \dd \wt x_1 \dd\wt x_2 \lesssim N^{-\e/2}.
\end{align*}
In sum, we get that
\be\label{K++}
\begin{split}
(\cal K_2)_{++}=& \frac{-1}{2\pi^2}\iint f\left( \wt x_1\right)f\left(\wt x_2\right)\frac{m_{2c}(E+\wt x_1 \eta )-m_{2c}(E+\wt x_2 \eta)}{\wt x_1-\wt x_2 } \wt\beta_{++}\left( E+\wt x_1 \eta ,E+\wt x_2\eta \right)  \dd \wt x_1 \dd\wt x_2  \\
&+ \OO(N^{-\e/2}).
\end{split}
\ee

Then, we study the integral $(\cal K_2)_{+-}$. Using \eqref{sqrtbeta} and \eqref{m'c}, we can simplify that
\begin{align*}
&(\cal K_2)_{+-}= -\frac{1}{2\pi^2}\int_{ 1\le \wt y_{1}\le 2,-2\le \wt y_2\le -1} f\left( \wt x_1\right)f\left(\wt x_2\right)\chi'(\wt y_1)\chi'(\wt y_2)  \nonumber\\
&\times\frac{m_{2c}((E+\wt x_1 \eta) + \ii \wt y_1 \wt\eta)-m_{2c}((E+\wt x_2 \eta) + \ii \wt y_2 \wt\eta)}{(\wt x_1-\wt x_2) + \ii (\wt y_1-\wt y_2)N^{-\e_0}} \wt\beta\left( (E+\wt x_1 \eta) + \ii \wt y_1 \wt\eta,(E+\wt x_2\eta) + \ii \wt y_2 \wt\eta\right)  \dd \wt x_1 \dd \wt y_1 \dd \wt x_2 \dd \wt y_2 \nonumber\\
&=-\frac{1}{2\pi^2}\int_{ 1\le \wt y_{1}\le 2,-2\le \wt y_2\le -1} f\left( \wt x_1\right)f\left(\wt x_2\right)\chi'(\wt y_1)\chi'(\wt y_2)  \nonumber\\
&\quad \times\frac{m_{2c}(E+\wt x_1 \eta )-\overline m_{2c}(E+\wt x_2 \eta)}{(\wt x_1-\wt x_2) + \ii (\wt y_1-\wt y_2)N^{-\e_0}} \wt\beta_{+-}\left(E+\wt x_1 \eta ,E+\wt x_2\eta \right)  \dd \wt x_1 \dd \wt y_1 \dd \wt x_2 \dd \wt y_2   \\
&\quad + \OO\left(  \int_{ 1\le \wt y_{1}\le 2,-2\le \wt y_2\le -1} \frac{N^{-\e_0/2}|f\left( \wt x_1\right)f\left(\wt x_2\right)|}{|\wt x_1-\wt x_2| + |\wt y_1-\wt y_2|N^{-\e_0}} \dd \wt x_1 \dd \wt y_1 \dd \wt x_2 \dd \wt y_2 \right).
\end{align*}
We can bound the second error term as  
\begin{align*}
& \int_{ 1\le \wt y_{1}\le 2,-2\le \wt y_2\le -1} \frac{N^{-\e_0/2}|f\left( \wt x_1\right)f\left(\wt x_2\right)|}{|\wt x_1-\wt x_2| + |\wt y_1-\wt y_2|N^{-\e_0}} \dd \wt x_1 \dd \wt y_1 \dd \wt x_2 \dd \wt y_2  \\
& \lesssim  \iint_{ 1\le \wt y_{1}\le 2,-2\le \wt y_2\le -1} \frac{N^{-\e_0/2}}{|\wt x_1-\wt x_2| +2N^{-\e_0}} \frac{1}{(1+|\wt x_1|)^{1+b}}\frac{1}{(1+|\wt x_2|)^{1+b}}\dd \wt x_1 \dd \wt y_1 \dd \wt x_2 \dd \wt y_2  \lesssim N^{-\e_0/2}\log N. 
\end{align*}
Then, we can write that
\begin{align*}
&(\cal K_2)_{+-}=-\frac{1}{2\pi^2}\iint_{ 1\le \wt y_{1}\le 2,-2\le \wt y_2\le -1} f\left( \wt x_1\right)f\left(\wt x_2\right)\chi'(\wt y_1)\chi'(\wt y_2) \left[m_{2c}(E+\wt x_1 \eta )-\overline m_{2c}(E+\wt x_2 \eta)\right]  \nonumber\\
&\times \wt\beta_{+-}\left(E+\wt x_1 \eta ,E+\wt x_2\eta \right)   \re\frac{1}{(\wt x_1-\wt x_2) + \ii (\wt y_1-\wt y_2)N^{-\e_0}}\dd \wt x_1 \dd \wt y_1 \dd \wt x_2 \dd \wt y_2   \\
&+\frac{\ii}{2\pi^2}\iint_{ 1\le \wt y_{1}\le 2,-2\le \wt y_2\le -1} f\left( \wt x_1\right)f\left(\wt x_2\right)\chi'(\wt y_1)\chi'(\wt y_2) \left[m_{2c}(E+\wt x_1 \eta )-\overline m_{2c}(E+\wt x_2 \eta)\right]   \nonumber\\
&\times \wt\beta_{+-}\left(E+\wt x_1 \eta ,E+\wt x_2\eta \right) \im\frac{1}{(\wt x_1-\wt x_2) - \ii (\wt y_1-\wt y_2)N^{-\e_0}}\dd \wt x_1 \dd \wt y_1 \dd \wt x_2 \dd \wt y_2 + \OO\left(N^{-\e_0/2}\log N \right) \\
&=:(\cal K^{(1)}_2)_{+-}+ (\cal K^{(2)}_2)_{+-}+\OO (N^{-\e_0/2} \log N ).
\end{align*}
For the term $(\cal K_2^{(1)})_{+-}$, we observe that the integral converges to the Cauchy principal value, while for the term $(\cal K_2^{(1)})_{+-}$, $\pi^{-1}\im[(\wt x_1-\wt x_2) - \ii (\wt y_1-\wt y_2)N^{-\e_0}]^{-1}$ is an approximate delta function. More precisely, we have the following estimates. The proof is standard, so we omit the details.

\begin{lemma}\label{lem_Hilbert}
Suppose $g(x_1, x_2)$ is $1/2$-H\"older continuous uniformly in $x_1$ and $x_2$,  and $|g(x_1,x_2)| \le C(1+|x_1|)^{-(1+b)}(1+|x_2|)^{-(1+b)}$ for some constant $C>0$. Then, for any $0<\delta \ll 1$, we have
\be\nonumber \left|\frac1\pi\iint_{x_1, x_2} g(x_1,x_2)\im\frac{1}{(x_1-x_2)-\ii \delta}\dd x_1 \dd x_2 -   \int g(x_1,x_1) \dd x_1  \right|\lesssim \delta^{1/2},\ee
and
\be\nonumber \left|\iint_{x_1, x_2} g(x_1,x_2)\Re\frac{1}{(x_1-x_2)+\ii \delta}\dd x_1 \dd x_2 - PV \iint_{x_1, x_2} \frac{g(x_1,x_2)}{x_1-x_2} \dd x_1 \dd x_2 \right|\lesssim \delta^{1/3},\ee
where
\be\nonumber  PV \iint_{x_1, x_2} \frac{g(x_1,x_2)}{x_1-x_2} \dd x_1 \dd x_2:=\lim_{\delta\downarrow 0} \iint_{x_1, x_2} g(x_1,x_2)\Re\frac{1}{(x_1-x_2)+\ii \delta}\dd x_1 \dd x_2 .\ee
\end{lemma}

With Lemma \ref{lem_Hilbert} and the fact $\rho_{2c}(x)=\pi^{-1}\im m_{2c}(x)$, we obtain that  
\begin{align}
(\cal K_2^{(1)})_{+-}&=\frac{1}{2\pi^2}PV\iint \frac{f\left( \wt x_1\right)f\left(\wt x_2\right)}{\wt x_1-\wt x_2} \left[m_{2c}(E+\wt x_1 \eta )-\overline m_{2c}(E+\wt x_2 \eta)\right] \wt\beta_{+-}\left(E+\wt x_1 \eta ,E+\wt x_2\eta \right)\dd\wt x_1 \dd \wt x_2  \nonumber\\
&+\OO(N^{-\e_0/3}),\label{K+-1}
\end{align}
\begin{align}
(\cal K_2^{(2)})_{+-}&=- \frac{\ii}{2\pi}\int f^2\left( \wt x_1\right)  \left[m_{2c}(E+\wt x_1 \eta )-\overline m_{2c}(E+\wt x_1 \eta)\right] \wt\beta_{+-}\left(E+\wt x_1 \eta ,E+\wt x_1\eta \right) \dd \wt x_1 +\OO(N^{-\e_0/2}) \nonumber\\
&= \int f^2\left( \wt x_1\right)  \rho_{2c}(E+\wt x_1 \eta ) \cdot \wt\beta_{+-}\left(E+\wt x_1 \eta ,E+\wt x_1\eta \right) \dd \wt x_1 +\OO(N^{-\e_0/2}).\label{K+-2}
\end{align}

Now, combining \eqref{K++}, \eqref{K+-1}, \eqref{K+-2}, and the simple facts $(\cal K_2)_{--}=\overline {(\cal K_2)_{++}}$ and $(\cal K_2)_{-+}=\overline {(\cal K_2)_{+-}}$,  we get that
\be
\begin{split}\label{K2}
\cal K_2 
& = \frac{1}{\pi^2}PV\iint_{x_1,x_2} \frac{f\left(  x_1\right)f\left(x_2\right)}{x_1-x_2}\beta(x_1, x_2, \bv,\bv) \dd x_1\dd x_2   \\
&+2\int f^2\left(x\right) \frac{ \rho_{2c}(E+x \eta ) }{(E+x\eta)^2} \left(\bu^\top \frac{\Lambda}{(1+m_{2c}(x)\Lambda)(1+\overline m_{2c}(x)\Lambda)}\bu\right)^2 \dd x+ \OO(N^{-\e/2})  , 
\end{split}
\ee
for small enough constant $\e>0$, where recall that $\bv=O^\top\bu$ and $\beta$ is defined in \eqref{defbeta}.  

Finally, plugging \eqref{errorF}, \eqref{K1} and \eqref{K2} into \eqref{functionE2}, we obtain that 
\be\nonumber \mathbb E\left[Z(f)\right]^{k}=(k-1)\varpi(f,f,\bv,\bv)\mathbb E\left[Z(f)\right]^{k-2}  +  \OO_\prec\left(N^{-c}\right)\ee
for some small constant $c>0$. In general, we can extend this induction relation to the more general expression in \eqref{moments partf} and hence conclude Lemma \ref{lem moment2}. 

\begin{proof}[Proof of Lemma \ref{lem moment2}]
We expand the left-hand side of \eqref{moments partf} using the Helffer-Sj\"ostrand formula, Lemma \ref{HSLemma}, and obtain a similar expression as in \eqref{Zfk}:
 \begin{align*}
\mathbb E\left[ \prod_{s=1}^k Z_{\eta,E}(\bv_{s}, f_{s})\right]=\frac1{\eta^{k/2}}\int  \frac{\phi_{f_{1}}(z_{1})\cdots\phi_{f_{k}}(z_{k})}{\sqrt{|y_{1}|\cdots |y_{k}|}}\mathbb E\left[ Y(\bu_1,z_{1})\cdots  Y(\bu_k,z_{k})\right] \dd^2 z_{1} \cdots \dd^2 z_{k}.
 \end{align*}
Then, applying the argument between \eqref{Yzsmall} and \eqref{K2}, we can obtain that 
\begin{align*} 
\mathbb E\left[ \prod_{s=1}^k Z_{\eta,E}(\bv_{s}, f_{s})\right] & =\sum_{s=2}^k \varpi(f_{1},f_{s},\bv_{1},\bv_{s}) \mathbb E \prod_{t\notin \{1,s\}}Z_{\eta,E}(\bv_{t},f_{t})+ \OO_\prec\left( N^{-c}\right)  
\end{align*}
for some constant $c>0$. With this induction relation and \eqref{K=1f}, we can conclude \eqref{moments partf}.
\end{proof}

\section{Weaker moment assumptions}\label{sec relax}

In this section we use a Green's function comparison argument to relax the moment assumptions in Propositions \ref{main_prop} and \ref{main_prop2}, and hence complete the proofs of Theorems \ref{main_thm}, \ref{main_thm2}, \ref{main_thm3} and \ref{main_thm4}. In this section, we focus on the proof of Theorems \ref{main_thm3} and \ref{main_thm4}. Later, we will explain how to extend the argument to the proof of Theorems \ref{main_thm} and \ref{main_thm2}.

For any fixed $c_0>0$, we can choose a constant $0<c_\phi <1/2$ small enough such that 
\be\nonumber \left((N/\eta)^{1/4}N^{-c_\phi}\right)^{a_\eta+c_0} \ge N^{2+\e_0}, \quad a_\eta= \frac{8}{1 - \log_N \eta},\ee
for some constant $\e_0>0$.  Then, we introduce the following truncated matrix $ X' , $ where
\be\label{phiN} X'_{i\mu} = \mathbf 1_{|X_{i\mu}|\le \phi_N}X_{i\mu},\quad \phi_N:= \frac{N^{-c_\phi}}{(N\eta)^{1/4}}.\ee
Without loss of generality, we choose $c_\phi$ small enough such that
\be\nonumber \phi_N\ge (N\eta)^{-1/2},\quad \text{for} \quad \eta\ge N^{-1+c_1}. \ee
With the moment condition \eqref{size_condition2} and a simple union bound, we get that
\begin{equation}\label{XneX}
\mathbb P(X' \ne X) =\OO ( N^{-\e_0}).
\end{equation}
Using (\ref{size_condition2}) and integration by parts, it is easy to verify that  
\begin{align*}
\mathbb E  \left|X_{i\mu}\right|1_{|X_{i\mu}|> \phi_N} =\OO(N^{-2-\e_0}), \quad \mathbb E \left|X_{i\mu}\right|^2 1_{|X_{i\mu}|> \phi_N} =\OO(N^{-2-\e_0}),
\end{align*}
which imply that
\be\label{IBP2}|\mathbb E  X'_{i\mu}| =\OO(N^{-2-\e_0}), \quad  \mathbb E |X'_{i\mu}|^2 = N^{-1} + \OO(N^{-2-\e_0}).\ee
Moreover, we trivially have
$\mathbb E  |X'_{i\mu}|^4 \le \mathbb E  |X_{i\mu}|^4 =\OO(N^{-2}).$
Then, we introduce the centered matrix $\mathring X = X' - \E X' ,$ where by \eqref{IBP2} we have that
\be\label{IBP3} 
\|\E X'\| =\OO(N^{-1-\e_0}) , \quad \text{Var}(\mathring X_{i\mu})= N^{-1} \left( 1+\OO(N^{-1-\e_0})\right).
\ee
Now, we can define $\mathring {\cal G}_{1,2}(\mathring X, z)$ (recall \eqref{def_green}) and $\mathring {G}(\mathring X, z)$ (recall \eqref{eqn_defG}) by replacing $X$ with $\mathring X$.

\begin{claim}\label{claim compcirc}
Under the above setting, we have that for any deterministic unit vectors $\bu,\bv\in \C^{\cal I}$, 
\be\nonumber \left| \langle \mathbf u, G(X,z) \mathbf v\rangle - \langle \mathbf u, \mathring G (\mathring X,z)\mathbf v\rangle \right| \prec N^{-1-\e_0}\eta^{-1/2}\ee
uniformly in $z\in \mathbf D$.
\end{claim} 
\begin{proof}
See the proof of Lemma 4.4 in Section A.1 of \cite{XYY_VESD}.
\end{proof}

Under the scaling $\sqrt{N\eta}$ in \eqref{YetaE}, $N^{-1-\e_0}\eta^{-1/2}$ is a negligible error. Hence, it suffices to prove that Theorems \ref{main_thm3} and \ref{main_thm4} hold under the following assumptions on $X$, which correspond to the above setting for $\mathring X$.

\begin{assumption}\label{main_assmadd}
Fix a small constant $\tau>0$.
\begin{itemize}
\item[(i)] $X=(X_{i\mu})$ is a real $n\times N$  matrix, whose entries are independent random variables satisfying \be\label{assm1add} 
\E X_{i\mu}=0,\quad \E X_{i\mu}^2 = N^{-1} +\OO(N^{-2-\e_0}),
\ee
and the following bounded support condition:
\begin{equation}
 \max_{i,\mu}\vert X_{i\mu}\vert \le \phi_N. \label{eq_support}
\end{equation}
Moreover, we assume that the matrix entries have bounded fourth moments
\be\label{conditionA3} 
\max_{i,\mu} \mathbb{E} | X_{i\mu} |^4 \le C N^{-2}.
\ee

\item[(ii)]  Assumption \ref{main_assm} (ii) and (iii) hold.  
\end{itemize}
\end{assumption}

The results in Section \ref{sec_maintools} can be extended to the setting with the above assumptions. In particular, we have the following version of Theorem \ref{lem_EG0}, where the only difference is that \eqref{assm1} is relaxed to \eqref{assm1add} in this theorem. 

\begin{theorem}[Theorem 3.6 of \cite{yang2018}]\label{thm_localadd} 
Suppose Assumption \ref{main_assmadd} holds.  
For any fixed $\e>0$ and deterministic unit vectors $\mathbf u, \mathbf v \in \mathbb C^{\mathcal I}$, the following anisotropic local laws holds: for any $z\in \mathbf D$, 
\begin{equation}\label{aniso_lawadd}
\left| \langle \mathbf u, G(z) \mathbf v\rangle - \langle \mathbf u, \Pi (z)\mathbf v\rangle \right| \prec \phi_N + \Psi(z).
\end{equation}
\end{theorem}

Given any random matrix $X$ satisfying Assumption \ref{main_assmadd}, we can construct another random matrix $\wt X$ that matches (in the sense of first four moments) $X$ but with smaller support of order $\OO_\prec(N^{-1/2}) $. 

\begin{lemma} [Lemma 5.1 of \cite{LY}]\label{lem_decrease}
Suppose $X$ satisfies Assumption \ref{main_assmadd}. Then, there exists another matrix $\wt{X}=(\wt X_{i\mu})$  such that $\wt{X}$  satisfies \eqref{assm1}, \eqref{high_moments} and the following moment matching condition: 
\begin{equation}\label{match_moments}
\mathbb EX_{i\mu}^k =\left[1+ \OO(N^{-1-\e_0})\right]\mathbb E\wt X_{i\mu}^k ,  \quad k=2,3,4.
\end{equation}
\end{lemma}
Define $\wt G(z):=G(\wt X, z)$ and $\wt Y_{\eta, E}$ by replacing $X$ with $\wt X$. We have shown that Lemma \ref{lem moment} holds for $\wt Y_{\eta, E}$. It remains to prove that  the joint moments of $\left( Y_{\eta,E}(\bu_1, w_1),\ldots, Y_{\eta,E}(\bu_{k}, w_{k})\right) $ match those of $( \wt Y_{\eta,E}(\bu_1, w_1),\ldots, \wt Y_{\eta,E}(\bu_{k}, w_{k}))$ asymptotically.

 \begin{proposition}\label{main_propadd}
Under the setting of Theorem \ref{main_thm3} or Theorem \ref{main_thm4} with $N^{-1+c_1}\le \eta \le 1$, for any deterministic unit vectors $\bu_1,\ldots, \bu_{r}\in \R^{\cal I_1}$ and fixed $w_1, \ldots, w_{r}\in \mathbb H$, there exists a constant $\e>0$ such that
 \be\label{compareout} \mathbb E\prod_{i=1}^r Y_{\eta,E}(\bu_i, w_i) = \mathbb E\prod_{i=1}^r \wt Y_{\eta,E}(\bu_i, w_i)  + \OO(n^{-\e}). \ee
 \end{proposition}
 
 \begin{proof}

To prove this proposition, we will use the continuous comparison method introduced in \cite{Anisotropic}. We first introduce the following interpolation.

\begin{definition}[Interpolating matrices]
Introduce the notations $X^0:=\wt X $ and $X^1:=X$. Let $\rho_{i\mu}^0$ and $\rho_{i\mu}^1$ be the laws of \smash{$\wt X_{i\mu} $} and $X_{i\mu}$, respectively. For $\theta\in [0,1]$, we define the interpolated laws
\smash{$ \rho_{i\mu}^\theta := (1-\theta)\rho_{i\mu}^0+\theta\rho_{i\mu}^1.$}
 Let \smash{$\{X^\theta: \theta\in (0,1) \}$} be a collection of random matrices such that the following properties hold. For any fixed $\theta\in (0,1)$, $(X^0,X^\theta, X^1)$ is a triple of independent $\mathcal I_1\times \mathcal I_2$ random matrices, and the matrix $X^\theta=(X_{i\mu}^\theta)$ has law
\begin{equation}\label{law_interpol}
\prod_{i\in \mathcal I_1}\prod_{\mu\in \mathcal I_2} \rho_{i\mu}^\theta(\dd X_{i\mu}^\theta).
\end{equation}
 Note that we do not require $X^{\theta_1}$ to be independent of $X^{\theta_2}$ for $\theta_1\ne \theta_2 \in (0,1)$. 
For $\lambda \in \mathbb R$, $i\in \mathcal I_1$ and $\mu\in \mathcal I_2$, we define the matrix $X_{(i\mu)}^{\theta,\lambda}$ as
\be\label{Ximulambda} 
\left(X_{(i\mu)}^{\theta,\lambda}\right)_{j\nu}:=\begin{cases}X_{i\mu}^{\theta}, &\text{ if }(j,\nu)\ne (i,\mu)\\ \lambda, &\text{ if }(j,\nu)=(i,\mu)\end{cases}.
\ee
Correspondingly, we define the resolvents 
\[G^{\theta}(z):=G\left(X^{\theta}, z\right),\ \ \ G^{\theta, \lambda}_{(i\mu)}(z):= G\left(X_{(i\mu)}^{\theta,\lambda},z\right),\]
and for $1\le s \le k$ (recall \eqref{Y12}) and $z_s :=E+ w_s\eta$,
\begin{align*}
Y_s^\theta := z_sY_{\eta,E}(\bu_s,w_s, X^\theta)= \sqrt{N\eta}\langle \bu_s, (G^\theta(z_s)-\Pi(z_s)) \bu_s\rangle, \quad (Y_s)_{(i\mu)}^{\theta,\lambda} := z_sY_{\eta,E}(\bu_s,w_s, X_{(i\mu)}^{\theta,\lambda}).\end{align*}
\end{definition}


Using (\ref{law_interpol}) and fundamental calculus, we get the following basic interpolation formula.
\begin{lemma}\label{lemm_comp_3}
 For $F:\mathbb R^{\mathcal I_1 \times\mathcal I_2} \rightarrow \mathbb C$ we have
\begin{equation}\label{basic_interp}
\begin{split}
\frac{\dd}{\dd\theta}\mathbb E F(X^\theta)&=\sum_{i\in\mathcal I_1 , \mu\in\mathcal I_2}\left[\mathbb E F\left(X^{\theta,X_{i\mu}^1}_{(i\mu)}\right)-\mathbb E F\left(X^{\theta,X_{i\mu}^0}_{(i\mu)}\right)\right] ,
\end{split}
\end{equation}
 provided all the expectations exist.
\end{lemma}

Then, the main work is devoted to proving the following estimate for the right-hand side of (\ref{basic_interp}). Note that Lemma \ref{lemm_comp_3} and Lemma \ref{lemm_comp_4} together conclude Proposition \ref{main_propadd}. 

\begin{lemma}\label{lemm_comp_4}
Under the assumptions of Proposition \ref{main_propadd}, there exists a constant $\e>0$ such that
 \begin{equation}\label{compxxx}
  \sum_{i\in\mathcal I_1}\sum_{\mu\in\mathcal I_2}\left[\mathbb EF\left(X^{\theta,X_{i\mu}^1}_{(i\mu)}\right)-\mathbb EF\left(X^{\theta,X_{i\mu}^0}_{(i\mu)}\right)\right] \le N^{-\e},
 \end{equation}
 for all $\theta\in[0,1]$, where $F(X^\theta) := \prod_{s=1}^r Y_s^\theta $. 
\end{lemma}

Underlying the proof of (\ref{compxxx}) is an expansion approach which we will describe below. We first rewrite the resolvent expansion \eqref{eq_comp_expansion} using the new notations: for any $\lambda,\lambda'\in \mathbb R$ and $K\in \mathbb N$,
\begin{equation}\label{eq_comp_expansion2}
\begin{split}
G_{(i \mu)}^{\theta,\lambda'} = G_{(i\mu)}^{\theta,\lambda}&+\sum_{k=1}^{K}  (\lambda-\lambda')^kG_{(i\mu)}^{\theta,\lambda}\left( \Delta_{i\mu} G_{(i\mu)}^{\theta,\lambda}\right)^k  + (\lambda-\lambda')^{K+1} G_{(i\mu)}^{\theta,\lambda'}\left(\Delta_{i\mu}G_{(i\mu)}^{\theta,\lambda}\right)^{K+1}.
\end{split}
\end{equation}
With this expansion, we can prove the following estimate: suppose that $y$ is a random variable satisfying $|y|\le \phi_N$, then for any deterministic unit vectors $\mathbf u,\mathbf v \in \mathbb C^{\mathcal I}$ and $z\in \mathbf D$, 
 \begin{equation}\label{comp_eq_apriori}
 \langle \bu,  \left(G_{(i\mu)}^{\theta,y}(z)-\Pi(z)\right)\bv\rangle \prec \phi_N + \Psi(z) ,\quad i\in\sI_1, \ \mu\in\sI_2.
 \end{equation}
In fact, to prove this estimate, we will apply the expansion (\ref{eq_comp_expansion2}) with $\lambda'=y$ and $\lambda = X_{i\mu}^\theta$, so that $ G_{(i\mu)}^{\theta,\lambda}=G^\theta$. To bound the right-hand side of \eqref{comp_eq_apriori}, we will use $y\le \phi_N$, $|X_{i\mu}^\theta|\le \phi_N$, the anisotropic local law \eqref{aniso_lawadd} for $G^{\theta}$, and the trivial bound $\|G_{(i\mu)}^{\theta,y}\|\le C\eta^{-1}$. We can choose $K$ such that $\phi_N^K \eta^{-1}\le 1$, and hence the last term in \eqref{eq_comp_expansion2} can be bounded by 
 \be\nonumber (\lambda-\lambda')^{K+1}\Big(G_{(i\mu)}^{\theta,\lambda'}\big(\Delta_{i\mu} G_{(i\mu)}^{\theta,\lambda}\big)^{K+1}\Big)_{\bu \bv} \prec \phi_N^{K+1} \eta^{-1}\le \phi_N.\ee
 Next, we give the proof of Lemma \ref{lemm_comp_4} using \eqref{eq_comp_expansion2} and \eqref{comp_eq_apriori}.

\begin{proof}[Proof of Lemma \ref{lemm_comp_4}]
For simplicity, we only consider the estimate for the case $Y_s^\theta=Y^\theta$ for all $1\le s \le r$, where 
\begin{align*}
Y^\theta := \sqrt{N\eta}\langle \bu, (G^\theta(z)-\Pi(z)) \bu\rangle,\quad z= E + w\eta,
\end{align*}
for any deterministic unit vector $\bu \in \R^{\cal I_1}$ and fixed $w\in \mathbb H$. In other words, we will show that
\begin{equation}\label{compxxx2}
  \sum_{i\in\mathcal I_1}\sum_{\mu\in\mathcal I_2}\left[\mathbb E\left(Y^{\theta,X_{i\mu}^1}_{(i\mu)}\right)^r-\mathbb E\left(Y^{\theta,X_{i\mu}^0}_{(i\mu)}\right)^r\right] \le n^{-\e}.
 \end{equation}
The general multi-variable case can be handled in the same way, except that the notations are a little more tedious.  

Using \eqref{eq_comp_expansion2} and \eqref{comp_eq_apriori}, we get that for any random variable $y$ satisfying $|y|\le \phi_N$ and any fixed $K\in \N$,
\be\label{My-0}
Y_{(i\mu)}^{\theta,y}-Y_{(i\mu)}^{\theta,0}=  \sum_{k=1}^K \sqrt{N\eta} (-y)^k x_k(i,\mu) + \OO_\prec ( \sqrt{N\eta}\phi_N^{K+1}),\ee
where
\be\label{xkimu}x_k(i,\mu) := \Big\langle \bu, G_{(i\mu)}^{\theta,0}\big( \Delta_{i\mu} G_{(i\mu)}^{\theta,0}\big)^k\bu \Big\rangle.\ee
In the following proof, we choose $K > 3/c_\phi $ large enough such that 
$\sqrt{N\eta}\phi_N^{K+1}\le N^{-3}.$ 
With \eqref{comp_eq_apriori}, we trivially have $x_k (i,\mu)\prec 1$ for $k\ge 1$. Moreover, we have a better bound for odd $k$:
\be\label{boundxk}
x_k(i,\mu) \prec \phi_N , \quad k\in 2\N+1.
\ee
This is because if $k$ is odd, then there exists at least one $(G_{(i\mu)}^{\theta,0})_{\bu \mu}$ or $(G_{(i\mu)}^{\theta,0})_{i\mu}$ factor in the expansion of $x_k(i,\mu)$. Using \eqref{boundxk} for $k=1$ and the bound $|y| \le \phi_N$, we obtain the rough bound 
\be\label{My-02}
\sqrt{N\eta} (-y)^k x_k(i,\mu)  \prec N^{-kc_\phi}, \quad k\ge 1.
\ee

Now, applying \eqref{My-0} and \eqref{My-02}, the Taylor expansion of $\big(Y_{(i\mu)}^{\theta,X_{i\mu}^a}  \big)^r$ up to $K$-th order gives that for $a\in\{0,1\}$,
\be\label{addtaylor}
\begin{split}
\E \left(Y_{(i\mu)}^{\theta,X_{i\mu}^a}  \right)^r -\mathbb E \left(Y_{(i\mu)}^{\theta,0} \right)^r  & = \sum_{k=1}^{K\wedge r} \begin{pmatrix}r \\ k \end{pmatrix} \E \left(Y_{(i\mu)}^{\theta,0}\right)^{r-k}  \left[\sum_{l=1}^K  \sqrt{N\eta} (-X_{i\mu}^{\theta,a})^l x_l(i,\mu) \right]^k  +\OO_\prec\left(N^{-3}\right)\\
&=  \sum_{s=1}^{K\wedge r}\sum_{k=1}^{s}\sum_{\mathbf s }^*  \begin{pmatrix}r \\ k \end{pmatrix}  \mathbb E(-X_{i\mu}^{\theta,a})^s \E   (Y_{(i\mu)}^{\theta,0} )^{r-k} \prod_{l=1}^k \sqrt{N\eta} x_{s_l}(i,\mu)   +\OO_\prec\left(N^{-3 }\right),
\end{split}
\ee
where the sum $\sum_{\mathbf s}^*$ means the sum over $\mathbf s =(s_1, \ldots, s_k)\in \N^k$ satisfying 
\be\label{lkstat}
1\le s_{i} \le K\wedge r, \quad \sum_{l=1}^k l \cdot s_l = s. 
\ee
Here, we only keep terms with $s\le K$, because otherwise by \eqref{My-02},
\be\nonumber \prod_{l=1}^k \sqrt{N\eta}  (-X_{i\mu}^{\theta,a})^{s_l} x_{s_l}(i,\mu) \prec N^{-Kc_\phi}\le N^{-3}.\ee
Then, combining \eqref{addtaylor} with \eqref{match_moments}, we get that
\begin{align*}
\left|\E \left(Y_{(i\mu)}^{\theta,X_{i\mu}^1}  \right)^r -\E \left(Y_{(i\mu)}^{\theta,X_{i\mu}^0}  \right)^r\right| &\prec N^{-1-\e_0}\sum_{s=2}^{4}\sum_{k=1}^{s} \sum_{\mathbf s } ^* N^{- s/2} \E\left| \prod_{l=1}^k  \sqrt{N\eta}x_{s_l}(i,\mu)\right| \\ &+\sum_{s=5}^{K}\sum_{k=1}^{s} \sum_{\mathbf s } ^* N^{- 2}\phi_N^{s-4}   \E\left| \prod_{l=1}^k  \sqrt{N\eta}x_{s_l}(i,\mu)\right|   + \OO_\prec(N^{-3}) ,
\end{align*}
where we used the moment bound $\mathbb E |X_{i\mu}^{\theta,a}|^s \le \phi_N^{s-4}\mathbb E |X_{i\mu}^{\theta,a}|^4 \lesssim \phi_N^{s-4}N^{-2}$ for $s\ge 4$. Thus, to show (\ref{compxxx2}), we only need to prove that there exists a constant $\e>0$ such that for $s=2,3,4$,
\begin{equation}\label{eq_comp_est234}
 N^{-1-\e_0}\sum_{i\in\mathcal I_1}\sum_{\mu\in\mathcal I_2}N^{-s/2} \E\left| \prod_{l=1}^k  \sqrt{N\eta}x_{s_l}(i,\mu)\right| \prec N^{-\e},\end{equation}
and for any fixed $s\ge 5$ and $\mathbf s$ such that \eqref{lkstat} holds, 
\begin{equation}\label{eq_comp_est}
 \sum_{i\in\mathcal I_1}\sum_{\mu\in\mathcal I_2}N^{- 2}\phi_N^{s-4}   \E\left| \prod_{l=1}^k  \sqrt{N\eta}x_{s_l}(i,\mu)\right| \prec N^{-\e}.
 \end{equation}
 To prove these two estimates, we shall use the following bounds:
 \be\label{Rimu2add} 
 |x_s({i, \mu})| \prec \begin{cases}R_i^2 + R_\mu^2,\ & \text{if } s\ge 2 \\ R_i  R_\mu + \phi_N (R_i^2 + R_\mu^2), \ & \text{if } s=1\end{cases},
 \ee
where
\be\nonumber R_i:= |\langle  \bu, G^{\theta}  \bt_i\rangle| ,\quad R_\mu:= |\langle  \bu, G^{\theta}  \mathbf e_\mu\rangle| .\ee
In fact, by definition \eqref{xkimu}, we have
\be\label{xkimu2}
x_k(i,\mu) \prec \begin{cases} |\langle \bu, G_{(i\mu)}^{\theta,0}\bt_i\rangle|^2 + |\langle \bu, G_{(i\mu)}^{\theta,0}\mathbf e_\mu\rangle|^2,\ & \text{if } s\ge 2 \\ |\langle \bu, G_{(i\mu)}^{\theta,0}\bt_i\rangle| |\langle \bu, G_{(i\mu)}^{\theta,0}\mathbf e_\mu\rangle|, \ & \text{if } s=1\end{cases}.
\ee
On the other hand, using \eqref{eq_comp_expansion2} and \eqref{comp_eq_apriori}, we get that
\be\label{Gui}
\begin{split}
 |\langle \bu,G_{(i \mu)}^{\theta,0} \mathbf t_i\rangle| &\le   |G^\theta_{\bu \bt_i}| + |X_{i\mu}^\theta| \left(|G^{\theta}_{\bu\mu}   | | \langle \bt_i,G_{(i \mu)}^{\theta,0} \bt_i\rangle |+|G^{\theta}_{\bu \mathbf t_{i}}| | \langle \mathbf e_\mu,G_{(i \mu)}^{\theta,0} \bt_i\rangle | \right)  \prec R_i + \phi_N R_\mu, 
 \end{split}
\ee
and 
\be\label{Gumu} 
\begin{split}
|\langle \bu,G_{(i \mu)}^{\theta,0} \mathbf e_\mu\rangle| & \le   |G^\theta_{\bu \mu}| + |X_{i\mu}^\theta| \left(|G^{\theta}_{\bu\mu}   | | \langle \bt_i,G_{(i \mu)}^{\theta,0} \mathbf e_\mu\rangle |+|G^{\theta}_{\bu \mathbf t_{i}}| | \langle \mathbf e_\mu,G_{(i \mu)}^{\theta,0}  \mathbf e_\mu\rangle | \right)  \prec R_\mu + \phi_N R_i.
 \end{split}
\ee
Plugging \eqref{Gui} and \eqref{Gumu} into \eqref{xkimu2}, we obtain \eqref{Rimu2add}.

Note that by Lemma \ref{lem_comp_gbound} and \eqref{aniso_lawadd}, the following estimates hold:
\be\label{l2sumimu}
R_\mu\prec \phi_N + \Psi(z) \lesssim \phi_N,\quad \sum_{i\in \cal I_1}R_i^2 + \sum_{\mu\in \cal I_2}R_\mu^2\prec \eta^{-1},
\ee
where we used $\phi_N\ge  (N\eta)^{-1/2}\gtrsim \Psi(z)$ for the first estimate. Then, with \eqref{Rimu2add} and \eqref{l2sumimu}, we can bound the left-hand side of \eqref{eq_comp_est234} by
\be\nonumber
\begin{split} 
N^{-1-\e_0}\sum_{i\in\mathcal I_1}\sum_{\mu\in\mathcal I_2}N^{-s/2} \E\left| \prod_{l=1}^k  \sqrt{N\eta}x_{s_l}(i,\mu)\right| &\prec N^{-1-\e_0}\sum_{i\in\mathcal I_1}\sum_{\mu\in\mathcal I_2}N^{-s/2} (N\eta)^{k/2}(R_i^2+ R_\mu^2) \\ &\prec N^{-\e_0}N^{-(s-k)/2}\eta^{(k-2)/2} \le N^{-\e_0}.
\end{split}
\ee
This concludes \eqref{eq_comp_est234}. For the proof of \eqref{eq_comp_est}, we consider the following three cases.

 \vspace{5pt}
 
 \noindent{\bf Case 1:} $s_l \ge 2$ for $1\le l \le k$, which gives $k\le s/2$. Then, using \eqref{Rimu2add} and \eqref{l2sumimu}, we obtain that 
 \begin{align*}
  \sum_{i\in\mathcal I_1}\sum_{\mu\in\mathcal I_2}N^{- 2}\phi_N^{s-4}   \E\left| \prod_{l=1}^k  \sqrt{N\eta}x_{s_l}(i,\mu)\right| &\prec \sum_{i\in\mathcal I_1}\sum_{\mu\in\mathcal I_2}(N\eta)^{k/2}N^{- 2}\phi_N^{s-4}  (R_i^2+R_\mu^2) \\
  &\prec (N\eta)^{k/2-1}\phi_N^{s-4} \le (N\eta \phi_N^4)^{s/4-1} \le N^{-c_\phi},
  \end{align*}
where we used the definition of $\phi_N$ in \eqref{phiN} and $s\ge 5$ in the last step.
 
    \vspace{5pt}
 
 \noindent{\bf Case 2:} There is only one $l$ such that $s_l =1$. Without loss of generality, we assume that $s_1=1$ and $s_l\ge 2$ for $2\le l\le k$. Thus, we have 
 $ s \ge 2k- 1$. Then, using \eqref{Rimu2add} and \eqref{l2sumimu}, we obtain that 
 \begin{align*}
 \sum_{i\in\mathcal I_1}\sum_{\mu\in\mathcal I_2}N^{- 2}\phi_N^{s-4}   \E\left| \prod_{l=1}^k  \sqrt{N\eta}x_{s_l}(i,\mu)\right| &\prec  \sum_{i\in\mathcal I_1}\sum_{\mu\in\mathcal I_2}(N\eta)^{k/2} N^{- 2}\phi_N^{s-4} \cdot \phi_N(R_i^2+R_\mu^2) \\
& \prec  (N\eta)^{k/2-1} \phi_N^{s-3} \le  (N\eta\phi_N^4)^{(s+1)/4-1} \le N^{-c_\phi}.
 \end{align*}
 
  \vspace{5pt}
 
 \noindent{\bf Case 3:} There are at least two $l$'s such that $s_l =1$. Without loss of generality, we assume that $s_1=s_2 = \cdots = s_j =1$ for some $2\le j\le k$. Thus, we have 
 $ s \ge 2(k-j) + j = 2k -j $. Then, using \eqref{Rimu2add} and \eqref{l2sumimu}, we can obtain that 
\begin{align*}
 \sum_{i\in\mathcal I_1}\sum_{\mu\in\mathcal I_2}N^{- 2}\phi_N^{s-4}   \E\left| \prod_{l=1}^k  \sqrt{N\eta}x_{s_l}(i,\mu)\right|  \prec  \sum_{i\in\mathcal I_1}\sum_{\mu\in\mathcal I_2}(N\eta)^{k/2} N^{- 2}\phi_N^{s-4}  \cdot \phi_N^{j-2} (R_i^2R_\mu^2 + \phi_N^2 (R_i^4 + R_\mu^4)) \\
 \prec  (N\eta)^{k/2} N^{- 2}\phi_N^{s+j-6}  \left(\frac1{\eta^2} + \frac{N}{\eta}\phi_N^2 \right) \lesssim  (N\eta)^{k/2-1}\phi_N^{s+ j-4 }  \le  (N\eta\phi_N^4)^{(s+j)/4-1} \le N^{-c_\phi}.
 \end{align*}

     \vspace{5pt}
     
Combining the above three cases, we conclude \eqref{eq_comp_est}. Then, \eqref{eq_comp_est234} and \eqref{eq_comp_est} together imply \eqref{compxxx}.
  \end{proof}
  
Combining \eqref{basic_interp} and \eqref{compxxx}, we conclude the proof of Proposition \ref{main_propadd}. 
  \end{proof}

  Finally, we complete the proof of the main theorems. 
   \begin{proof}[Proof of Theorems \ref{main_thm}, \ref{main_thm2}, \ref{main_thm3} and \ref{main_thm4}]
Combining Proposition \ref{main_propadd} with Lemma \ref{lem moment} for $\wt Y_{\eta,E}$, we get that \eqref{moments part} holds under the weaker moment assumption \eqref{size_condition2}:
\begin{align}\label{moments partweak}
&\mathbb E\left[ \prod_{s=1}^kY(\bu_{s},w_{s})\right]=\begin{cases} 
\sum \prod \eta \gamma(z_{s},z_{t},\bv_{s},\bv_{t})+ \OO_\prec\left( N^{-\e}\right), \ &\text{if $k\in 2\mathbb N$} \\ 
 \OO_\prec\left( N^{-\e}\right), \ &\text{otherwise}
\end{cases},
\end{align}
for some constant $\e>0$. By Wick's theorem, \eqref{moments partweak} shows that the convergence of $( \cal Y_{\eta,E}(\bv_1, w_1),\ldots, \cal Y_{\eta,E}(\bv_{k}, w_{k}) )$ in Theorems \ref{main_thm3} and \ref{main_thm4} holds in the sense of moments, which further implies the weak convergence. 
For the convergence of $ \eta^{-1/2}( \cal Y_{\eta,E}(\bv_1, w_1),\ldots, \cal Y_{\eta,E}(\bv_{k}, w_{k}) )$ in Theorem \ref{main_thm4} for $E\in S_{out}(\tau)$, we can prove a similar comparison estimates as in \eqref{compareout}:
\be\label{compareout2} 
\mathbb E\prod_{i=1}^r \eta^{-1/2}Y_{\eta,E}(\bu_i, w_i) = \mathbb E\prod_{i=1}^r \eta^{-1/2} \wt Y_{\eta,E}(\bu_i, w_i)  + \OO(n^{-\e}). 
\ee
Its proof is similar to that of \eqref{compareout}, so we omit the details. Then, \eqref{compareout2} and Lemma \ref{lem moment}  together imply the convergence of $ \eta^{-1/2}( \cal Y_{\eta,E}(\bv_1, w_1),\ldots, \cal Y_{\eta,E}(\bv_{k}, w_{k}) )$ for $E\in S_{out}(\tau)$.
  
Next, Theorems \ref{main_thm} and \ref{main_thm2} can be derived from \eqref{moments partweak} in the same way that Proposition \ref{main_prop2} is derived from Lemma \ref{lem moment}. As in Section \ref{sec generalf}, we apply the Helffer-Sj{\"o}strand formula to get a similar expression as \eqref{Zfk}. The only difference is about the local law for the $Y(z)$ terms: under the weaker moment assumption \eqref{size_condition2}, we only have the bound
\be\nonumber |Y(z)|\prec \sqrt{N\eta}\phi_N +1, \quad z\in \mathbf D.\ee
Let $\eta_1>0$ be such that $(N\eta_1)^{-1/2}=\phi_N$. Then, for $\im z\le \eta_1$, the local law \eqref{aniso_law0}
holds as before. For $\im z> \eta_1$, we do not have the high probability bound $Y(z)\prec 1$. However, by \eqref{moments partweak}, we still have $|\mathbb E\left[Y(z_1)\cdots Y(z_k)\right]|\prec 1$, such that the argument after \eqref{Zfk} still works and leads to Theorems \ref{main_thm} and \ref{main_thm2}.
 \end{proof}
%
%
%
%

\begin{acks}[Acknowledgments]
I would like to thank Jun Yin and Haokai Xi for helpful discussions. I also want to thank the editor, the associated editor and an anonymous referee for their helpful comments, which have resulted in a significant improvement. I am grateful to the support of Yau Mathematical Sciences Center, Tsinghua University, and Beijing Institute of Mathematical Sciences and Applications.
\end{acks}

\end{document}